\documentclass[11pt,makeidx]{amsart}
\usepackage[foot]{amsaddr}

\usepackage{graphicx}
\usepackage[T1]{fontenc}
\usepackage{textcomp}
\RequirePackage[numbers]{natbib}
\usepackage{mathscinet}
\usepackage{enumerate}
\usepackage{amsmath,amssymb,amsthm}  
\usepackage[pdftex,bookmarks,colorlinks,breaklinks]{hyperref}  
\usepackage{url}
\usepackage{multirow}
\usepackage{verbatim}
\usepackage{graphicx}

\hypersetup{linkcolor=blue,citecolor=blue,filecolor=cyan,urlcolor=blue} 

\newcommand{\dotp}[2]{\left\langle #1, #2\right\rangle}
\newcommand{\argmin}{\mathop{\rm argmin~}}
\newcommand{\Cov}{\mathop{\rm Cov}}

\newcommand{\fr}{\mathfrak}
\newcommand{\eps}{\varepsilon}
\def\Var{\mathop{\rm Var}}
\def\m{\mathcal}
\def\mb{\mathbb}
\newcommand{\dist}{\rm{dist}}
\newcommand{\sign}{\mathop{\rm sign}}

\newcommand{\card}{\mathop{\rm Card}}

\newcommand{\supp}{\mathop{\rm supp}}

\def\l{\left}
\def\r{\right}
\newcommand{\sinc}{\mathop{\rm sinc}}

\numberwithin{equation}{section}
\theoremstyle{plain}
\newtheorem{definition}{Definition}[section]
\newtheorem{theorem}{Theorem}[section]
\newtheorem{corollary}{Corollary}[section]
\newtheorem{proposition}{Proposition}[section]
\newtheorem{lemma}{Lemma}[section]
\newtheorem{assumption}{Assumption}[section]
\newtheorem*{acknowledgements}{Acknowledgements}
\newtheorem{remark}{Remark}[section]
\makeindex
\begin{document}

{\baselineskip=16.5pt{
\title[$L_1$-Penalization in Functional Linear Regression]{$L_1$-Penalization in Functional Linear Regression\\ 
with Subgaussian Design}
}}
\author{Vladimir Koltchinskii $^1$}
\address{$^1$ School of Mathematics, Georgia Institute of Technology}
\email{vlad@math.gatech.edu}
\thanks{V. Koltchinskii was partially supported by NSF grants DMS-1207808, DMS-0906880, CCF-0808863 and CCF-1415498.}
\author{Stanislav Minsker $^2$}
\address{$^2$ Department of Mathematics, Duke University}
\email{stas.minsker@gmail.com}
\thanks{S. Minsker was partially supported by grant R01-ES-017436 from the National Institute of Environmental Health Sciences (NIEHS) of the National Institutes of Health (NIH), NSF grants DMS-0650413, CCF-0808847 and DOE contract 113054 G002745.}

\begin{abstract}
We study functional regression with random subgaussian design and real-valued response. 
The focus is on the problems in which the regression 
function can be well approximated by a functional linear model with the 
slope function being ``sparse'' in the sense that it can be represented 
as a sum of a small number of well separated ``spikes''. This can be viewed 
as an extension of now classical sparse estimation problems to the case
of infinite dictionaries. We study an estimator of the regression function based on penalized empirical risk minimization with quadratic loss and the complexity penalty defined in terms of $L_1$-norm (a continuous version of LASSO). The main goal is to introduce several 
important parameters characterizing sparsity in this class of problems and to prove sharp oracle inequalities 
showing how the $L_2$-error of the continuous LASSO estimator depends on the underlying sparsity 
of the problem. 

\smallskip
\noindent \textbf{Keywords:} Functional regression \and Sparse recovery \and LASSO \and Oracle inequality \and Infinite dictionaries.

\smallskip
\noindent \textbf{Mathematics Subject Classification (2000):} 62J02, 62G05, 62J07.
\end{abstract}
\maketitle
\tableofcontents

\section{Introduction}

Let $(X,Y)$ be a random couple defined on a probability space $(\Omega,\Sigma,{\mb P}),$
where $X=\{X(t):t\in {\mb T}\}$ is a stochastic process with parameter set $\mb T$ and 
$Y$ is a real valued response variable. In what follows, it will be assumed that the process $X$ 
is {\it subgaussian}. Denote 
\index{aa@Pseudometric $d_X(\cdot,\cdot)$}
\begin{align}
\label{metric}
&
d_X(s,t):=\sqrt{\Var\l(X(s)-X(t)\r)}, s,t\in {\mb T}.
\end{align}
It will be also assumed that the space $\mb T$ is totally bounded with respect 
to pseudometric $d_X$ and, moreover, it satisfies Talagrand's generic chaining conditions  
ensuring that there exists a version 
of the process $X(t), t\in \mb T$ that is a.s. uniformly bounded and $d_X$-uniformly 
continuous. In what follows, we assume that $X(t), t\in \mb T$ is such a version.   
Let $\mu$ be a finite measure on the Borel $\sigma$-algebra ${\m B}_{\mb T}$ of the pseudometric space 
$(\mb T,d_X).$ 

Consider the following regression model 
$$
Y=f_{\ast}(X)+\xi,
$$
where $f_{\ast}(X)= {\mb E}(Y|X)$ is the regression function and $\xi$ is a random 
noise with ${\mb E}\xi=0$ and variance $\Var(\xi)=\sigma_\xi^2$ independent of the design variable $X.$
We will be interested in estimating the regression function $f_{\ast}(X)$ 
under an underlying assumption that $f_{\ast}(X)$ can be well approximated 
by a functional linear model (``oracle model'')
\index{ab@$f_{\lambda,a}(\cdot)$}
$$
f_{\lambda,a}(X)= a + \int\limits_{\mb T} X(t)\lambda(t)\mu (dt),
$$ 
where $\lambda\in L_1(\mu)$ is the ``slope'' function and $a\in {\mb R}$
is the intercept of the model. More precisely, we will focus on the 
problems in which the oracle models are ``sparse'' in the sense 
that the slope function $\lambda$ is supported in a relatively small subset ${\rm supp}(\lambda):=\{t\in \mb T: \lambda(t)\neq 0\}$ 
of parameter space $\mb T$ such that the set of random variables $\{X(t):t\in {\rm supp}(\lambda)\}$ can be well approximated by a linear space of a small dimension.
Often, $\lambda$ will be a sum of several ``spikes'' with disjoint and well separated 
supports. Such models might be useful in a variety of applications, in particular,
in image processing where, in many cases, only sparsely located regions of the image 
are correlated with the response variable.  
\index{ac@Distributions $P,\Pi$}
In what follows, $\Pi$ denotes the marginal distribution of $X$ in the space 
$C_{bu}({\mb T};d_X)$ of all uniformly bounded and uniformly continuous functions 
on $({\mb T};d_X)$, and $P$ denotes the joint distribution of $(X,Y)$
in $C_{bu}({\mb T};d_X)\times {\mb R}.$ Let 
$
(X_1,Y_1),\ldots,\ (X_n,Y_n)
$
be a sample consisting of $n$ i.i.d. copies of $(X,Y)$ defined on $(\Omega,\Sigma,{\mb P}).$ 
The regression function $f_{\ast}$ is to be estimated based 
on the data 
$
(X_1,Y_1),\ldots,\ (X_n,Y_n).
$
Our estimation method can be seen as a direct extension of (a version of) LASSO to the infinite-dimensional case. Namely, let  
$\mb D$ be a convex subset of the space $L_1(\mu)$ such that $0\in\mb D.$  
Consider the following penalized empirical risk minimization problem:
\begin{align}
(\hat\lambda_{\eps},\hat a_\eps):=
\argmin\limits_{\lambda\in\mb D,a\in \mb R}
\biggl[\frac 1 n \sum_{j=1}^n \l(Y_j-f_{\lambda,a}(X_j)\r)^2+\varepsilon \|\lambda\|_1\biggr],
\label{empirical}
\end{align}
where $\|\lambda\|_1:=\|\lambda\|_{L_1(\mu)}=\int\limits_{\mb T}|\lambda(t)|\mu(dt)$
and $\eps>0$ is the regularization parameter. The function $f_{\hat \lambda_{\eps}, \hat a_{\eps}}$ will be used as an estimator of the regression function $f_{\ast}.$

When the parameter set $\mb T$ is finite, (\ref{empirical}) defines a standard LASSO-estimator
of the vector of parameters of linear regression model (see \cite{tibshirani1996regression}). 
This estimator is among the most popular in high-dimensional statistics and it has been 
intensively studied in the recent years (e.g., see \cite{bunea2007sparsity}, \cite{van2008high}, \cite{Koltchinskii2007Sparsity-in-Pen00}, 
\cite{bickel2009simultaneous}, \cite{Koltchinskii2011Oracle-inequali00}, \cite{bartlett2009},
\cite{koltchinskii2011nuclear}; see also the book by B\"uhlmann and van de Geer \cite{Buhlmann2011Geer} for further references). 

We will be more interested in the case
of uncountable infinite parameter sets $\mb T$ (functional linear models). In such problems,
standard characteristics of finite dictionaries used in the theory of sparse recovery 
(restricted isometry constants, restricted eigenvalues, etc) are not directly applicable. 
Our goal will be to develop proper parameters characterizing sparsity in the case of 
functional models and to prove oracle inequalities for the $L_2(\Pi)$-error $\|f_{\hat \lambda_{\eps}, \hat a_{\eps}}-f_{\ast}\|_{L_2(\Pi)}^2$ of continuous LASSO-estimator 
in terms of these sparsity parameters. We concentrate on the case of subgaussian random 
design (that, of course, includes an important example of Gaussian design processes)
since, in this case, we can rely on a number of probabilistic tools from the theory 
of subgaussian and empirical processes. 
In particular, we extensively use in the proofs recent generic chaining bounds for empirical processes due to Mendelson \cite{mendelson2010empirical}, \cite{mendelson2012oracle}. 

It should be emphasized that there is vast literature on functional regression (see, e.g., \cite{ramsay2006functional}, \cite{ramsay2002applied} and references therein). 
A commonly used general idea in this literature is to estimate the eigenfunctions of the covariance operator and to project the unknown slope function onto the linear span of the ``principal components'' corresponding to the largest eigenvalues
(see \cite{muller2005generalized}, \cite{cai2006prediction} and references therein). 
Under smoothness assumptions on the slope function, a natural approach to its estimation is to use a regularization penalty (see \cite{crambes2009smoothing} for construction of estimators based on smoothing splines and \cite{yuan2010reproducing} for a more general reproducing kernel Hilbert space approach). 

The problem studied in our paper is much closer to the theory of sparse estimation in high-dimensional statistics and can be viewed as an extension of this 
theory to the case of functional models and uncountable dictionaries. 
Our approach is similar in spirit to \cite{Koltchinskii2009Sparse-recovery00}, \cite{Koltchinskii2011Oracle-inequali00} 
where such characteristics as ``alignment coefficient'' (used below for functional models) were introduced and studied in the case of finite dictionaries, and \cite{koltchinskii2010sparse} which extended some of these results to the case of infinite dictionaries. 
For a review of some modern methods in functional data processing and their connections to various notions of sparsity, we refer the reader to \cite{james2011sparseness}. 
A recent paper by James, Wang and Zhu \cite{james2009functional} is similar to the present work in terms of motivation and approach, however, the theoretical analysis in \cite{james2009functional} is performed under the assumptions on the design distribution that might not hold if $X$ has smooth trajectories.

It is important to note that in practice we never observe the whole trajectory of $X$ but rather its densely sampled version. 
In this case, the natural choice for $\mu$ is a uniform measure on the sampling grid, whence (\ref{empirical}) becomes the usual LASSO once again. 
However, there is often no reasons to assume that Gram matrix of the design satisfies RIP \cite{candes2006stable} or restricted eigenvalue type conditions \cite{bickel2009simultaneous, koltchinskii2009dantzig} in this case. 
Although LASSO might not perform well as a variable selection procedure in such a framework, we will provide examples showing that prediction power of an estimator can still benefit from the fact that the underlying model is (approximately) sparse. 
In particular, oracle inequalities with error rates depending on sparsity can be 
derived from the general results of our paper.
Other interesting approaches to theoretical analysis of LASSO with highly correlated design were proposed in \cite{van2012lasso}, \cite{Hebiri2013How-Correlation00}. For instance, in  \cite{van2012lasso} (see, in particular, Corollary 4.2) the authors show that in the case of highly correlated design, it is often possible to choose the regularization parameter to be small $\eps \ll n^{-1/2}$ and achieve reasonable error rates. 

It should be also mentioned that in a number of very important applications one has to deal 
with sparse recovery in infinite dictionaries with random designs that are not subgaussian,
or with deterministic designs. For instance, in \cite{candes2012fernandez}, the authors develop 
a theory of super-resolution. In this case, the dictionary consists of complex exponentials 
$e^{i\langle t, \cdot\rangle}, t\in {\mb T}\subset {\mb R}^d,$ the design is deterministic and the estimation method is based on minimizing the total variation norm of a signed measure $\Lambda$ on $\mb T$ subject to data dependent constraints. Although the results of our 
paper do not apply immediately to such problems, it is possible to extend our approach in this direction.

We will introduce several assumptions and definitions used throughout the paper.

\begin{definition}
A closed linear subspace ${\mathcal L}\subset L_2(\mb P)$ will be called a \textit{subgaussian space} if there exists a constant $\Gamma>0$ such that for all $\eta\in \m L$
$$
{\mathbb E}e^{s\eta}\leq e^{\Gamma s^2 \sigma_{\eta}^2}, s\in {\mathbb R},
$$
where $\sigma_{\eta}^2:=\Var(\eta).$
\end{definition}

It is well known that ${\mathbb E}\eta=0, \eta\in {\mathcal L}$ and that 
$\psi_2$-and $L_2$-norms are equivalent on ${\mathcal L}$ (more precisely, they 
are within a constant $\sim \Gamma$). 
Also, if ${\mathcal L}$ is a closed linear subspace
of $L_2(\mb P)$ such that $\{\eta: \eta\in {\mathcal L}\}$ are jointly 
normal centered random variables, then ${\mathcal L}$ is a subgaussian 
space with $\Gamma=1.$ Another example is the closed linear span of independent centered subgaussian random variables $\{\eta_j\}$ such that 
$$
{\mathbb E}e^{s\eta_j}\leq e^{\Gamma\sigma_{\eta_j}^2 s^2}, s\in {\mathbb R}, j\geq 1
$$
for some $\Gamma>0:$ 
$$
{\mathcal L}:=\Big\{\sum_{j\geq 1} c_j \eta_j: \sum_{j\geq 1}\sigma_{\eta_j}^2 c_j^2 <+\infty\Big\}.
$$ 
For instance, one can consider a sequence $\{\eta_j\}$ of i.i.d.
Rademacher random variables (that is, $\eta_j$ takes valued $+1$ and $-1$ with probability 
$1/2$). In the case of a single random variable $\eta,$ its linear span is a subgaussian 
space if and only if $\eta$ is subgaussian. 

\index{ad@Subgaussian space $\mathcal L, \ \mathcal L_X$}
In what follows, a subgaussian space $\mathcal L$ and constant $\Gamma$ will be fixed.
All the constants depending only on $\Gamma$ will be called {\it absolute}.

\begin{assumption}
Suppose that
\begin{align}
\label{subgaussian}
X(t)-{\mathbb E}X(t)\in \m L \text{ for all }  t\in \mb T.
\end{align}
Denote by ${\mathcal L}_{X}$ the closed (in $L_2$ and, as a consequence, also in the $\psi_2$-norm) linear span of $\{X(t)-{\mb E} X(t): t\in \mb T\}.$
\end{assumption}

This assumption easily implies that the stochastic process $Z(t):=X(t)-{\mb E}X(t), \ t\in \mb T$ is subgaussian, meaning the for all $t,s\in {\mathbb T},$ 
$Z(t)-Z(s)$ is a subgaussian random 
variable with parameter $\Gamma d_X^2(t,s).$

Next, we recall the notion of Talagrand's generic chaining complexity (see \cite{talagrand2005generic} for a comprehensive introduction). 
Given a pseudo-metric space $(\mb T,d_X)$, 
let $\left\{\Delta_n\right\}$ be a nested sequence of partitions such that 
${\rm card}\,\Delta_0=1$ and ${\rm card}\,\Delta_n\leq 2^{2^n}$. 
For $s\in \mb T$, let $\Delta_n(s)$ be the unique subset of $\Delta_n$ containing $s$. 
The generic chaining complexity $\gamma_2(\mb T;d_X)$ is defined as 
\index{ae@Generic chaining complexity $\gamma_2(\mb T;d_X), \ \gamma_2(\delta)$}
$$
\gamma_2(\mb T;d_X):=\inf_{\left\{\Delta_n\right\}}\sup_{s\in \mb T}\sum_{n\geq 0} 2^{\frac n2}D(\Delta_n(s))
$$
where $D(A)$ stands for the diameter of a set $A$.
Let 
$$
\gamma_2(\delta):=\gamma_2(\mb T; d_X;\delta)=\inf_{\{\Delta_n\}}\sup_{t\in {\mathbb T}}
\sum_{n\geq 0} 2^{n/2}\l(D(\Delta_n(t))\wedge \delta\r).
$$
If $d_Y$ is another metric on $\mb T$ such that $d_Y(t,s)\leq d_X(t,s)$ for all $t,s\in \mb T$, and $\sup\limits_{t,s\in \mb T} d_Y(t,s)\leq \delta$, then clearly 
\begin{align}
\label{eq:chaining_bound}
\gamma_2(\mb T;d_Y)\leq \gamma_2(\delta).
\end{align}
This bound will be often used below. 
Our main complexity assumptions on the design distribution are the following:
\begin{assumption}
\label{assumption1}
Pseudometric space $(\mb T,d_X)$ is such that
$
\gamma_2(\mb T; d_X)<\infty
$
and, moreover,
$$
\gamma_2(\mb T;d_X;\delta)\to 0\ {\rm as}\ \delta\to 0.
$$
\end{assumption}

Under these assumptions, the process $Z=X-\mb E X$ has a version that is uniformly bounded 
and $d_X$-uniformly continuous a.s. Moreover, $\big\|\|X-\mb E X\|_{\infty}\big\|_{\psi_2}<\infty$ (in particular, all the moments of $\|X-\mb EX\|_{\infty}$ are finite). 
It what follows, we will denote 	
\index{af@$S(\mb T)$}
$$
S(\mb T):=S(\mb T,d_X)=\inf\limits_{t\in \mb T}\sqrt{\Var(X(t))}+L\gamma_2(\mb T;d_X).
$$
Note that Theorem \ref{tal1} implies that there exists a numerical constant $L>0$
such that 
\begin{align}
\label{eq:chaining1}
\mb E\sup\limits_{t\in \mb T}|X(t)-\mb EX(t)|\leq S(\mb T). 
\end{align}

We will also need the following assumptions on the regression function $f_{\ast}$
and the noise $\xi:$ 

\begin{assumption} 
Suppose that $f_{\ast}(X)-{\mb E}f_{\ast}(X)\in {\mathcal L}$
and $\xi \in {\mathcal L}.$
\end{assumption} 

Since ${\mb E}f_{\ast}(X)={\mb E}Y,$ this assumption also implies that $Y-{\mb E}Y\in {\mathcal L}.$ 
Note that if $\{X(t), \ t\in {\mb T}\}\cup \{Y\}$ is 
a family of centered Gaussian random variables and ${\mathcal L}$ is its closed 
linear span, then ${\mathcal L}$ is a subgaussian space and $f_{\ast}(X)$ is the 
orthogonal projection of $Y$ onto the subspace ${\mathcal L}_X.$ Thus, $f_{\ast}(X)\in {\mathcal L}_X\subset {\mathcal L}.$

\section{Approximation error bounds, alignment coefficient and Sobolev norms}
\label{approx_error}

Recall that $P$ is the joint distribution of $(X,Y)$ and let $P_n$ be the empirical distribution 
based on the sample $(X_1,Y_1),\dots, (X_n,Y_n).$ 
The integrals with respect to $P$ and $P_n$ are denoted by
\index{afa@Empirical distribution $P_n$}
$$
Pg:=\mathbb{E}g(X,Y), \quad P_n g:=\frac{1}{n}\sum\limits_{i=1}^n g(X_i,Y_i).
$$
In what follows, it will be convenient to denote 
\index{afb@Loss function $\ell$}
$
\ell (y,u):=(y-u)^2,\ y,u\in {\mb R}
$
and 
$$
(\ell \bullet f)(x,y):= \ell (y,f(x))=(y-f(x))^2.
$$
We also use the notation 
$\ell'(y,u)$ for the derivative of quadratic loss $\ell(y,u)$ with respect to $u:$
$
\ell'(y,u)=2(u-y).
$
Throughout the paper, $\dotp{\cdot}{\cdot}$ denotes the bilinear form 
$$
\dotp{f}{g}:=\int\limits_{\mb T}f(t)g(t)\mu(dt).
$$ 
Let
\index{ag@$F(\lambda,a), \ F_n(\lambda,a)$}
$$
F_n(\lambda,a):=P_n(\ell \bullet f_{\lambda,a})+\varepsilon\|\lambda\|_1,\ \ 
F(\lambda,a):=P(\ell \bullet f_{\lambda,a})+\varepsilon\|\lambda\|_1.
$$
Denote also 
\index{aga@$\bar X_n, \ \bar Y_n$}
$$
\bar Y_n:= n^{-1}\sum_{j=1}^n Y_j,\ \ \bar X_n(t):=n^{-1}\sum_{j=1}^n X_j(t), t\in \mb T. 
$$
Note that 
\index{ah@$a(\lambda),\ \hat a(\lambda)$}
\begin{align}
\label{g1}
&
\hat a(\lambda):=\argmin\limits_{a\in \mb R}F_n(\lambda,a)=\bar Y_n-\langle \lambda, \bar X_n\rangle, \\
& \nonumber
a(\lambda):=\argmin\limits_{a\in \mb R}F(\lambda,a)=\mb EY - \langle \lambda, \mb E X\rangle.
\end{align}

The following penalized empirical risk minimization problem
\begin{align}
(\hat \lambda_{\eps},\hat a_\eps)
:=\argmin\limits_{\lambda\in \mb D,a\in \mb R}F_n(\lambda;a),
\label{empirical1}
\end{align}
is exactly problem (\ref{empirical}) written in a more concise form. 
Note that (\ref{empirical1}) is the empirical version of
\begin{align}
(\lambda_{\eps},a_\eps):=
\argmin\limits_{\lambda\in\mb D,a\in \mb R}F(\lambda,a).
\label{true}
\end{align}
Due to convexity of the loss, both (\ref{true}) and (\ref{empirical1}) are convex optimization problems. 
It will be shown (Theorem \ref{existence} in the appendix) that, under certain assumptions, they admit (not necessarily unique) solutions $\lambda_\eps,  \ \hat\lambda_\eps.$ 

\begin{assumption}
It is assumed throughout the paper that solutions $(\lambda_{\eps},a_{\eps})$ of (\ref{true}) and 
$(\hat \lambda_{\eps},\hat a_{\eps})$ of (\ref{empirical1}) exist.
\end{assumption}

It might be also possible to study the problem under an assumption 
that $(\lambda_{\eps},a_{\eps})$ and $(\hat \lambda_{\eps},\hat a_{\eps})$
are approximate solutions of the corresponding optimization problems,
but we are not pursuing this to avoid further technicalities.

The goal of this section is to determine the parameters responsible for the size of the $L_2(\Pi)$ risk of $f_{\lambda_\eps,a_\eps}$, where $(\lambda_\eps,a_\eps)$ is the (distribution-dependent) solution of the problem (\ref{true}), and to find upper bounds on these 
parameters in terms of classical Sobolev type norms. 
Later on, it will be shown that the same parameters affect the error rate of empirical solution $f_{\hat\lambda_\eps,\hat a_\eps}$.

Recall that ${\mathbb D}\subset L_1(\mu)$ is a convex subset that
contains zero. It immediately follows from (\ref{true}) that we can take 
$a_{\eps}=a(\lambda_{\eps})$ and also that 
\index{ai@$q(\eps)$}
\begin{equation}
\label{def:q_eps}
\|f_{\lambda_{\eps},a_{\eps}}-f_{\ast}\|_{L_2(\Pi)}^2
\leq q(\eps):= 
\inf_{\lambda\in \mb D, a\in \mb R}\l[\|f_{\lambda,a}-f_{*}\|^2_{L_2(\Pi)}+\eps\|\lambda\|_1\r].
\end{equation}
Clearly, $q$ is a non-decreasing concave function (concavity follows from the fact that it is an infimum of linear functions). Therefore, $\frac{q(\eps)}{\eps}$ is a non-increasing 
function. Note also that $q(\eps)\leq \|f_{\ast}-\Pi f_{\ast}\|_{L_2(\Pi)}^2$ (take $\lambda=0,a={\mb E}Y={\mb E}f_{\ast}(X)$ in the expression under the
infimum) and 
$$
q(\eps)\leq \eps \|\lambda_{\ast}\|_1
$$
provided that $f_{\ast}=f_{\lambda_{\ast},a_{\ast}},$ where 
$\lambda_{\ast}\in {\mathbb D}, a_{\ast}\in \mb R$ (take $\lambda=\lambda_{\ast}, a=a_{\ast}$). The infimum in the definition of $q(\eps)$ is attained at $(\lambda_{\eps},a_{\eps})$ (a solution of problem (\ref{true}) that is assumed to exist). 
Then, in addition to the bound  
$
\|f_{\lambda_{\eps},a_{\eps}}-f_{\ast}\|_{L_2(\Pi)}^2 \leq q(\eps),
$
(\ref{true}) also implies
$$
\|\lambda_{\eps}\|_1 \leq \frac{q(\eps)}{\eps}.
$$

We will be interested, however, in other bounds on  $\|f_{\lambda_{\eps},a_{\eps}}-f_{\ast}\|_{L_2(\Pi)}^2,$
in which the ``regularization error'' is proportional to $\eps^2$ 
rather than to $\eps$ (as it is the case in the bounds for $q(\eps)$). 
To this end, we have to introduce some new characteristics of the oracles $\lambda \in {\mathbb D}.$

\index{aj@Covariance function $k(s,t)$}
\index{ak@Covariance operator $K$}
Let $k(s,t):=\Cov(X(s),X(t)), s,t\in {\mb T}$ be the covariance function of the stochastic process $X.$ 
Clearly, under Assumption \ref{assumption1}, $\iint k^2(s,t)\mu(ds)\mu(dt)<\infty$ and the {\it covariance operator} $K:L_2(\mb T,\mu)\mapsto L_2(\mb T,\mu)$ defined by
$$
(Kv)(s):=\int\limits_\mb T k(s,t)v(t)\mu(dt).
$$
is Hilbert--Schmidt. 
For $u\in L_2(\mb T)$, define
\index{al@Norm $\|\cdot\|_K$}
\begin{align}
\label{H_p}
&
\l\|u\r\|_{K}:=\sup_{\dotp{Kv}{v}\leq 1}\dotp{u}{v}.
\end{align}

\begin{remark}
In the case when ${\mb T}$ is finite, operator $K$ is represented by the Gram matrix of a finite dictionary and standard ``restricted isometry'' and ``restricted eigenvalue'' type 
constants are defined in terms of $K$ and are involved in oracle inequalities for LASSO 
and other related estimators.
\end{remark}
Note that 
$\dotp
{Kv}{v}=\Var(f_v(X))$, where $f_v(X):=\int\limits_\mb T v(t)X(t)\mu(dt)$.
The set
$$
\mb H(K):=\left\{u\in L_2(\mb T): \ \l\|u\r\|_K<\infty\right\}
$$ 
is a reproducing kernel Hilbert space of the covariance kernel $k.$

We will need the following description of the subdifferential of the convex function $\|\cdot\|_1$:
\begin{align}
\label{sub}
&
\partial \|\lambda\|_1=\left\{w:\mb T\mapsto [-1,1]:\ \mu-{\rm a.s.}\ w(t)=\sign(\lambda(t)) \text{ whenever } \lambda(t)\ne 0\right\}.
\end{align}
It follows from the general description of the subdifferential of a norm $\|\cdot\|$ in a Banach space $\mathfrak X$:
\index{ala@$\partial \|\lambda\|_1$}
$$
\partial\|x\|=\begin{cases}
& \l\{x^*\in \mathfrak X^*: \ \|x^*\|= 1,\ x^*(x)=\|x\|\r\}, \quad x\ne 0,\\
& \l\{x^*\in \mathfrak X^*: \ \|x^*\|\leq 1\r\},  \quad x=0,
\end{cases}
$$
where $\mathfrak X^*$ is the dual space. 
For details on our specific example, see \cite{ioffe1974theory}, paragraph 4.5.1.\par

Note that, in standard examples (such as $\mb T\subset {\mb R}^d$), the ``canonical'' version of 
subgradient of $\|\lambda\|_1,$ $w(t)={\rm sign}(\lambda(t)),t\in \mb T,$ lacks smoothness 
and RKHS-norms are often large or infinite for such a choice of $w.$ It will be seen below 
that existence of smoother versions of subgradient is important in such cases. 
Given a measurable $w:\mb T\mapsto [-1,1]$, let $\mb T_w=\left\{t\in \mb T: \ |w(t)|\geq \frac12\right\}$. For smooth $w,$ $\mb T_w$ will play a role of support of $\lambda.$ 
Given $b\in [0,\infty],$ define the cone $C_w^{(b)}$ by 
\index{am@Cone $C_w^{(b)}$}
\begin{align}
\label{eq:cone}
C_w^{(b)}:=\Bigg\{u\in L_1(\mu): \ \int\limits_{\mb T\setminus \mb T_w}|u|d\mu\leq b\dotp{w}{u}\Bigg\}.
\end{align}
Note that, for $w\in \partial \|\lambda\|_1,$ we have $|w(t)|\leq 1, t\in {\mathbb T}.$ Therefore, 
$u\in C_w^{(b)}$ implies that 
$$
\int\limits_{\mb T\setminus \mb T_w}|u|d\mu\leq b\int\limits_{\mb T_w}|u|d\mu.
$$  
Roughly, this means that, for functions $u\in C_w^{(b)},$ ${\mathbb T}_w$ is a ``dominant set''.
Let 
\index{an@Alignment coefficient $\fr a^{(b)}(w)$}
\begin{align}
\label{def:alignment}
\fr a^{(b)}(w):=\sup\l\{\dotp{w}{u}: \ u\in C_w^{(b)}, \ \|f_u\|_{L_2(\Pi)}=1\r\}.
\end{align} 
Such quantities were introduced in the framework of sparse recovery in \cite{Koltchinskii2009Sparse-recovery00}, \cite{Koltchinskii2011Oracle-inequali00}
and its size is closely related to the RIP and restricted eigenvalue - type conditions. 
In some sense, $\fr a^{(b)}(w)$ characterizes the way in which vector (function) $w$ is ``aligned'' with eigenspaces of the covariance operator of the process $X$ and, following \cite{Koltchinskii2009Sparse-recovery00}, it will be called the {\it alignment coefficient}.
Clearly, we always have the bound $\fr a^{(b)}(w)\leq \|w\|_K$, however, it can be improved in several important cases, see Section \ref{sec:multiple}. Note that $\fr a^{(b)}(w)$ is a nondecreasing 
function of $b.$ For $b=\infty,$ we have $C_w^{(\infty)}=L_1(\mu).$ In this case,  
$\fr a^{(\infty)}(w)=\|w\|_K,$ so, the alignment coefficient coincides with the RKHS-norm associated to the covariance function $k.$ 
For $b=0,$ we have $C_w^{(0)}=\{u\in L_1(\mu): u=0\ {\rm a.s.\ on}\ \mb T\setminus \mb T_w\},$ so, the cone $C_w^{(0)}$ coincides with the subspace of functions supported 
in ${\mathbb T}_w.$ In this case, $\fr a^{(0)}(w)$ is the RKHS-norm associated with restriction of the kernel $k$ to ${\mathbb T}_{w}.$ 

In what follows, it will be convenient to take $b=16$ and denote $\fr a(w)=\fr a^{(16)}(w)$
(although in the statement of Theorem \ref{approx} below a smaller value $b=2$ could be used). 

We will be interested in those oracles $\lambda$ for which there exists a subgradient $w\in \partial \|\lambda\|_1$ 
such that $\fr a(w)$ is not too large and 
$\mb T_w$ is a ``small'' subset of $\mb T$. 
Such functions provide a natural analogue of sparse vectors in finite-dimensional problems. 

\begin{theorem}
\label{approx}
The following inequality holds:
\begin{align}
\label{eq:approx}
\l\|f_{\lambda_\eps,a_\eps}-f_{\ast}\r\|_{L_2(\Pi)}^2
\leq \inf\limits_{\lambda\in \mb D, w\in \partial\|\lambda\|_1,a\in \mb R}
\l[\l\|f_{\lambda, a}-f_{\ast}\r\|_{L_2(\Pi)}^2+\frac 1 4\eps^2 \fr a^2(w)\r].
\end{align}
\end{theorem}

\begin{remark}
It easily follows from the proof of this theorem that for all $\lambda \in {\mathbb D}, w\in \partial\|\lambda\|_1, 
a\in {\mathbb R},$
$$
\int\limits_{\mb T\setminus \mb T_{w}}|\lambda_{\eps}|d\mu \leq \frac{4}{\eps}
\l[\l\|f_{\lambda, a}-f_{\ast}\r\|_{L_2(\Pi)}^2+\frac 1 4\eps^2 \fr a^2(w)\r].
$$
\end{remark}

The intuition behind these results is the following: 
if there exists an oracle $(\lambda,w,a)$ with a small approximation 
error $\l\|f_{\lambda, a}-f_{\ast}\r\|_{L_2(\Pi)}^2$ (say, of the order $o(\eps)$) and not very large alignment coefficient 
$\fr a(w),$ then the risk $\l\|f_{\lambda_\eps, a_\eps}-f_{\ast}\r\|_{L_2(\Pi)}^2$ is also small and 
$\lambda_{\eps}$ is ``almost'' concentrated on the set $\mb T_w.$

As we show below, in some cases $\|\cdot\|_K$ and $\fr a(\cdot)$ can be bounded in terms of Sobolev-type norms. 

Since self-adjoint integral operator $K$ with kernel $k$ is Hilbert--Schmidt,
it is compact, and the orthogonal complement to its kernel possesses an orthonormal system of eigenfunctions 
$\l\{f_j\r\}_{j=1}^\infty \subset L_2(\mb T,\mu)$ corresponding to positive eigenvalues $\nu_j$. 
It is well-known that 
\begin{align}
\label{ker}
\mb H(K)=\l\{w(\cdot)=\sum\limits_{j=1}^\infty w_j f_j(\cdot): \ \|w\|^2_{K}=\sum\limits_{j=1}^\infty \frac{w_j^2}{\nu_j}<\infty\r\}.
\end{align}
However, one might want to find a more direct characterization of $\mb H(K)$. 
One way to proceed is to use the so-called {\it Factorization theorem}:
\begin{theorem}[\cite{lifshits1995gaussian}, Theorem 4 in Section 9]
Assume that there exists a Hilbert space $\mb V$ and an injective linear operator $L:\mb V\mapsto \ell_\infty(\mb T)$ such that $K=LL^*$, where $L^*$ is the adjoint of $L$. 
Then $\mb H(K)=L(\mb V)$, and $\dotp{Lu_1}{Lu_2}_{\mb H(K)}=\dotp{u_1}{u_2}_\mb V$. 
\end{theorem}
The most obvious choice is $\mb V={\rm ker}(K)^{\perp}$ and $L=K^{1/2}$, whence $\|w\|_K=\|K^{-1/2}w\|_{L_2(\mu)}$ which again gives (\ref{ker}). 
Other choices often lead to more insightful description. 
For example, if $X$ is the standard Brownian motion on $[0,1]$, then one can check \cite{lifshits1995gaussian} that $\mb V=L_2[0,1]$ with the standard Lebesgue measure and 
$(L x)(t):=\int\limits_0^t x(s)ds$ satisfy the requirements. 
It immediately implies
\begin{corollary}
\label{brownian1}
The reproducing kernel Hilbert space associated with the Brownian motion is defined by
\begin{align}
\label{brownian}
&
\mb H(K)=\l\{h\in L_2[0,1], \ h(0)=0, \ \|h\|_K^2:=\int_0^1 \l(h'(s)\r)^2 ds<\infty \r\}\subset \mb W^{2,1}[0,1],
\end{align}
where 
\begin{align*}
\mb W^{2,1}[0,1]=\Big\{h&\in L_2[0,1], \ h \text{ is abs. continuous,} \\
&
\|h\|_{{\mb W}^{2,1}}^2:=\int_0^1 \l[h^2(s)+\l(h'(s)\r)^2 \r]ds<\infty\Big\}
\end{align*}
is the Sobolev space.
\end{corollary}

In particular, it means that $\fr a(w)\leq \|w\|_{{\mb W}^{2,1}}.$
Suppose now that ${\mb T}\subset {\mb R}^m$ is a bounded open subset and, for some $C>0$ and $\beta>0,$
\begin{equation}
\label{sob_bound}
{\fr a}^2(w)\leq C \|w\|_{{\mb W}^{2,\beta}}^2.
\end{equation}
Let $\lambda\in L_1({\mb T},\mu)$ be a ``sparse'' oracle such that 
$
{\rm supp}(\lambda):= \bigcup\limits_{j=1}^d \mb T_j, 
$
where $\mb T_j, j=1,\dots, d$ are disjoint sets. 
Moreover, assume that the distance 
between $\mb T_j$ and $\mb T_k$ is positive for all $j\neq k.$ In other words, $\lambda$
has $d$ components with well separated supports and it is zero in between. 
In this case, one can find 
$w\in \partial \|\lambda\|_1$ such that $w=\sum\limits_{j=1}^d w_j$ and $w_j, j=1,\dots, d$ are 
smooth functions (from the space ${\mb W}^{2,\beta}$ to be specific) with disjoint supports.
For any such function $w,$ we have 
$$
{\fr a}^2(w)\leq C \|w\|_{{\mb W}^{2,\beta}}^2 \leq C_1 \sum_{j=1}^d\|w_j\|_{{\mb W}^{2,\beta}}^2
\leq C_1 d\max_{1\leq j\leq d}\|w_j\|_{{\mb W}^{2,\beta}}^2 
$$
and the bound of Theorem \ref{approx} implies that  
\begin{align}
\label{eq:approx_A}
\l\|f_{\lambda_\eps,a_\eps}-f_{\ast}\r\|_{L_2(\Pi)}^2
\leq 
\l\|f_{\lambda, a(\lambda)}-f_{\ast}\r\|_{L_2(\Pi)}^2+
\frac C 4 d\max_{1\leq j\leq d}\|w_j\|_{{\mb W}^{2,\beta}}^2 \eps^2. 
\end{align}
Thus, the size of the error explicitly depends on the number 
$d$ of components of ``sparse'' oracles $\lambda$ approximating
the target. 

In Section \ref{examples}, we will show that bound (\ref{sob_bound}) holds for a number of 
stochastic processes $X$ and, moreover, there are other ways to take advantage of 
sparsity in the cases when the domain $\mb T$ of $X$ can be partitioned in a number 
of regions $\mb T_j, j=1,\dots, N$ such that the processes $\l\{X(t), t\in \mb T_j\r\}$
are ``weakly correlated''.  


\section{Basic oracle inequalities}

In this section, we present general oracle inequalities for the $L_2$-risk of estimator  $f_{\hat\lambda_\eps, \hat a_{\eps}}.$ The main goal is to show that if there exists an oracle $(\lambda,w,a), \lambda \in {\mb D}, w\in \partial \|\lambda\|_1, a\in {\mb R}$ such that 
the approximation error $\|f_{\lambda,a}-f_{\ast}\|_{L_2(\Pi)}^2$ is small, the alignment 
coefficient ${\fr a}(w)$ is not large and $\lambda$ is ``sparse'' in the sense that 
the set of random variables $\{X(t):t\in {\mb T}_w\}$ can be well approximated by 
a linear space $L\subset {\mathcal L}_X$ of small dimension, then the $L_2$-error
$\|f_{\hat \lambda_{\eps},\hat a_{\eps}}-f_{\ast}\|_{L_2(\Pi)}^2$ of the estimator 
$f_{\hat \lambda_{\eps},\hat a_{\eps}}$ can be controlled in terms of the dimension 
of $L$ and the alignment coefficient ${\fr a}(w).$ To state the result precisely,
we have to introduce one more parameter, an ``approximate dimension'', providing 
an optimal choice of approximating space $L.$ Thus, the degree of ``sparsity'' of 
the oracle will be characterized by the alignment coefficient that already appeared 
in approximation error bounds of Section \ref{approx_error} and also by ``approximate dimension''
$d(w,\lambda)$ introduced below.   

We start, however, with a ``slow-rate'' oracle inequality that does not depend on 
``sparsity''. The inequalities of this type are well known in the literature 
on sparse recovery, in particular, for LASSO estimator in the case of finite 
dictionaries, see \cite{bartlett2009}, \cite{massart2010meynet}.  

Recall that ${\mathbb D}\subseteq L_1(\mu)$ is a convex set and $0\in {\mathbb D}.$
Recall also the definition of $q(\eps)$ (see \ref{def:q_eps}) and its properties. 
Note that  
\index{ao@$\sigma^2_Y$}
\begin{equation}
\label{sigmaY}
\sigma_Y^2= {\rm Var}(f_{\ast}(X))+ \sigma_{\xi}^2=\|f_{\ast}-\Pi f_{\ast}\|_{L_2(\Pi)}^2+
\sigma_{\xi}^2.
\end{equation}

\begin{theorem}
\label{th:slow_rate}
There exist absolute constants $\fr{C}, \fr{c}$ and $D$ such that the following holds.
\index{ap@$\bar s$}
For any $s\geq 1$ with 
$\bar s:=s+3\log (\log_2 n+2)+3 \leq \fr{c}\frac{\sqrt{n}}{\log n}$ and for all $\eps$ satisfying 
\begin{equation}
\label{eps-condition_sl}
\eps \geq D\frac{\sigma_YS(\mb T)}{ \sqrt n},
\end{equation}
with probability at least $1-e^{-s}$
\begin{align}
\label{slow_oracle}
\|f_{\hat \lambda_{\eps}, \hat a_{\eps}}-f_{\ast}\|_{L_2(\Pi)}^2 
+\frac 3 4\eps\|\hat\lambda_\eps\|_1\leq \inf\limits_{\lambda\in \mb D,a\in \mb R}\l[
\|f_{\lambda, a}-f_{\ast}\|_{L_2(\Pi)}^2+\frac 3 2 \eps\|\lambda\|_1\r]+\fr{C}\frac{\sigma_Y^2 \bar s}{n}.
\end{align}
\end{theorem}

As was mentioned earlier, our main goal is to obtain sharper bounds which would demonstrate connections between the risk of $f_{\hat\lambda_\eps,\hat a}$ and the degree of sparsity of an underlying model. 
Our next result is a step in this direction. 
\index{aq@Kolmogorov's $d$-width $\rho_d(\cdot)$}
We will need the notion of Kolmogorov's $d$-width of the set of random variables 
$C\subset {\mathcal L}_X$ defined as follows:
\begin{align*}
\rho_d(C):= \inf_{L\subset {\mathcal L}_X, {\rm dim}(L)\leq d}\sup_{\eta\in C}
\|P_{L^{\perp}}\eta\|_{L_2(\mb P)}.
\end{align*}
It characterizes the optimal accuracy of approximation of the set $C$
by $d$-dimensional linear subspaces of ${\mathcal L}_X.$
Given ${\mathbb T}'\subset {\mathbb T},$ let  
$$
X_{{\mathbb T}'}:= \{X(t)-{\mathbb E}X(t):t\in {\mathbb T}'\}.
$$
Recall that $\mb T_{w}:=\{t\in {\mathbb T}: |w(t)|\geq 1/2\}.$
Given an oracle $\lambda \in {\mathbb D}$ and $w\in \partial\|\lambda\|_1,$ 
let 
$$
\rho_d(w):= \rho_d\l(X_{\mb T_{w}}\r).
$$ 
\index{ar@$d(w,\lambda)$}
The following number will play a role of approximate dimension 
of the set of random variables $X_{\mb T_{w}}:$
\begin{align}
\label{eq:approx_dim}
d(w, \lambda):= \min\Bigl\{d\geq 0:\frac{d\sigma_Y^2}{n}\geq \|\lambda\|_1 
\frac{\gamma_2\l(\rho_d(w)\r)}{\sqrt{n}}\Bigr\}.
\end{align}

\begin{theorem}
\label{th:main_AA}
There exist absolute constants $\fr{C}, \fr{c}$ and $D$ such that the following holds.
For any $s\geq 1$ with 
$\bar s:=s+3\log (\log_2 n+2)+3 \leq \fr{c}\frac{\sqrt{n}}{\log n}$ and for all $\eps$ satisfying 
\begin{equation}
\label{eps-condition}
\eps \geq D\frac{\sigma_YS(\mb T)\sqrt{s}}{\sqrt n},
\end{equation}
with probability at least $1-e^{-s}$
\begin{align}
\label{oracle}
\l\|f_{\hat\lambda_\eps,\hat a_{\eps}}-f_{\ast}\r\|_{L_2(\Pi)}^2 
\leq 
\inf_{\lambda \in {\mathbb D}, w\in \partial \l\|\lambda\r\|_1, a\in {\mathbb R}}&
\biggl[\l\|f_{\lambda,a}-f_{\ast}\r\|_{L_2(\Pi)}^2+ 2\eps^2 \fr a^2(w)
+\fr{C}\frac{\sigma_Y^2 \, d(w,\lambda)}{n} \\
\nonumber
&+\fr{C}\frac{\|\lambda\|_1^2 S^2({\mathbb T})}{n}\biggr]
+\fr{C}\frac{\sigma_Y^2 \bar s}{n}.
\end{align}
\end{theorem}

Under an additional assumption that $\|\lambda\|_1$ is not too large, it is possible 
to prove the following modified version of Theorem \ref{th:main_AA} without the term 
$\fr{C}\frac{\|\lambda\|_1^2 S^2({\mathbb T})}{n}$ in the oracle inequality.

\begin{theorem}
\label{th:bounded_norm}
Assume that conditions of Theorem \ref{th:main_AA} hold. 
If $\mb D$ is such that 
$$
\mb D\subset \l\{\lambda\in L_1(\mu): \  \|\lambda\|_1\leq \frac{\fr{c}\sigma_Y\sqrt{n}}{S(\mathbb T)}\r\}
$$ 
for some absolute constant $\fr{c}>0$, then 
with probability $\geq 1-e^{-s}$ 
\begin{align}
\label{eq:oracle_2}
&\l\|f_{\hat\lambda_\eps,\hat a_{\eps}}-f_{\ast}\r\|_{L_2(\Pi)}^2 
\leq 
\nonumber
\\
&\inf_{\lambda \in {\mathbb D}, w\in \partial \l\|\lambda\r\|_1, a\in {\mathbb R}}
\biggl[\l\|f_{\lambda,a}-f_{\ast}\r\|_{L_2(\Pi)}^2+ 2\eps^2 \fr a^2(w)
+\fr{C}\frac{\sigma_Y^2 \, d(w,\lambda)}{n}\biggr]
+\fr{C}\frac{\sigma_Y^2 \bar s}{n}.
\end{align}
\end{theorem}

The proof of this result follows from the proof of Theorem \ref{th:main_AA}, see remark \ref{rem:bounded_norm} for more details.

\begin{remark}
Note that the oracle inequality of Theorem \ref{th:main_AA} is \textit{sharp}, meaning that the constant in front of $\|f_{\lambda,a}-f_\ast\|_{L_2(\Pi)}$ (the leading constant) is $1.$ 
It is possible to derive an oracle inequality with the leading constant larger than $1$ which might yield faster rates when the variance of the noise $\sigma^2_\xi$ is small. 
Define the following version of the ``approximate dimension'' (compare to (\ref{eq:approx_dim})):
$$
d_{\sigma_\xi}(w, \lambda):= \min\Bigl\{d\geq 0:\frac{d\sigma^2_\xi}{n}\geq \|\lambda\|_1 
\frac{\gamma_2\l(\rho_d(w)\r)}{\sqrt{n}}\Bigr\}.
$$
Then, under the assumptions of Theorem \ref{th:main_AA}, the following inequality holds with probability $\geq 1-e^{-s}$:
\begin{align*}
\l\|f_{\hat\lambda_\eps,\hat a_\eps}-f_{\ast}\r\|_{L_2(\Pi)}^2 
\leq &
\inf_{\lambda \in {\mathbb D}, w\in \partial \l\|\lambda\r\|_1, a\in {\mathbb R}}
\biggl[2\l\|f_{\lambda,a}-f_{\ast}\r\|_{L_2(\Pi)}^2+ 2\eps^2 \fr a^2(w)
\\
\nonumber
&+\fr{C}\frac{\sigma^2_\xi d_{\sigma_\xi}(w,\lambda)}{n}
+\fr{C}\frac{\|\lambda\|_1^2 S^2({\mathbb T})}{n}\biggr]
+\fr{C}\frac{\sigma_Y^2 \bar s}{n}.
\end{align*}
\end{remark}
The proof of this result uses arguments similar to the proof of Theorem \ref{th:main_AA}, so we omit the details.

Inequality (\ref{oracle}) above depends on rather abstract parameters (such as
the alignment coefficient and the approximate dimension)  
that have to be further bounded before one can get a meaningful bound in concrete examples. 
This will be dicussed in some detail in the following sections.

\section{Bounding the alignment coefficient}
\label{examples}

First, we discuss the bounds on alignment coefficient in terms of Sobolev-type 
norms in some detail. After this, we turn to the problem of bounding the alignment 
coefficient in the cases when there exists a weakly correlated partition for the 
design process $X.$

\subsection{Sacks-Ylvisacker conditions}

In the univariate case $\mb T=[0,1],$ it is possible to determine whether (a certain subspace of) the Sobolev space can be continuously embedded into $\mb H(K)$ based on the smoothness of the covariance function $k(\cdot,\cdot).$ 
Existence of such an embedding is given by the so-called {\it Sacks-Ylvisaker} conditions \cite{sacks1966designs}. 
This provides a way to bound the RKHS norm $\|\cdot\|_K$ generated 
by the covariance function of $X$ (and, thus, also the alignment coefficient) in terms of a Sobolev norm.
Definitions and statements below are taken from \cite{ritter1995multivariate}, Section 3. 

Set $\Omega_+:=\l\{(s,t)\in (0,1)^2: \ s>t\r\}$ and $\Omega_-:=\l\{(s,t)\in (0,1)^2: \ s<t\r\}$. 
Let $G$ be a continuous function on $\Omega_+\cup \Omega_-$ such that the restrictions $\l.G\r|_{\Omega_j}$ are continuously extendable to the closures ${\rm cl}(\Omega_j), \ j\in\l\{+,-\r\}$. 
$G_j$ will stand for the extension of $G$ to $[0,1]^2$ which is continuous on ${\rm cl}(\Omega_j)$ and on $ [0,1]^2\setminus {\rm cl}(\Omega_j)$. 
Set $R^{(k,l)}(s,t)=\frac{\partial^{k+l}}{\partial s^k\partial t^l}R(s,t)$. 
Then, the covariance kernel $k(\cdot,\cdot)$ defined on $[0,1]^2$ satisfies the Sacks-Ylvisaker conditions of order $r\in \mb N$ if the following holds true:
\begin{enumerate}[(A)]
\item $k\in C^{r,r}([0,1]^2)$, the partial derivatives of $G=k^{(r,r)}$ up to order 2 are continuous on $\Omega_+\cup \Omega_-$ and are continuously extendable to ${\rm cl}(\Omega_+)$ and to ${\rm cl}(\Omega_-)$. 
\item $\min\limits_{0\leq t\leq 1}\l(G_-^{(1,0)}(t,t)-G_+^{(1,0)}(t,t)\r)>0$.
\item $G_+^{(2,0)}(t,\cdot)$ belongs to the RKHS with reproducing kernel $G$ and 
$$
\sup\limits_{t\in [0,1]}\|G_+^{(2,0)}(t,\cdot)\|_{G}<\infty.
$$
\item In the case $r\geq 1$, 
$
k^{(0,j)}(\cdot,0)=0 \text{ for } 0\leq j\leq r-1.
$
\end{enumerate}

Let $\mb W_0^{2,r+1}$ be the subspace of $\mb W^{2,r+1}$ defined by 
$$
\mb W_0^{2,r+1}=\l\{f\in \mb W^{2,r+1}: \ f^{(j)}(0)=f^{(j)}(1)=0 \text{ for } 0\leq j\leq r\r\}.
$$
\begin{theorem}[Corollary 1 in \cite{ritter1995multivariate}]
Assume $k(\cdot,\cdot)$ satisfies the Sacks-Ylvisaker conditions of order $r$. 
Then
$\mb W_0^{2,r+1}\subset \mb H(K)$ and the embedding $\mb W_0^{2,r+1}\hookrightarrow \mb H(K)$ is continuous.
\end{theorem}
As a result, we have the bound 
$
\|w\|_K\leq C\|w\|_{{\mb W}^{2,r+1}}
$
that holds for all $w$ with some constant $C>0.$

It is well-known that the covariance function $k_1(s,t)=s\wedge t$ of the Brownian motion and $k_2(s,t)=e^{-|s-t|}$ of Ornstein-Uhlenbeck process satisfy Sacks-Ylvisaker conditions of order $r=0$. 
\begin{corollary}
Let $X(t), \ t\in[0,1]$ be the Ornstein-Uhlenbeck process and let $\mb H(K)$ be the associated reproducing kernel Hilbert space. 
If $w\in \mb W^{2,1}[0,1]$ is such that $w(0)=w(1)=0$, then $w\in \mb H(K)$ and
$$
\|w\|_K\leq C\|w\|_{\mb W^{2,1}[0,1]}.
$$
\end{corollary}
This should be compared to the exact description of $\mb H(K)$, the kernel of Ornstein-Uhlenbeck process, which is known to be 
\begin{align*}
\mb H(K)=\Big\{w\in L_2[0,1]: &\ \|w\|^2_{K}=\frac{w^2(0)+w^2(1)}{2}\\
&
+\frac 1 4 \int_0^1 w^2(t)dt+\int_0^1 \l(w'(t)\r)^2 dt<\infty\Big\}.
\end{align*}

\subsection{Discrete Sobolev norms and the Brownian motion}
\label{released_brownian}

In this example, we look back at the case when the design process is a Brownian motion 
(it was already discussed in Section \ref{approx_error}). However, this time we make 
a more realistic assumption that the design processes are observed only at discrete 
points. 

Assume that $\l\{X(t), \ t\in [0,1]\r\}$ is a standard Brownian motion released at zero, that is 
$X(t)=Z+W(t)$, where $Z$ is a standard normal random variable independent of $W$.  
Suppose that we observe $n$ iid copies of $X$, $X_1,\ldots, X_n$ on a grid $\mb T=\m G_N=\l\{0\leq  t_1<\ldots<t_N\leq 1\r\}$. Let $\mu$ be a counting measure on $\mb T.$
If, for example, the grid is uniform with mesh size $1/N$ for some large $N$, with high probability the adjacent columns of the design matrix 
$\bigl(X_i(t_j)\bigr)_{i\leq n, j\leq N}$ will be almost collinear. 
To the best of our knowledge, direct analysis based on the restricted eigenvalue type conditions \cite{bickel2009simultaneous} provides unsatisfactory bounds in such cases. 
On the other hand, results that hold true without any assumptions on the design (e.g., \cite{koltchinskii2011nuclear}, first statement of Theorem 1) only guarantee ``slow'' rates of convergence (of order $n^{-1/2}$, where $n$ is the size of a training data set).

The covariance function $k(\cdot,\cdot)$ of $X$ satisfies 
$
k(t_i,t_j)=1+t_i\wedge t_j.
$ 
Let $K=\Bigl(1+t_i\wedge t_j\Bigr)_{i,j=1}^N$ be the associated Gram matrix and let $K=LL^T$ be its Cholesky factorization. 
Note that 
$$
L=\begin{pmatrix}
\sqrt{1+t_1} &0 &\ldots & 0\\ 
\sqrt{1+t_1} & \sqrt{t_2-t_1} & 0&\vdots \\
\vdots & &  \ddots & 0\\
\sqrt{1+t_1} & \sqrt{t_2- t_1} & \ldots & \sqrt{t_N-t_{N-1}}
\end{pmatrix}
$$
By Factorization theorem (or a straightforward argument), for any $w\in \mb R^N$, $\|w\|_K=\l\|L^{-1}w\r\|_2$. 
If $w=Lv$, then direct computation shows
\begin{align}
\label{discrete}
&
\l\|v\r\|^2_2=\|w\|^2_K=\frac{w_1^2}{1+t_1}+\sum_{j=2}^N \frac{(w_j-w_{j-1})^2}{t_j-t_{j-1}}.
\end{align}
The latter expression can be seen as a discrete analogue of the Sobolev norm. 
For example, let the grid $\m G_N$ be uniform, that is, $t_j=\frac{j-1}{N}, \ j=1\ldots N$, and let $\lambda: \m G_N\mapsto \mb R$ be sparse in the following sense: 
${\supp(\lambda)}=\l\{t_{i_1}<t_{i_2}<\ldots<t_{i_s}\r\}$ so that $\l|\supp(\lambda)\r|=s$ and 
$$
\min\limits_{2\leq k\leq s}|t_{i_k}-t_{i_{k-1}}|= \sigma\gg\frac 1 N.
$$
It is clear from (\ref{discrete}) that 
$\inf\limits_{w\in \partial \|\lambda\|_1}\|w\|_K\leq C\sqrt{\frac{s}{\sigma}}$ for some absolute constant $C>0$ (e.g., take a vector whose entries linearly interpolate the sign pattern of $\lambda$) while the trivial choice $w(t_j)=\sign(\lambda(t_j))$ leads to $\|w\|_K\geq c\sqrt{Ns}$. \\

Note also that if $w_j:=w(t_j), j=1,\dots, N$ for a smooth function $w\in {\mb W}^{2,1}([0,1])$
(with a slight abuse of notation, we write $w$ both for the vector in $\mb R^N$ and 
for the function),  
then, by Cauchy-Schwarz inequality, 
$$
\frac{(w_j-w_{j-1})^2}{t_j-t_{j-1}}=\frac{(w(t_j)-w(t_{j-1}))^2}{t_j-t_{j-1}}
=\frac{\biggl(\int_{t_{j-1}}^{t_j}w'(s)ds\biggr)^2}{t_j-t_{j-1}}
\leq \int_{t_{j-1}}^{t_j}|w'(s)|^2 ds.
$$
It immediately implies that 
$
\|w\|_K^2\leq |w(0)|^2 + \int_{0}^1 |w'(s)|^2 ds,
$
so the discrete Sobolev norm needed to control the alignment coefficient 
is bounded from above by its continuous counterpart. 
As a matter of fact, we have that 
$$
\|w\|_K^2\leq \inf_{\tilde w}\biggl[{|\tilde w(0)|^2 + \int_{0}^1 |\tilde w'(s)|^2 ds}\biggr],
$$
where the infimum is taken over all functions $\tilde w\in {\mb W}^{2,1}([0,1])$ such that 
$\tilde w(t_j)=w_j, j=1,\dots, N.$

These observations allow one to characterize the prediction performance of LASSO estimator in terms of $s$ and $\sigma$, in particular, rates faster than $n^{-1/2}$ can be deduced from Theorem \ref{th:main_AA}.

\subsection{Stationary processes}
\label{stationary}

In this subsection, we derive Sobolev norm bounds on alignment coefficient in the case 
when $X$ is a stationary process (or a stationary random field).

Let $\mb T\subset \mb R^d$ be a bounded open set and let $\mu$ be the Lebesgue measure. 
Consider a stationary random field $\l\{X(t), \ t\in \mb R^d \r\}$ with continuos covariance function $k:$
$$
k(t-s)=\Cov(X(t),X(s)), \ t,s\in \mb R^d.
$$ 
By Bochner's theorem, there exists a finite Borel measure $\nu$ such that 
\begin{align}
\label{bochner}
k(t)=\int_{\mb R^d} e^{i \langle t,u\rangle}\nu(du),\  t\in \mb R^d 
\end{align}
called the \textit{spectral measure} of $X.$ In what follows, we assume that 
$\nu$ is absolutely continuous with spectral density $v:\mb R^d\mapsto \mb R_+.$

\begin{proposition}
\label{stationary_sobolev}
Suppose that 
\begin{equation}
\label{spectral_lower}
v(t)\geq \frac{c}{(1+|t|^2)^{p}}, t\in \mb R^d
\end{equation}
for some $p>d/2$ and $c>0.$ For $w$ defined on $\mb T$, let 
$$
\Omega(w):=\l\{\tilde w:\mb R^d\mapsto \mb R : \ \tilde w(t)=w(t), \ t\in \mb T\r\}.
$$  
Then 
$$
\|w\|_K\leq C\inf\limits_{\tilde w\in \Omega(w)}\|\tilde w\|_{\mb W^{2,p}(\mb R^d)}.
$$
\end{proposition}

Note that condition (\ref{spectral_lower}) could not hold for $p\leq d/2$
since this would contradict integrability of the spectral density $v$.

\begin{proof}
Given $u\in L_1(\mb R^d)\cap L_2(\mb R^d)$, let $\hat u$ be its Fourier transform. 
Observe that
$$
\Var(f_u(X))=\iint k(t-s)u(t)u(s)dtds=\int |\hat u(z)|^2 v(z)dz.
$$
For $u$ supported in $\mb T$ and $\tilde w\in \Omega(w)$, this gives
\begin{align*}
&
\dotp{w}{u}_{L_2(\mb T,\mu)}=\dotp{\tilde w}{u}_{L_2(\mb R^d)}=\dotp{\widehat{\tilde w}}{\hat u}_{L_2(\mb R^d)}=\\
&
=\dotp{\frac{\widehat{\tilde w}}{\sqrt v}}{\hat u \sqrt{v}}\leq C\l\|(1+|x|^2)^{p/2} \widehat{\tilde w}\r\|_{L_2(\mb R^d)}\Var (f_u(X)),
\end{align*}
hence $\|w\|_K\leq C\l\|(1+|x|^2)^{p/2} \widehat{\tilde w} \r\|_{L_2(\mb R^d)}$. 
It remains to note that by the properties of Fourier transform
$$
\l\|(1+|x|^2)^{p/2}\widehat{\tilde w}\r\|_{L_2(\mb R^d)}\leq C\|\tilde w\|_{\mb W^{2,p}(\mb R^d)}.
$$

\end{proof} 

We now turn to the case of stationary processes observed at discrete points. 
Let $\l\{X(t), \ t\in\mb R^d\r\}$ be a (weakly) stationary random field, 
and let $X_1,\ldots,X_n$ be i.i.d. copies of $X$ observed on the grid 
$\mb T=\m G_N=\l\{t_j=\frac{2\pi}{N}j, \ j\in \{1,\ldots,N\}^d\r\}$ for some even $N$.
In this case, functions on $\mb T$ can be identified with vectors in ${\mb R}^{N^d}.$ 
We also assume 
that $\mu$ is the counting measure on $\mb T.$

\begin{assumption}
Suppose the following condition on the spectral density $v$ of the process $X$ holds:
\begin{align}
\label{density}
&
c_1\l(\frac{1}{1+|t|^2}\r)^p\leq v(t)\leq c_2\l(\frac{1}{1+|t|^2}\r)^p \text{ for some } p>\frac d 2
\end{align} 
where $0<c_1\leq c_2<\infty$. 
\end{assumption}
\begin{proposition}
\label{interpolation}
Given $\vec w=(w_1,\ldots, w_{N^d})^T\in \partial \|\lambda\|_1$, let 
$$
\Omega_N(\vec w)=\l\{w\in \mb W^{2,p}(\mb R^d): \ w(2\pi j/N)=w_j, \ j\in \mb Z^d\r\},
$$ 
where $w_j$ are defined arbitrarily for $j\notin\l\{1,\ldots,N\r\}^d$.
Under the above-stated assumptions,
$$
\l\|\vec w\r\|_K\leq C\inf_{w\in \Omega_N(\vec w)}\|w\|_{\mb W^{2,p}(\mb R^d)}.
$$
\end{proposition}
Proof is outlined in section \ref{pf:interpolation}. 
Implications of this result for the risk of $\hat\lambda_\eps$ are presented in Theorem \ref{th:main_stationary_B} below. 
In particular, we show that rates faster than $n^{-1/2}$ are often possible. 

\subsection{Sparse multiple linear models and weakly correlated partitions}
\label{sec:multiple}

In this section, we assume that 
\begin{align}
\label{l2}
&
Y=a+\sum_{j=1}^N \int\limits_{\mb T_j} X^{(j)}(t_j)d\Lambda_j(t_j) + \xi,
\end{align}
where $a\in \mb R,$ $\mb T_1,\dots, \mb T_N$ are measurable spaces equipped with 
$\sigma$-algebras ${\mathcal B}_1,\dots, {\mathcal B}_N$ and finite measures 
$\mu_1,\dots, \mu_N,$ $X^{(1)},\dots, X^{(N)}$ are subgaussian stochastic processes on 
$\mb T_1,\dots, \mb T_N,$ 
$\Lambda_1,\dots,\Lambda_N$ are signed measures on spaces $\mb T_1,\dots, \mb T_N$ with bounded total variations, and $\xi$ is a zero-mean random variable independent 
of $X^{(1)},\dots, X^{(N)}.$ Suppose ${\mathcal B}_j={\mathcal B}_{{\mb T}_j}$
(Borel $\sigma$-algebra in the semimetric space $(\mb T_j, d_{X^{(j)}})$). 
Without loss of generality, we can assume that the sets $\mb T_1,\dots,\mb T_N$ form a partition 
of the space $\mb T:=\bigcup\limits_{j=1}^N \mb T_j$ equipped with $\sigma$-algebra ${\mathcal B}$
and measure $\mu$ such that the measures $\mu_j$ are restrictions of $\mu$ on $\mb T_j.$ Similarly, 
signed measures $\Lambda_j$ become restrictions on $\mb T_j$ of a signed measure $\Lambda_{\ast}$
on $(T,{\mathcal B}).$ We will set $X(t):=X^{(j)}(t), \ t\in \mb T_j, j=1,\dots, N$
and, finally, we can assume that ${\mathcal B}={\mathcal B}_{\mb T}$ is the Borel $\sigma$-algebra in the semimetric space $({\mb T},d_X).$ 

We are interested in the situation when the processes $\l\{X^{(j)}(t), \ t\in \mb T_j\r\}, \ j=1,\ldots, N$ are {\it weakly correlated} (in particular, they can be independent). 
The number of predictors $N$ can be very large, but $Y$ might depend only on $X^{(j)}(t),t\in \mb T_j \ j\in J\subset \l\{1,\ldots,N\r\},$ where $\card(J)\ll N$, 
whence $\card(J)$ naturally represents the degree of sparsity of the problem.
Another interpretation of the model is to assume that the domain $\mb T$ of the stochastic process 
$X$ can be partitioned in disjoint sets $\mb T_j$ so that $\{X(t):t\in \mb T_j\}, j=1,\dots, N$ are 
``weakly correlated'', but only few of the elements of partition are correlated with the response 
variable $Y.$ It is important to emphasize that the results of the following sections concerning estimator (\ref{empirical}) are \textit{adaptive} with respect to the partitions, in particular, we do not need to know the ``weakly correlated'' parts in advance, but the estimator adapts to such a structure (given that it exists).
 
Let $K_j$ be the covariance operator of $X^{(j)}$ and $k_j$ its kernel (the covariance 
function of $X^{(j)}$).  
Our next goal is to understand how to control the alignment coefficient $\fr a(\cdot)$ associated with the process $X$ in terms of the RKHS-norms $\|\cdot\|_{K_j}, \ j=1,\ldots,N$. 

Without loss of generality, assume that $X^{(j)}, \ j=1,\ldots, N$, are centered. 
Given $u\in L_1(\mb T,\mu)$, it can be represented as  
$u=\sum\limits_{j=1}^N u_j$ with $\supp(u_j)\subseteq \mb T_j$. 
Given $\gamma>0$, define  
\index{as@$C_{\gamma,J}$}
$$
C_{\gamma,J}:=
\Bigl\{u\in L_1(\mb T,\mu): 
\sum_{j\not\in J}\|f_{u_j}\|_{L_2(\Pi)}\leq 
\gamma \sum_{j\in J}\|f_{u_j}\|_{L_2(\Pi)}
\Bigr\}
$$ 
and
\index{at@$\beta_2^{(\gamma)}(J)$}
$$
\beta_2^{(\gamma)}(J):=\inf\biggl\{\beta>0: \sum_{j\in J}\|f_{u_j}\|_{L_2(\Pi)}^2
\leq \beta^2 \|f_u\|_{L_2(\Pi)}^2, u\in C_{\gamma,J}\biggr\}.
$$
Clearly, if $X^{(j)}(t), t\in {\mb T}\ j=1\ldots N$ are uncorrelated, then $\beta_2^{(\gamma)}(J)=1$ for any nonempty $J\subseteq \l\{1,\ldots,N\r\}$.  
More generally, we have the following result:

\begin{proposition}
\label{bound_beta}
For all $J\subset \l\{1,\dots, N\r\}$ and all $w=\sum\limits_{j\in J}w_j$
such that  $\supp(w_j)\subseteq \mb T_j$ and $\|w_j\|_{K_j}<\infty$, we have
\begin{align}
\label{gamma_02}
\fr a^{(b)}(w)\leq \beta_2^{(\gamma)}(J)\Bigl(\sum_{j\in J}\|w_j\|_{K_j}^2\Bigr)^{1/2},
\end{align}
where $\gamma=b\max\limits_{1\leq j\leq N}\|k_j\|_{\infty}^{1/2}\max_{j\in J}\|w_j\|_{K_j}$.
\end{proposition}

\begin{proof}
Note that since $w=\sum\limits_{j\in J}w_j$ with ${\rm supp}(w_j)\subseteq {\mathbb T}_j,$
$$
\mb T_w\subset \bigcup_{j\in J}{\mathbb T}_j \text{ and } {\mathbb T}\setminus {\mathbb T}_w
\supset \bigcup_{j\not\in J}{\mathbb T}_j.
$$
For all $u\in C_w^{(b)}$ (defined in (\ref{eq:cone})), we have 
\begin{align*}
\sum_{j\not\in J}\|f_{u_j}\|_{L_2(\Pi)}&\leq 
\max\limits_{1\leq j\leq N}\|k_j\|_{\infty}^{1/2}
\sum_{j\not\in J}\|u_j\|_1\leq \\
&\leq
\max\limits_{1\leq j\leq N}\|k_j\|_{\infty}^{1/2}\int\limits_{{\mathbb T}\setminus {\mathbb T}_w}|u|d\mu 
\leq 
b\max\limits_{1\leq j\leq N}\|k_j\|_{\infty}^{1/2}\dotp{w}{u}.
\end{align*}
Since 
$$
\dotp{w}{u}=
\sum_{j\in J}
\dotp{w_j}{u_j}
\leq 
\sum_{j\in J}\l\|w_j\r\|_{K_j} \l\|f_{u_j}\r\|_{L_2(\Pi)}
\leq 
\max_{j\in J}\l\|w_j\r\|_{K_j} \sum_{j\in J}\l\|f_{u_j}\r\|_{L_2(\Pi)},
$$
we can conclude that 
$$
\sum_{j\not\in J}\l\|f_{u_j}\r\|_{L_2(\Pi)}\leq 
b\max\limits_{1\leq j\leq N}\|k_j\|_{\infty}^{1/2}
\max_{j\in J}\l\|w_j\r\|_{K_j} \sum_{j\in J}\l\|f_{u_j}\r\|_{L_2(\Pi)}.
$$
We proved that $C_w^{(b)}\subseteq C_{\gamma,J}$ for 
$\gamma:=b\max\limits_{1\leq j\leq N}\|k_j\|_{\infty}^{1/2}\max\limits_{j\in J}\|w_j\|_{K_j}.$ 
For all $u\in C_w^{(b)}\subseteq C_{\gamma,J},$ we have 
\begin{align*}
\dotp{w}{u}&=\sum_{j\in J}
\dotp{w_j}{u_j}\leq 
\Bigl(\sum_{j\in J}\|w_j\|_{K_j}^2\Bigr)^{1/2}
\Bigl(\sum_{j\in J}\|f_{u_j}\|_{L_2(\Pi)}\Bigr)^{1/2}\leq \\
&
\leq \beta_2^{(\gamma)}(J)\Bigl(\sum_{j\in J}\|w_j\|_{K_j}^2\Bigr)^{1/2}
\l\|f_u\r\|_{L_2(\Pi)},
\end{align*}
implying that 
$$
\fr a^{(b)}(w)\leq \beta_2^{(\gamma)}(J)\Bigl(\sum_{j\in J}\|w_j\|_{K_j}^2\Bigr)^{1/2}
$$
with 
$\gamma:=b\max\limits_{1\leq j\leq N}\|k_j\|_{\infty}^{1/2}\max_{j\in J}\|w_j\|_{K_j}.$

\end{proof}

Next, we will relate $\beta_2^{(\gamma)}(J)$ to the size of \textit{restricted isometry} \cite{candes2006stable} constants associated with partition 
$\mb T_1,\ldots,\mb T_N$. 
\index{au@Restricted isometry constant $\delta_d$}
Given an integer $d\geq 1$, we define the restricted isometry constant $\delta_d$ as the smallest $\delta>0$ with the following property:
for any $J\subset \l\{1,\ldots, N\r\}$ with $\card(J)=d$, any $u_j, \ j\in J$ such that  $\supp(u_j)\subseteq \mb T_j$ and $\Var(f_{u_j}(X))=1$, 
the spectrum of the $d\times d$ matrix $\l(\Cov(f_{u_i}(X),f_{u_j}(X))\r)_{i,j\in J}$ belongs to $[1-\delta,1+\delta]$. 

\begin{proposition}
\label{prop:rip}
The following inequality holds for all $J\subset \l\{1,\ldots, N\r\}$ with 
$\card(J)\leq d:$
$$
\beta_2^{(\gamma)}(J)\leq \frac{1+\delta_{2d}}{(1-\delta_{2d})^2-\gamma\delta_{3d}}.
$$ 
\end{proposition}

In particular, it means that $\beta_2^{(\gamma)}$ can be bounded by a constant as soon as $\delta_{3d}< \frac{1}{2+\gamma}$.

\begin{proof}
The argument is similar to Lemma 7.2 in \cite{Koltchinskii2011Oracle-inequali00}, the details are included in Appendix \ref{sec:rip} for the reader's convenience.

\end{proof}

\section{Oracle inequalities and weakly correlated partitions}
\label{weakly_corr}

First, we will state a corollary of Theorem \ref{th:main_AA} concerning 
the model of weakly correlated partitions discussed in Section \ref{examples}. 
Let $\Delta:=\{\mb T_1,\dots, \mb T_N\}$ be a partition of the parameter space 
$\mb T$ into $N\geq 1$ measurable disjoint sets. 
\index{aua@$\mathcal T$}
Let ${\mathcal T}$ be the set 
of all such partitions. 
Let $X^{(j)}$ denote the restriction 
of stochastic process $X$ to the set $\mb T_j$ and let $K_j$ be the covariance 
operator of the process $X_j$ and $k_j$ be its covariance function. 
Consider an oracle $\lambda\in L_1(\mu)$ and denote 
\index{av@$J_\lambda, \ N(\lambda)$}
$$
J_{\lambda}:= \{j=1,\dots, N: \mb T_j \cap {\rm supp}(\lambda)\neq \emptyset\}.
$$
Also, denote 
$
N(\lambda):={\rm card}(J_{\lambda}).
$
Usually, we assume that $N$ is very large and $N(\lambda)$ is much smaller than $N,$
so, $N(\lambda)$ plays the role of ``sparsity parameter'' in this framework.
Let $w=\sum\limits_{j\in J_{\lambda}}w_j\in \partial \|\lambda\|_1$ be a subgradient 
such that ${\rm supp}(w_j)\subset \mb T_j, j\in J_{\lambda}.$ 
\index{aw@${\mathcal W}_{\lambda,\Delta}$}
In what follows, denote ${\mathcal W}_{\lambda,\Delta}$ the set 
of all such subgradients $w.$
Recall the definition of the quantity $\beta_2^{(\gamma)}(J)$ (Section \ref{examples})
and denote 
$$
\beta(w,\lambda):= \beta_2^{(\gamma)}(J_{\lambda}),\ \ 
\gamma:=16\max\limits_{1\leq j\leq N}\|k_j\|_{\infty}^{1/2}\max_{j\in J_{\lambda}}\|w_j\|_{K_j}.
$$
Proposition \ref{bound_beta} implies 
that 
\begin{align}
\label{a_beta}
\fr a(w)\leq \beta(w,\lambda)\Bigl(\sum_{j\in J_{\lambda}}\|w_j\|_{K_j}^2\Bigr)^{1/2}.
\end{align}

We will also need the following quantities that would play the role of  
``approximate dimensions'' of the sets of random variables $X_{{\mb T}_{w_j}}, j\in J$
(local versions of $d(w,\lambda)$):
\index{ax@${\fr d}_j (w,\lambda)$}
\begin{align}
\label{eq:approx_dim_j}
{\fr d}_j (w,\lambda):= \min\Bigl\{m\geq 0:
\frac{m\sigma_Y^2}{n}\geq 
\|\lambda\|_1
\frac{\gamma_2\l(\rho_m(w_j)\r)}{\sqrt{n}}\Bigr\}.
\end{align}

\begin{proposition}
\label{dim_loc}
Under the above notations, the following bound holds:
$$
d(w,\lambda)\leq \sum_{j\in J_{\lambda}}{\fr d}_j (w,\lambda).
$$
\end{proposition}

\begin{proof}
Denote $m_j:= {\fr d}_j (w,\lambda).$ Then, 
$$
\frac{m_j\sigma_Y^2}{n}\geq  
\|\lambda\|_1
\frac{\gamma_2\l(\rho_{m_j}(w_j)\r)}{\sqrt{n}}, j\in J_{\lambda}
$$
and, for all $j\in J_{\lambda}$ and $\delta>0,$ there exist $L_j\subset {\mathcal L}_X$ such 
that ${\rm dim}(L_j)\leq m_j$ and 
$$
\sup_{t\in \mb T_{w_j}}\|P_{L_j}^{\perp}(X(t)-{\mb E}X(t))\|_{L_2(\Pi)}\leq 
\rho_{m_j}(w_j)+\delta.
$$
Denote $L:={\rm l.s.}\Bigl(\bigcup\limits_{j\in J_{\lambda}}L_j\Bigr).$ Then, 
\begin{align*}
\sup_{t\in \mb T_{w}}\|P_{L}^{\perp}(X(t)-{\mb E}X(t))\|_{L_2(\Pi)} &
\leq \max_{j\in J_{\lambda}}\sup_{t\in \mb T_{w_j}}\|P_{L_j}^{\perp}(X(t)-{\mb E}X(t))\|_{L_2(\Pi)}\\
&
\leq \max_{j\in J_{\lambda}}\rho_{m_j}(w_j)+\delta
\end{align*}
and 
$$
\frac{\sigma_Y^2\sum\limits_{j\in J_{\lambda}}m_j}{n}\geq 
\|\lambda\|_1 
\frac{\gamma_2\l(\max_{j\in J_{\lambda}}\rho_{m_j}(w_j)\r)}{\sqrt{n}}.
$$
Since ${\rm dim}(L)\leq \sum\limits_{j\in J_{\lambda}}m_j:=m,$ we have  
$$
\rho_m(w)\leq 
\sup_{t\in \mb T_{w}}\|P_{L}^{\perp}(X(t)-{\mb E}X(t))\|_{L_2(\Pi)}
\leq \max_{j\in J_{\lambda}}\rho_{m_j}(w_j)+\delta.
$$
It follows that 
$$
\frac{\sigma_Y^2 m}{n}\geq 
\|\lambda\|_1
\frac{\gamma_2\l(\rho_{m}(w)-\delta\r)}{\sqrt{n}}.
$$
Since $\delta>0$ is arbitrary, this yields 
$$
\frac{\sigma_Y^2 m}{n}\geq 
\|\lambda\|_1
\frac{\gamma_2\l(\rho_{m}(w)\r)}{\sqrt{n}},
$$
and the result follows.

\end{proof}

As a very simple example, let $I_1,\dots, I_N$ be disjoint finite subsets of the set $\mb N$ of natural numbers and 
let  
$$
X(t)=X^{(j)}(t)=\sum_{k\in I_j} \eta_k^{(j)} \phi_k^{(j)}(t), t\in \mb T_j, \ j=1,\dots, N,
$$
where $\phi_k^{(j)}, k\in I_j$ are bounded measurable functions on $\mb T_j,$ $j=1,\dots, N$
and $\{\eta_k^{(j)}: k\in I_j, j=1,\dots, N\}$ are centered jointly normal random 
variables. Denote $m_j:={\rm card}(I_j), \ j=1,\dots, N.$ Let $\lambda \in {\mb D}$
and $w\in {\mathcal W}_{\lambda,\Delta}.$
Obviously, 
$$
{\fr d}_j(w,\lambda)\leq m_j,  \ j\in J_{\lambda},
$$
so, we have a simple bound 
$$
d(w,\lambda)\leq \sum_{j\in J_{\lambda}} m_j.
$$

The next statement immediately follows from Theorem \ref{th:main_AA}, Proposition \ref{dim_loc}
and bound (\ref{a_beta}).

\begin{corollary}
\label{cor:weak_correlation}
Suppose that assumptions and notations of Theorem \ref{th:main_AA} hold.
There exists an absolute constant $\fr{C}>0$ such that  
with probability at least $1-e^{-s}$
\begin{align}
\label{oracle_weak_correlation}
\l\|f_{\hat\lambda_\eps,\hat a_{\eps}}-f_{\ast}\r\|_{L_2(\Pi)}^2 
\leq &
\inf_{\substack{\Delta\in {\mathcal T},\lambda \in {\mathbb D}, \\ w\in {\mathcal W}_{\lambda,\Delta}, a\in {\mathbb R}}}
\biggl[\l\|f_{\lambda,a}-f_{\ast}\r\|_{L_2(\Pi)}^2+ 2\eps^2 \beta^2(w,\lambda)\sum_{j\in J_{\lambda}}\|w_j\|_{K_j}^2+
\\
\nonumber
&\fr{C}\frac{\sigma_Y^2 \sum\limits_{j\in J_{\lambda}}{\fr d}_j(w,\lambda)}{n}
+\fr{C}\frac{\|\lambda\|_1^2 S^2({\mathbb T})}{n}\biggr]
+\fr{C}\frac{\sigma_Y^2 \bar s}{n}.
\end{align}
\end{corollary}

The term $\frac{\|\lambda\|_1^2 S^2({\mathbb T})}{n}$ that depends 
on $\|\lambda\|_1^2$ can be dropped if $\|\lambda\|_1$ is not too large (see Theorem \ref{th:bounded_norm}). In general, this term can be controlled in terms of sparsity 
parameter $N(\lambda)$ and $\|\lambda\|_{L_2(\mu)}.$ 
To this end, note that, by Cauchy-Schwarz inequality,  
\begin{align*}
\|\lambda\|_1&=
\sum_{j\in J_{\lambda}}\int_{\mb T_j}|\lambda|d\mu 
\leq 
\sum_{j\in J_{\lambda}}\l(\int_{\mb T_j}|\lambda|^2d\mu\r)^{1/2}\mu^{1/2}(\mb T_j) \\
&
\leq \biggl(\sum_{j\in J_{\lambda}}\int_{\mb T_j}|\lambda|^2d\mu\biggr)^{1/2}
\biggl(\sum_{j\in J_{\lambda}}\mu (\mb T_j)\biggr)^{1/2}
\leq 
\|\lambda\|_{L_2(\mu)} \max_{j\in J_{\lambda}}\mu^{1/2}(\mb T_j) \sqrt{N(\lambda)}.
\end{align*}

For an arbitrary oracle $\lambda \in \mb T,$ arbitrary partition $\Delta\in {\mathcal T},$
arbitrary subgradient $w\in {\mathcal W}_{\lambda, {\mathcal T}}$ and for 
$$
\eps = D\frac{\sigma_YS(\mb T)\sqrt{s}}{\sqrt n},
$$
we have the following 
inequality that holds with probability at least $1-e^{-s}:$
\begin{equation}
\label{oracle_weak_correlation_aa}
\l\|f_{\hat\lambda_\eps,\hat a_{\eps}}-f_{\ast}\r\|_{L_2(\Pi)}^2 
\leq 
\l\|f_{\lambda,a(\lambda)}-f_{\ast}\r\|_{L_2(\Pi)}^2+ 
\fr{C}\biggl[{\fr Q}(w,\lambda,\Delta)\frac{N(\lambda)}{n}+\frac{\sigma_Y^2 
\bar s}{n}\biggr],
\end{equation}
where 
\begin{align*}
{\fr Q}(w,\lambda,\Delta):=& \sigma_Y^2 S(\mb T)^2\beta^2(w,\lambda)
\max_{j\in J_{\lambda}}\|w_j\|_{K_j}^2 s+ 
\sigma_Y^2 \max_{j\in J_{\lambda}}{\fr d}_j(w,\lambda)+\\
&
S^2({\mathbb T})\|\lambda\|_{L_2(\mu)}^2 \max_{j\in J_{\lambda}}\mu(\mb T_j).
\end{align*}
Thus, if there is an oracle $\lambda\in \mb D$ for which the approximation error
$\l\|f_{\lambda,a(\lambda)}-f_{\ast}\r\|_{L_2(\Pi)}^2$
is small and the quantitity ${\fr Q}(w,\lambda,\Delta)$
is of a moderate size, then the error of the estimator $(\hat \lambda_{\eps}, \hat a_{\eps})$ is essentially controlled by the quantity $\frac{N(\lambda)}{n}$ (up to log factors). Since $N(\lambda)$
can be viewed as a degree of sparsity of the oracle $\lambda,$ this explains the connection 
of the oracle inequality of Corollary \ref{cor:weak_correlation} and now classical bounds for 
LASSO in the case of large finite dictionaries. 
Once again, it is important to emphasize that the estimation method (\ref{empirical}) does not require 
any knowledge of a ``weakly correlated partition'' $\Delta.$ The method is adaptive in the sense 
that, if there exists a partition $\Delta$ such that $\beta (w,\lambda)$ and other quantities 
involved in the definition of ${\fr Q}(w,\lambda,\Delta)$ are not large, then the size 
of the error depends on the degree of sparsity $N(\lambda)$ with respect to the partition 
of oracles $\lambda$ that provide good approximation of the target.

In the simplest example, $\mb T:=\{1,\dots, N\}$ and the partition $\Delta:=\Bigl\{\{1\},\dots, \{N\}\Bigr\}$ (so, $\mb T$ is partitioned in one point sets). Let $\mu $ be the counting 
measure. Thus, $X$ is an $N$-dimensional 
subgaussian vector and we are in the framework of standard high-dimensional multiple regression 
model. 
For simplicity, assume that $X$ is scaled in such a way that 
${\mb E} X(t)=0,$ ${\mb E} X^2(t)=1.$
The estimator (\ref{empirical}) becomes a version of usual LASSO-estimator.  
Then, it is easy to check that 
$
S(\mb T)\leq C \sqrt{\log N}.
$
Also, in this case RKHS-spaces ${\mb H}(K_j), j=1,\dots, N$ are one-dimensional and 
we have 
$
\|w_j\|_{K_j}= |w(j)|, j=1,\dots, N.
$
For an oracle $\lambda \in {\mb D},$ 
$$
N(\lambda)={\rm card}(J_{\lambda}),\ J_{\lambda}={\rm supp}(\lambda)=\{1\leq j\leq N: \lambda_j\neq 0\}.
$$
In this case, we can set $w(j)={\rm sign}(\lambda(j)), j=1,\dots, N.$
Also, we obviously have ${\fr d}_j(w,\lambda)=1.$ 
Finally, in this case the quantity $\beta_2^{(\gamma)}(J)$ coincides with standard  
``cone constrained'' characteristiscs frequently used in the literature on sparse 
recovery (see, e.g., \cite{Koltchinskii2011Oracle-inequali00}, Section 7.2.2). 
We will use $\beta(\lambda)=\beta(w,\lambda)=\beta_2^{(16)}(J_{\lambda}).$
Then, Corollary \ref{cor:weak_correlation} takes the following form.

\begin{corollary}
\label{cor:finite_dictionary}
There exist absolute constants $\fr{C}, \fr{c}$ and $D$ such that the following holds.
For any $s\geq 1$ with 
$\bar s:=s+3\log (\log_2 n+2)+3 \leq \fr{c}\frac{\sqrt{n}}{\log n}$ and for all $\eps$ satisfying 
\begin{equation}
\label{eps-condition''}
\eps \geq D\frac{\sigma_Y\sqrt{s \log N}}{\sqrt n},
\end{equation}
with probability at least $1-e^{-s}$
\begin{align}
\label{oracle_finite_dictionary}
\l\|f_{\hat\lambda_\eps,\hat a_{\eps}}-f_{\ast}\r\|_{L_2(\Pi)}^2 
\leq &
\inf_{\lambda \in {\mathbb D}, a\in {\mathbb R}}
\biggl[\l\|f_{\lambda,a}-f_{\ast}\r\|_{L_2(\Pi)}^2+ 2\beta^2(\lambda)N(\lambda)\eps^2
+
\\
\nonumber
&\fr{C}\frac{\sigma_Y^2 N(\lambda)}{n}
+\fr{C}\frac{\|\lambda\|_1^2 \log N}{n}\biggr]
+\fr{C}\frac{\sigma_Y^2 \bar s}{n}.
\end{align}
\end{corollary}

More generally, assume that $\mb T$ is a finite set with a counting measure $\mu$ and consider an arbitrary partition $\Delta =\{\mb T_1,\dots, \mb T_N\}$ of $\mb T.$ Denote 
$$
m_j:=\mu (\mb T_j)={\rm card}(\mb T_j), j=1,\dots, N
$$
and $m:=\mu(\mb T)={\rm card}(\mb T).$ As before, $X$ is subgaussian 
and ${\mb E} X(t)=0,$ ${\mb E} X^2(t)=1.$ Then, we have 
$
S(\mb T)\leq C \sqrt{\log m}.
$
In this case, covariance operators $K_j$ are acting in $m_j$-dimensional 
Euclidean spaces and we have 
$$
\|w_j\|_{K_j}=\|K_j^{-1/2}w_j\|_2,\ j=1,\dots, N.
$$
Clearly, we also have ${\fr d}_j(w,\lambda)\leq m_j.$ Thus, the oracle inequality 
of Corollary \ref{cor:weak_correlation} implies that 
\begin{align}
\label{oracle_finite_dictionary_2}
\l\|f_{\hat\lambda_\eps,\hat a_{\eps}}-f_{\ast}\r\|_{L_2(\Pi)}^2 
\leq &
\inf_{\substack{\Delta\in {\mathcal T},\lambda \in {\mathbb D}, \\ w\in {\mathcal W}_{\lambda,\Delta}, a\in {\mathbb R}}}
\biggl[\l\|f_{\lambda,a}-f_{\ast}\r\|_{L_2(\Pi)}^2+ 2\beta^2(w,\lambda)\sum_{j\in J_{\lambda}}
\|K_j^{-1/2} w_j\|_2^2 \eps^2+
\\
\nonumber
&
\fr{C}\frac{\sigma_Y^2 \sum\limits_{j\in J_{\lambda}}m_j}{n}
+\fr{C}\frac{\|\lambda\|_1^2 \log m}{n}\biggr]
+\fr{C}\frac{\sigma_Y^2 \bar s}{n}.
\end{align}
This holds with probability at least $1-e^{-s}$ for all $\eps $ satisfying $\eps \geq D\frac{\sigma_Y\sqrt{s \log m}}{\sqrt n}$.

\section{Stationary and piecewise stationary processes}

Suppose $\mb T$ is a bounded subset of $\mb R^d$ with Lebesgue measure $\mu$ and let $\Delta=\{\mb T_1,\dots, \mb T_N\}$
be a measurable partition of $\mb T.$ 

\begin{assumption}
Suppose that each set $\mb T_j$ is contained 
in a ball of radius $r.$ In what follows, we asume that $r\geq N^{-1/d}.$ 
It is easy to see that there exists a constant $\kappa\geq 2$ depending only on $d$
such that the $\eps$-covering numbers of $\mb T$ with respect to the standard 
Euclidean distance satisfy the condition
\begin{equation}
\label{cover_T}
N(\mb T;\eps)\leq \biggl(\frac{R}{\eps}\biggr)^d\bigvee N, \eps \in (0,R),
\end{equation}
where $R=\kappa N^{1/d}r.$
\end{assumption}

Let $X^{(j)}, j=1,\dots, N$ be centered stationary 
subgaussian processes on $\mb R^d$ and let 
$$
X(t):=\sum_{j=1}^N X^{(j)}(t)I_{{\mb T}_j}(t), t\in \mb T.
$$
Thus, we can view the process $X$ as ``piecewise stationary''.
Let $K_j$ denote the covariance operator and $v_j$ denote the spectral density of $X^{(j)}, j=1,\dots, N$ (we assume that the spectral densities exist).

\begin{assumption}
Suppose that, for some constant $B>0$ and some $p>d/2,$
\begin{equation}
\label{ul_spectrum}
\frac{1}{B}\frac{1}{(1+|t|^2)^{p}}\leq v_j(t)\leq B\frac{1}{(1+|t|^2)^{p}}, t\in {\mb R}^d, j=1,\dots, N.
\end{equation}
\end{assumption}

We use the notations $J_{\lambda},$ $N(\lambda)={\rm card}(J_{\lambda})$ and $\beta(\lambda)=\beta(w,\lambda)$ introduced in Section \ref{weakly_corr}. 
Let $\lambda \in {\mb D}$ be an oracle such that, for each $j\in J_{\lambda}$ we either 
have that $\lambda (t)\geq 0$ for all $t\in \mb T_j$,
or $\lambda (t)\leq 0$ for all $t\in \mb T_j$. 
Thus, $\lambda$ does not change its sign inside the elements of the partition.
Denote ${\mb D}_{\Delta}$ the set of all such oracles in $\mb D.$

Finally, denote $R(\lambda)=\kappa(N(\lambda))^{1/d} r.$ Clearly, $r\leq R(\lambda)\leq R$
(we assume that $N(\lambda)\geq 1$) and condition (\ref{cover_T}) holds for the covering numbers of the set 
$\bigcup\limits_{j\in J_{\lambda}}\mb T_j$ with $R(\lambda)$ in place of $R.$

\begin{theorem}
\label{th:main_stationary}
There exist constants $\fr{C}, \fr{c}$ and $D$ depending only on $B,p,d$ such that the following holds.
For any $s\geq 1$ with 
$\bar s:=s+3\log (\log_2 n+2)+3 \leq \fr{c}\frac{\sqrt{n}}{\log n},$ for all $\eps$ satisfying 
\begin{equation*}
\eps \geq D\frac{\sigma_Y\sqrt{s(\log N\vee \log r)}}{\sqrt n}, 
\end{equation*}
with probability at least $1-e^{-s}$
\begin{align}
\label{oracle1}
\Big\|f_{\hat\lambda_\eps,\hat a_{\eps}}-&f_{\ast}\Big\|_{L_2(\Pi)}^2 
\leq 
\inf_{\lambda \in \mb D_{\Delta}, a\in \mb R}\biggl[\l\|f_{\lambda,a}-f_{\ast}\r\|_{L_2(\Pi)}^2  +
\\
\nonumber
&
 \fr{C}\l(\sigma_Y^{2}
r^{d}\r)^{\frac{2p-d}{2p+d}} L^{d/(2p+d)}
\frac{\|\lambda\|_1^{2d/(2p+d)}N(\lambda)^{(2p-d)/(2p+d)}}{n^{2p/(2p+d)}}
+\fr{C}\frac{\sigma_Y^2 N(\lambda)}{n}+
\\
\nonumber
&
\fr{C} r^{d}(1+r^{-p})^2 \beta^2(\lambda) N(\lambda)\eps^2 
+\fr{C}\frac{\|\lambda\|_1^2 (\log N\vee |\log r|)}{n}\biggr]
+\fr{C}\frac{\sigma_Y^2 \bar s}{n},
\end{align}
where
$
L:=\log N\vee \log n \vee |\log \sigma_Y|\vee |\log r|.
$
\end{theorem}

We will now consider a stationary subgaussian random field $X(t), t\in \mb R^d$
observed in a ball $\mb T=\{t:|t|\leq R\}$ of radius $R\geq 2.$ 

\begin{assumption}
\label{two_sided}
Suppose that $X$ has a spectral density $v(t), t\in \mb R^d$ and, for some constant $B>0$ and some $p>d/2,$
\begin{equation}
\label{ul_spectrum_01}
\frac{1}{B}\frac{1}{(1+|t|^2)^{p}}\leq v(t)\leq B\frac{1}{(1+|t|^2)^{p}}, t\in {\mb R}^d.
\end{equation}
\end{assumption}

Let $\lambda \in \mb D$ be an oracle such that ${\rm supp}(\lambda)$ can be covered 
by a union of $N(\lambda)$ disjoint balls $B(t_1;r), \dots, B(t_{N(\lambda)};r)$ of radius $r\leq R/2.$ Moreover, let us assume that 
the balls in this covering are well separated in the sense that the distance between 
any two distinct balls is at least $2r.$ 
In addition to this, assume that $\lambda$ does not change sign on each of the sets 
$B(t_j;r)\cap {\rm supp}(\lambda), j=1,\dots, N(\lambda)$. 
\index{ay@$\mb D_r$}
Let ${\mb D}_r$ denote the set of all such oracles $\lambda\in \mb D.$

Then, the following theorem holds.

\begin{theorem}
\label{th:main_stationary_A}
There exist constants $\fr{C}, \fr{c}$ and $D$ depending only on $B,p,d$ such that the following holds.
For any $s\geq 1$ with 
$\bar s:=s+3\log (\log_2 n+2)+3 \leq \fr{c}\frac{\sqrt{n}}{\log n},$ for all $\eps$ satisfying 
\begin{equation*}
\eps \geq D\frac{\sigma_Y\sqrt{s\log R}}{\sqrt n},
\end{equation*}
with probability at least $1-e^{-s}$
\begin{align}
\label{oracle2}
\Big\|f_{\hat\lambda_\eps,\hat a_{\eps}}-&f_{\ast}\Big\|_{L_2(\Pi)}^2 
\leq 
\inf_{\lambda \in {\mb D}_r, a\in \mb R}\biggl[\l\|f_{\lambda,a}-f_{\ast}\r\|_{L_2(\Pi)}^2+  
\\
\nonumber
& 
\fr{C}\l(\sigma_Y^{2}
r^{d}\r)^{\frac{2p-d}{2p+d}} L^{d/(2p+d)}
\frac{\|\lambda\|_1^{2d/(2p+d)}N(\lambda)^{(2p-d)/(2p+d)}}{n^{2p/(2p+d)}}+ \\
\nonumber
&
\fr{C}\frac{\sigma_Y^2 N(\lambda)}{n}+
\fr{C} r^{d}(1+r^{-p})^2 N(\lambda)\eps^2 
+\fr{C}\frac{\|\lambda\|_1^2 \log R}{n}\biggr]
+\fr{C}\frac{\sigma_Y^2 \bar s}{n},
\end{align}
where
$
L:=\log n \vee |\log \sigma_Y|\vee \log R \vee |\log r|.
$
\end{theorem}

Note that in Theorems \ref{th:main_stationary} and \ref{th:main_stationary_A} the error 
rate depends on ``sparsity parameter'' $N(\lambda)$ (its meaning is somewhat different 
in these two cases). Moreover, the error rate involves a ``nonparametric term'' 
$O(n^{-{2p/(2p+d)}}).$ Thus, $p>d/2$ plays a role of smoothness parameter in this 
problem.

Often, it is natural to assume that the target $f_\ast(X)$ can be approximated by $f_\Lambda(X),$ where $\Lambda$ is a discrete signed measure 
supported on a ``well-separated'' subset of $\mb T$, so that  
$
f_{\Lambda}(X)=\sum\limits_{j=1}^{N(\Lambda)}\lambda_{j}X(t_j),  
$
where
$
\min\limits_{1\leq i<j\leq N(\Lambda)}|t_i-t_j|\geq 3\delta(\Lambda)>0
$ 
and $\delta(\Lambda)$ is large enough. Such a discrete oracle $\Lambda$ can 
be further approximated by a linear combination of continuous ``spikes'' 
supported in well separated disjoint balls of radius $r>0.$
This can be done for an arbitrary $r<\delta(\Lambda)$ and optimizing 
the bound of Theorem \ref{th:main_stationary_A} with respect to $r$
would lead to a bound with a faster error rate. We will implement this 
in a special (and practically important) case when the design processes $X_j, j=1,\dots, n$ are observed on a discrete grid in $\mb R^d.$
Specifically, assume that 
$\mb T=\m G_N=\l\{t_j=\frac{2\pi}{N}j, \ j\in \{1,\ldots,N\}^d\r\}$ and it is equipped with the counting measure $\mu,$ see section \ref{stationary} for more details. Note that in this case we are in the framework of a standard high-dimensional linear regression with highly correlated design. 
Functions $\lambda$ on $\mb T$ can be identified with vectors in 
${\mb R}^{N^d}$ and we will assume that $\mb D:=R^{N^d}.$
Suppose that Assumption \ref{two_sided} holds and 
let $\lambda$ be an oracle such that $J(\lambda)=\supp(\lambda)\subset \{1,\ldots,N\}^d$, $N(\lambda):=\card(J(\lambda))$, and
$$
\min_{i,j\in J(\lambda),i\ne j} \frac{|i-j|}{N}=:2\delta(\lambda)\geq \frac{1}{N},
$$
where $|i-j|$ stands for the usual Euclidean distance in $\mb R^d$. 
We are mainly interested in the oracles $\lambda$ with ``well-separated'' non-zero elements, meaning that 
$\delta(\lambda)\gg \frac 1 N$.  
In this setting, the following result holds. 

\begin{theorem}
\label{th:main_stationary_B}
There exist constants $\fr{C}, \fr{c}$ and $D$ depending only on $B,p,d$ such that the following holds.
For any $s\geq 1$ with 
$\bar s:=s+3\log (\log_2 n+2)+3 \leq \fr{c}\frac{\sqrt{n}}{\log n},$ let 
\begin{equation*}
\eps = D\frac{\sigma_Y\sqrt{s}}{\sqrt n}.
\end{equation*}
Then with probability at least $1-e^{-s}$
\begin{align*}
\Big\|f_{\hat\lambda_\eps,\hat a_{\eps}}-f_{\ast}\Big\|_{L_2(\Pi)}^2 
&\leq 
\inf_{\lambda \in \mb R^{N^d}, a\in \mb R}\biggl[\l\|f_{\lambda,a}-f_{\ast}\r\|_{L_2(\Pi)}^2+  \fr{C} \delta(\lambda)^{d-2p}\sigma_Y^2 \frac{N(\lambda)s}{n}+
\\
& 
\fr{C}\sigma_Y^{\frac{2p}{p+d}}\big(s L\|\lambda\|_1^2\big)^{\frac{d}{2p+2d}}
\frac{N(\lambda)^{p/(p+d)}}{n^{(2p+d)/(2p+2d)}}+ 
\fr{C}\frac{\|\lambda\|_1^2}{n}\biggr]
+\fr{C}\frac{\sigma_Y^2 \bar s}{n}.
\end{align*}
where $L=\log n \vee \log N \vee|\log  \sigma_Y|.$
\end{theorem}

\section{Proofs of the main results}

\subsection{Preliminaries}
Recall that 
$$
F_n(\lambda,a):=P_n(\ell \bullet f_{\lambda,a})+\varepsilon\|\lambda\|_1,\ \ 
F(\lambda,a):=P(\ell \bullet f_{\lambda,a})+\varepsilon\|\lambda\|_1.
$$

In the proofs of the main results, we will use necessary conditions for the minima in problems (\ref{empirical}), (\ref{true}) that will be stated now. 
Given a convex functional $H: L_1(\mu)\times \mb R\mapsto \mb R,$ define its directional 
derivative at a point $(\lambda,a)\in L_1(\mu)\times \mb R$ in direction 
$u=(u_1,u_2)\in L_1(\mu)\times \mb R$ as 
\begin{align*}
&
DH(\lambda,a)(u):=\lim_{t\downarrow 0}\frac{H((\lambda,a) +t u)-H(\lambda,a)}{t}.
\end{align*}

\begin{proposition}
\label{derivative}
For any $\lambda_1,\lambda_2\in \mb D$ and $a_1,a_2\in \mb R$,  
\begin{align*}
&
DF_n(\lambda_1,a_1)(\lambda_2-\lambda_1, a_2-a_1)=
P_n(\ell'\bullet f_{\lambda_1,a_1})(f_{\lambda_2,a_2}-f_{\lambda_1,a_1})+
\eps\dotp{w_1}{\lambda_2-\lambda_1}
\end{align*}
for some $w_1\in \partial \|\lambda_1\|_1$ that depends on $\lambda_2$.
Similarly,
\begin{align*}
&
DF(\lambda_1,a_1)(\lambda_2-\lambda_1, a_2-a_1)=
P(\ell'\bullet f_{\lambda_1,a_1})(f_{\lambda_2,a_2}-f_{\lambda_1,a_1})+
\eps\dotp{w_1}{\lambda_2-\lambda_1}
\end{align*}
\end{proposition}

\begin{proof}
Treatment of the terms $P_n(\ell\bullet f_{\lambda,a}), P(\ell\bullet f_{\lambda,a})$ is straightforward, so it only remains to examine the $L_1$-penalty term. 
Let $v:=\lambda_2-\lambda_1$. 
Since the function $(0,1)\ni s\mapsto |\lambda_1(t)+s v(t)|$ is convex, we have that 
$$
(0,1)\ni s\mapsto \frac{|\lambda_1(t)+s v(t)|-|\lambda_1(t)|}{s}
$$ 
is nondecreasing. 
Given a decreasing sequence $\l\{s_n\r\}_{n\geq 0}\subset (0,1)$ such that $s_n\to 0$, the sequence of functions 
$$
g_n(t):=\frac{|\lambda_1(t)+s_n v(t)|-|\lambda_1(t)|}{s_n}
$$ 
monotonically converges to 
$$
g(t)=
\begin{cases}
\sign(\lambda_1(t)) v(t), & \lambda_1(t)\ne 0 \\
\sign(v(t)) v(t), & \text{else}.
\end{cases}
$$
Moreover, $g_n(t)$ are integrable, and the monotone convergence theorem implies that
$$
\lim\limits_{n\to\infty} \int\limits_\mb T g_n d\mu=\int\limits_\mb T g d\mu=\int\limits_\mb T w_1u d\mu,
$$
where $|w_1(t)|\leq 1, t\in \mb T$ and 
$w_1(t)=\sign(\lambda_1(t)), \ \lambda_1(t)\ne 0.$  
In particular, $w_1\in \partial \|\lambda_1\|_1$. 

\end{proof}

When $\lambda_1=\hat\lambda_\eps$ (which minimizes $F_n$), the corresponding directional derivatives must be nonnegative for any $\lambda_2\in \mb D$.

\subsection{Proof of Theorem \ref{approx}.}
Let $(\bar\lambda,\bar w,\bar a)$ be a triple that minimizes the right hand side of (\ref{eq:approx}). 
If the infimum is not attained, one can consider the triple for which the right hand side is arbitrarily close to the infimum and follow the argument below. 
 
Since $(\lambda_\eps,a_\eps)$ minimizes $F(\lambda,a)$ over $\mb D\times \mb R$, the directional derivative 
\[
DF(\lambda_\eps,a_\eps)(\bar\lambda-\lambda_\eps,\bar a-a_\eps)
\] 
is nonnegative for any 
$\bar\lambda\in \mb D$, $\bar a\in \mb R$. 
By Proposition \ref{derivative}, this is equivalent to the following:
there exists $w_\eps\in \partial \|\lambda_\eps\|_1$ such that 
\begin{align}
\label{bb2}
P(\ell'\bullet f_{\lambda_{\eps},a_{\eps}})(f_{\lambda_{\eps},a_{\eps}}-f_{\bar \lambda, \bar a})+
\eps\dotp{w_\eps}{\lambda_\eps-\bar\lambda}\leq 0.
\end{align}
Let $\bar w\in \partial \|\bar\lambda\|_1$. 
Since 
$$
(\ell'\bullet f_{\lambda_{\eps},a_{\eps}})(x,y)=2(f_{\lambda_{\eps},a_{\eps}}(x)-y)
$$
and also $Y=f_{\ast}(X)+\xi,$ where ${\mb E}(\xi|X)=0,$
we have 
\begin{align*}
P(\ell'\bullet f_{\lambda_{\eps},a_{\eps}})(f_{\lambda_{\eps},a_{\eps}}-f_{\bar \lambda, \bar a})&=
2{\mathbb E}(f_{\lambda_{\eps},a_{\eps}}(X)-Y)(f_{\lambda_{\eps},a_{\eps}}(X)-f_{\bar \lambda, \bar a}(X))\\
&
=2\langle f_{\lambda_{\eps},a_{\eps}}-f_{\ast}, f_{\lambda_{\eps},a_{\eps}}-f_{\bar \lambda,\bar a}\rangle_{L_2(\Pi)}. 
\end{align*}
Thus, (\ref{bb2}) can be rewritten as
\begin{align}
\label{bb3}
2\langle f_{\lambda_{\eps},a_{\eps}}-f_{\ast}, f_{\lambda_{\eps},a_{\eps}}-f_{\bar \lambda,\bar a}\rangle_{L_2(\Pi)}
+
\eps\dotp{w_\eps-\bar w}{\lambda_\eps-\bar\lambda}
\leq \eps\dotp{\bar w}{\lambda_\eps-\bar\lambda}.
\end{align}
Note that 
\begin{align*}
2\langle f_{\lambda_{\eps},a_{\eps}}-f_{\ast}, f_{\lambda_{\eps},a_{\eps}}-f_{\bar \lambda,\bar a}\rangle_{L_2(\Pi)}
=&
\|f_{\lambda_\eps,a_\eps}-f_{\ast}\|_{L_2(\Pi)}^2+
\|f_{\lambda_\eps,a_\eps}-f_{\bar\lambda,\bar a}\|_{L_2(\Pi)}^2 \\
&
-\|f_{\bar\lambda,\bar a}-f_{\ast}\|_{L_2(\Pi)}^2
\end{align*}
and
\begin{align*}
\dotp{w_\eps-\bar w}{\lambda_\eps-\bar\lambda}\geq \frac 1 2 \int\limits_{\mb T\setminus \mb T_{\bar w}}|\lambda_\eps|d\mu.
\end{align*}
Hence
\begin{align}
\label{eq:base_ineq}
\|f_{\lambda_\eps,a_\eps}-f_{\ast}\|_{L_2(\Pi)}^2+&
\|f_{\lambda_\eps,a_\eps}-f_{\bar\lambda,\bar a}\|_{L_2(\Pi)}^2+\frac{\eps}{2} \int\limits_{\mb T\setminus \mb T_{\bar w}}|\lambda_\eps|d\mu\leq \\
&\nonumber
\leq
\|f_{\bar\lambda,\bar a}-f_{\ast}\|_{L_2(\Pi)}^2+\eps\dotp{\bar w}{\lambda_\eps-\bar\lambda}.
\end{align}
Consider two cases: first, if 
$$
\|f_{\lambda_\eps,a_\eps}-f_{\ast}\|_{L_2(\Pi)}^2+\frac{\eps}{4} \int\limits_{\mb T\setminus \mb T_{\bar w}}|\lambda_\eps|d\mu\leq \|f_{\bar\lambda,\bar a}-f_{\ast}\|_{L_2(\Pi)}^2,$$
then inequality (\ref{eq:approx}) clearly holds. 
Otherwise, (\ref{eq:base_ineq}) implies that 
$$
\int\limits_{\mb T\setminus \mb T_{\bar w}}|\lambda_\eps|d\mu=
\int\limits_{\mb T\setminus \mb T_{\bar w}}|\lambda_\eps-\bar \lambda|d\mu
\leq 4\dotp{\bar w}{\lambda_\eps-\bar\lambda}.
$$ 
Hence, $\lambda_\eps-\bar\lambda\in C_{\bar w}^{(4)}$ 
and
$$
\eps\dotp{\bar w}{\lambda_\eps-\bar\lambda}\leq \eps \l\|f_{\lambda_\eps,a_\eps}-f_{\bar\lambda,\bar a}\r\|_{L_2(\Pi)}\fr a(\bar w)\leq
\frac 14 \eps^2 \fr a^2(\bar w) + \l\|f_{\lambda_\eps,a_\eps}-f_{\bar\lambda,\bar a}\r\|_{L_2(\Pi)}^2,
$$
where we used the definition of $\fr a(\bar w)$ and a simple inequality $ab\leq \frac 14 a^2+b^2$. 
Substituting this bound into (\ref{eq:base_ineq}) gives the result.

\subsection{Proof of Theorem \ref{th:main_AA}.}
Throughout the proof, $C,C_1, c, c_1,$ etc denote absolute constants 
whose values may change from line to line.

\textit{Step 1. Reduction to empirical processes.}
Let $(\bar \lambda, \bar w, \bar a)$ be a triple that minimizes the right hand side 
of bound (\ref{oracle}). 
Clearly, $\bar a= a(\bar \lambda),$ see (\ref{g1}). 
If the infimum is not attained, it is easy to modify 
the argument by considering a triple for which the right hand side is arbitrarily 
close to the infimimum. 
Since $0\in {\mathbb D}$ and, for $\lambda=0,$ one can also 
take $w=0$ and $a={\mb E}Y={\mb E}f_{\ast}(X),$ we have that 
\begin{equation}
\label{barlambda_1}
\|f_{\bar \lambda, \bar a}-f_{\ast}\|_{L_2(\Pi)}^2 \leq 
\|f_{\ast}-\Pi f_{\ast}\|_{L_2(\Pi)}^2={\rm Var}(f_{\ast}(X))
\end{equation} 
and 
\begin{equation}
\label{barlambda_2}
\|\bar \lambda\|_1^2 \leq 
\frac{\|f_{\ast}-\Pi f_{\ast}\|_{L_2(\Pi)}^2 n}{C S^2({\mathbb T})}.
\end{equation}
We will write in what follows $\hat \lambda=\hat \lambda_{\eps}$ and
set $\hat a:=\hat a(\hat\lambda)$. 
Since $(\hat\lambda,\hat a)$ minimizes $F_n(\lambda,a)$ over $\mb D\times \mb R$, the directional derivative 
$DF_n(\hat \lambda, \hat a)(\bar\lambda-\hat\lambda,\bar a-\hat a)$ is nonnegative.
Here and in what follows, we use the ``optimal'' value $\bar a=a(\bar \lambda)$, see (\ref{g1}).  
By Proposition \ref{derivative}, this is equivalent to the following:
there exists $\hat w\in \partial \|\hat\lambda\|_1$ such that 
\begin{align}
\label{r2}
P_n(\ell' \bullet f_{\hat \lambda, \hat a})(f_{\hat \lambda, \hat a}-f_{\bar \lambda, \bar a})+ 
\eps\dotp{\hat w}{\hat\lambda-\bar\lambda}\leq 0.
\end{align}
Since $\bar w\in \partial \|\bar\lambda\|_1,$
(\ref{r2}) can be rewritten as
\begin{align}
\label{r3}
\nonumber
P(\ell' \bullet f_{\hat \lambda, \hat a})(f_{\hat \lambda, \hat a}-f_{\bar \lambda, \bar a})
&+\eps\dotp{\hat w-\bar w}{\hat\lambda-\bar\lambda}\leq \\
& \leq
\eps\dotp{\bar w}{\bar\lambda-\hat\lambda}+
(P-P_n)(\ell' \bullet f_{\hat \lambda, \hat a})\l(f_{\hat \lambda, \hat a}-f_{\bar \lambda, \bar a}\r).
\end{align}
Denote 
\index{az@$\eta(x,y)$}
$
\eta (x,y):= y-f_{\bar \lambda, \bar a}(x).
$
Observe that 
$$
\l(\ell'\bullet f_{\hat \lambda,\hat a}\r)(x,y)= 
-2(y-f_{\hat \lambda, \hat a}(x))=
-2\eta(x,y)+ 2\l(f_{\hat \lambda, \hat a}(x)-f_{\bar \lambda, \bar a}(x)\r)
$$
and, since $Y=f_{\ast}(X)+\xi,$ ${\mathbb E}(\xi|X)=0$, 
\begin{align*}
&
-P\l[\eta(f_{\hat \lambda, \hat a}-f_{\bar \lambda, \bar a})\r]=
-{\mathbb E}\eta(X,Y)\l(f_{\hat \lambda, \hat a}(X)-f_{\bar \lambda, \bar a}(X)\r)=\\
&
=
-{\mathbb E}\l(\xi + f_{\ast}(X)-f_{\bar \lambda,\bar a}(X)\r)\l(f_{\hat \lambda, \hat a}(X)-f_{\bar \lambda, \bar a}(X)\r)=
\dotp{f_{\bar \lambda,\bar a}-f_{\ast}}{f_{\hat \lambda, \hat a}-f_{\bar \lambda, \bar a}}_{L_2(\Pi)}.
\end{align*}
Therefore, we get the following bound:
\begin{align*}
&
2\dotp{f_{\bar \lambda,\bar a}-f_{\ast}}{f_{\hat \lambda, \hat a}-f_{\bar \lambda, \bar a}}_{L_2(\Pi)}+2\l\|f_{\hat \lambda, \hat a}-f_{\bar \lambda, \bar a}\r\|_{L_2(\Pi)}^2 
+\eps\dotp{\hat w-\bar w}{\hat\lambda-\bar\lambda}\leq 
\\
&
\leq
\eps\dotp{\bar w}{\bar\lambda-\hat\lambda}+
2(P_n-P)\eta(f_{\hat \lambda, \hat a}-f_{\bar \lambda, \bar a})
+2(\Pi-\Pi_n)(f_{\hat \lambda, \hat a}-f_{\bar \lambda, \bar a})^2.
\end{align*}
Using the fact that 
\begin{align*} 
2\dotp{f_{\bar \lambda,\bar a}-f_{\ast}}{f_{\hat \lambda, \hat a}-f_{\bar \lambda, \bar a}}_{L_2(\Pi)}=&
\|f_{\hat \lambda, \hat a}-f_{\ast}\|_{L_2(\Pi)}^2
-\|f_{\hat \lambda, \hat a}-f_{\bar \lambda, \bar a}\|_{L_2(\Pi)}^2 \\
&
-\|f_{\bar \lambda, \bar a}-f_{\ast}\|_{L_2(\Pi)}^2,
\end{align*}
it can be rewritten as 
\begin{align}
\label{r3'}
&
\|f_{\hat \lambda, \hat a}-f_{\ast}\|_{L_2(\Pi)}^2
+\|f_{\hat \lambda, \hat a}-f_{\bar \lambda, \bar a}\|_{L_2(\Pi)}^2 
+\eps\dotp{\hat w-\bar w}{\hat\lambda-\bar\lambda}\leq 
\|f_{\bar \lambda, \bar a}-f_{\ast}\|_{L_2(\Pi)}^2+\\ 
&\nonumber
+
\eps\dotp{\bar w}{\bar\lambda-\hat\lambda}+ 
2(P_n-P)\eta(f_{\hat \lambda, \hat a}-f_{\bar \lambda, \bar a})
+2(\Pi-\Pi_n)\l(f_{\hat \lambda, \hat a}-f_{\bar \lambda, \bar a}\r)^2.
\end{align}

The main part of the proof deals with bounding the empirical processes in the right hand side of (\ref{r3'}). 
In what follows, $L$ denotes a subspace of subgaussian space 
${\mathcal L}_X\subset L_2(\mb P)$ (the closed linear span of $\{X(t)-{\mathbb E}X(t):t\in {\mathbb T}\}$). 
\index{ba@$d, \ L, \ \rho(L)$}
Let $d:={\rm dim}(L)$ and let 
\index{bb@$P_L, \ P_{L^{\perp}}$}
$P_L, P_{L^{\perp}}$ be the orthogonal projections onto subspace $L$ and its orthogonal complement
$L^{\perp}\subset {\mathcal L}_X$, and
\begin{align}
\label{eq:rho}
\rho := \rho(L):=\sup_{t\in \mb T_{\bar w}}\l\|P_{L^{\perp}}(X(t)-{\mathbb E}X(t))\r\|_{L_2({\mb P})}.
\end{align}

\textit{Step 2. Bounds for $(P_n-P)\l[\eta\l(f_{\hat \lambda, \hat a}-f_{\bar \lambda, \bar a}\r)\r]$.}
Let
\index{bc@$f_{\lambda}^0(\cdot)$}
$
f_{\lambda}^0(\cdot):=f_{\lambda,a}(\cdot)-\Pi f_{\lambda,a},
$
which clearly satisfies $\Pi f_{\lambda}^0=0$. 
Observe that the following decomposition holds:
\begin{equation}
\label{decomp}
f_{\hat \lambda, \hat a}-f_{\bar \lambda, \bar a}=
f_{\hat \lambda}^0-f_{\bar \lambda}^0+ \bar Y_n- {\mathbb E}Y+ 
\dotp{\hat \lambda-\bar \lambda}{{\mathbb E}X-\bar X_n}
+\dotp{\bar \lambda}{{\mathbb E}X-\bar X_n}.
\end{equation}
This implies 
\begin{align*}
& 
(P_n-P)\eta(f_{\hat \lambda, \hat a}-f_{\bar \lambda, \bar a})
=
\\
& 
(P_n-P)\eta(f_{\hat \lambda}^0-f_{\bar \lambda}^0)+ 
(P_n-P)\eta (\bar Y_n- {\mathbb E}Y)+ 
\\
&
(P_n-P)\eta \cdot (\Pi-\Pi_n)(f_{\hat \lambda}^0-f_{\bar \lambda}^0)+
(P_n-P)\eta \dotp{\bar \lambda}{{\mathbb E}X-\bar X_n}.
\end{align*}
Denote 
\index{bd@$\Lambda(\delta, \Delta, R), \ \alpha_n(\delta;\Delta;R), \ \tau_n (\delta;\Delta;R)$}
\begin{align*}
&
\Lambda(\delta, \Delta, R):= 
\biggl\{\lambda \in {\mathbb D}:  
\big\|f_{\lambda}^0-f_{\bar \lambda}^0\big\|_{L_2(\Pi)}\leq \delta, 
\int\limits_{{\mathbb T}\setminus {\mathbb T}_{\bar w}}|\lambda|d\mu \leq \Delta, \|\lambda\|_1\leq R\biggr\},
\\
&
\alpha_n(\delta;\Delta;R):= \sup_{\lambda\in \Lambda(\delta,\Delta,R)}
\l|(P_n-P)\eta(f_{\lambda}^0-f_{\bar \lambda}^0)\r|, \\
&
\tau_n (\delta;\Delta;R):= \sup_{\lambda\in \Lambda(\delta,\Delta,R)}
\l|(\Pi_n-\Pi)(f_{\lambda}^0-f_{\bar \lambda}^0)\r|.
\end{align*}
Then 
\begin{align}
\label{bou_1}
&
\l|(P_n-P)\eta(f_{\hat \lambda, \hat a}-f_{\bar \lambda, \bar a})\r|\leq 
\alpha_n\biggl(\|f_{\hat\lambda}^0-f_{\bar \lambda}^0\|_{L_2(\Pi)},\int\limits_{{\mathbb T}\setminus {\mathbb T}_{\bar w}}|\hat \lambda|d\mu ,\|\hat \lambda\|_1\biggr) +\\
&\nonumber
+ \l|\bar Y_n- {\mathbb E}Y\r|\cdot\l|(P_n-P)\eta\r|+
\l|(P_n-P)\eta\r|\cdot \tau_n\biggl(\|f_{\hat\lambda}^0-f_{\bar \lambda}^0\|_{L_2(\Pi)},\int\limits_{{\mathbb T}\setminus {\mathbb T}_{\bar w}}|\hat \lambda|d\mu ,\|\hat \lambda\|_1\biggr)+
\\
\nonumber
&
+ \l|(P_n-P)\eta\r|\cdot \l|\dotp{\bar \lambda}{{\mathbb E}X-\bar X_n}\r|.
\end{align}

To provide upper bounds on each of the terms in the right hand side of (\ref{bou_1}) we 
need several lemmas.  

\begin{lemma}
\label{subg}
Let $\l\{Y(t), t\in {\mathbb T}\r\}$ be a centered subgaussian process such that 
$$
{\mathbb E} Y(t)Y(s)= \Cov(X(t),X(s)), t,s \in {\mathbb T}.
$$
There exists a constant $C>0$ such that 
$$
{\mathbb E} \sup_{\lambda \in \Lambda(\delta, \Delta, R)}\l|\dotp{Y}{\lambda -\bar \lambda}\r|
\leq C \Bigl[\delta \sqrt{d}\vee (R+\|\bar\lambda\|_1)\gamma_2(\rho)\vee \Delta S({\mathbb T})\Bigr].
$$
\end{lemma}

\begin{proof}
Denote ${\mathcal L}_Y$ the closed linear span of 
$\l\{Y(t), t\in {\mathbb T}\r\}$, the subgaussian space of the process $Y.$ 
Clearly, the mapping $X(t)-{\mathbb E}X(t)\mapsto Y(t), t\in {\mathbb T}$
can be extended to an $L_2({\mb P})$-isometry of spaces ${\mathcal L}_X, {\mathcal L}_Y\subset L_2(\mb P).$
Let $\tilde L$ be the image of the subspace $L$ under this isometry.  
For all $\lambda \in \Lambda(\delta,\Delta,R)$
and for $u:=\lambda - \bar \lambda,$
\begin{align}
\label{eq:decomposition}
\langle Y, u\rangle  = P_{\tilde L}\langle Y, u\rangle + 
\int\limits_{\mb T_{\bar w}}P_{{\tilde L}^{\perp}}Y(t)u(t)\mu(dt)+
\int\limits_{{\mathbb T}\setminus \mb T_{\bar w}}P_{{\tilde L}^{\perp}}Y(t)u(t)\mu(dt).
\end{align} 
We will use this representation and bound separately the supremum of each term 
to control 
$$
\sup \{\langle Y,\lambda-\bar \lambda\rangle: \lambda\in \Lambda(\delta,\Delta,R)\}.
$$
For the first term, let $\xi_1,\dots, \xi_d$ be an orthonormal basis of $\tilde L.$ 
Note that for $u=\lambda -\bar \lambda, \ \lambda \in \Lambda(\delta,\Delta,R)$, we have 
$
{\mathbb E}\langle Y,u\rangle^2 = \|f_u^0\|_{L_2(\Pi)}^2\leq \delta^2.
$
Therefore, 
\begin{align*}
&{\mathbb E} \sup \l\{|P_{\tilde L}\langle Y,\lambda-\bar \lambda\rangle|: \lambda \in \Lambda(\delta,\Delta,R)\r\}
\leq 
{\mathbb E} \sup \l\{P_{\tilde L}\langle Y,u\rangle: {\mathbb E}\langle Y,u\rangle^2\leq \delta^2\r\}\leq 
\\
&
{\mathbb E}\sup \biggl\{\biggl|\sum_{k=1}^d \alpha_k \xi_k\biggr|: \sum_{k=1}^d \alpha_k^2\leq \delta^2\biggr\}=
\delta {\mathbb E}\biggl(\sum_{k=1}^d \xi_k^2\biggr)^{1/2}\leq \delta \sqrt{d}. 
\end{align*}
For the second term, observe that 
for $u=\lambda-\bar \lambda, \lambda \in \Lambda(\delta,\Delta,R)$ 
$$
\biggl|\int\limits_{\mb T_{\bar w}}P_{{\tilde L}^{\perp}}Y(t)u(t)\mu(dt)\biggr|\leq 
\sup_{t\in \mb T_{\bar w}}|P_{{\tilde L}^{\perp}}Y(t)| \|u\|_1
\leq (R+\|\bar\lambda\|_1)
\sup_{t\in \mb T_{\bar w}}|P_{{\tilde L}^{\perp}}Y(t)|.
$$
Denote $U(t):=P_{{\tilde L}^{\perp}}Y(t), t\in {\mathbb T}.$
Clearly, $U$ is a centered subgaussian process such that 
$$
{\mathbb E}(U(t)-U(s))^2 \leq {\mathbb E}(Y(t)-Y(s))^2, t,s\in {\mathbb T}
$$
and, as a consequence, 
$
\|U (t)-U(s)\|_{\psi_2}\leq c \|Y(t)-Y(s)\|_{\psi_2}
$
with an absolute constant $c>0.$
Moreover, since the spaces ${\mathcal L}_Y, \tilde L$ are isometric images 
of the spaces ${\mathcal L}_X, L,$ we also have that 
$$
\sup_{t\in \mb T_{\bar w}}{\mathbb E}U^2(t)=
\sup_{t\in \mb T_{\bar w}}
{\mathbb E}|P_{{\tilde L}^{\perp}}Y(t)|^2=
\sup_{t\in \mb T_{\bar w}}
{\mathbb E}|P_{{L}^{\perp}}(X-{\mathbb E}X)(t)|^2
=\rho^2,
$$
which implies that 
$
\sup_{t\in \mb T_{\bar w}}\|U(t)\|_{\psi_2}\leq c \rho.
$
Then, it follows from the upper bound on sup-norms of subgaussian processes in terms of 
generic chaining complexities (in particular, (\ref{eq:chaining_bound})) that 
$$ 
{\mathbb E}\sup_{t\in \mb T_{\bar w}}|P_{{\tilde L}^{\perp}}Y(t)|\leq 
C \gamma_2(\rho)
$$
and 
$$
{\mathbb E}\sup_{\lambda\in \Lambda(\delta,\Delta,R)}
\biggl|\int\limits_{\mb T_{\bar w}}P_{{\tilde L}^{\perp}}Y(t)(\lambda-\bar \lambda)(t)\mu(dt)\biggr|
\leq C(R+\|\bar\lambda\|_1)\gamma_2(\rho).
$$
with an absolute constant $C>0.$
Finally, for the third term, note that for $u=\lambda-\bar \lambda,  \lambda \in \Lambda(\delta,\Delta,R)$ 
$$
\biggl|\int\limits_{{\mathbb T}\setminus \mb T_{\bar w}}P_{{\tilde L}^{\perp}}Y(t)u(t)\mu(dt)\biggr|\leq 
\sup_{t\in {\mathbb T}}|P_{{\tilde L}^{\perp}}Y(t)| 
\int\limits_{{\mathbb T}\setminus \mb T_{\bar w}}|u|d\mu 
\leq \Delta
\sup_{t\in {\mathbb T}}|P_{{\tilde L}^{\perp}}Y(t)|,
$$
where we used the fact that 
$\bar \lambda (t)=0, u(t)=\lambda(t)$ for 
$t\in {\mathbb T}\setminus \mb T_{\bar w}$.
By an argument similar to the one used for the second term of (\ref{eq:decomposition}), we get 
$$
{\mathbb E}\sup_{\lambda\in \Lambda(\delta,\Delta,R)}
\biggl|\int_{{\mathbb T}\setminus \mb T_{\bar w}}P_{{\tilde L}^{\perp}}Y(t)(\lambda-\bar \lambda)(t)\mu(dt)\biggr|\leq 
C\Delta S({\mathbb T}),  
$$
which implies the bound of the lemma.

\end{proof}

\begin{lemma}
\label{expect_001}
There exists a constant $C>0$ such that, for all $\delta>0, \Delta>0, R>0,$
\begin{align*}
{\mathbb E}\alpha_n (\delta,\Delta, R)\leq 
C&\l(\|f_{\bar \lambda,\bar a}-f_{\ast}\|_{L_2(\Pi)}\vee \|\xi\|_{\psi_2}\r)
\biggl[\delta \sqrt{\frac{d}{n}}\bigvee (R\vee\|\bar\lambda\|_1)\frac{\gamma_2(\rho)}{\sqrt{n}}
\bigvee \Delta \frac{S({\mathbb T})}{\sqrt{n}}\biggr]
\\
& \bigvee C\biggl[\delta\sqrt{\frac{d}{n}}\bigvee (R\vee \|\bar\lambda\|_1)\frac{\gamma_2(\rho)}{\sqrt{n}}
\bigvee \Delta \frac{S({\mathbb T})}{\sqrt{n}}\biggr]^2.
\end{align*}
\end{lemma}

\begin{proof}
\index{bda@${\mathcal F}(\delta,\Delta,R)$}
Let $\m F:={\mathcal F}(\delta,\Delta,R):=\{f_{\lambda}^0-f_{\bar \lambda}^0:\lambda \in \Lambda(\delta,\Delta,R)\}.$ 
We use a recent result by S. Mendelson (see Theorem \ref{th:mendelson1}, statement (i)) which implies that, for all $\delta, \Delta, R,$
\begin{equation}
\label{mendelson_A}
{\mathbb E}\alpha_n(\delta, \Delta, R)
\leq C\biggl[\frac{\|\eta\|_{\psi_2} \gamma_2({\mathcal F};\psi_2)}{\sqrt{n}}
\bigvee \frac{\gamma_2^2({\mathcal F};\psi_2)}{n}\biggr],
\end{equation}
for an absolute constant $C>0$.
Since $\Bigl\{f(X): f\in {\mathcal F}(\delta,\Delta,R)\Bigr\}\subset {\mathcal L}$ and ${\mathcal L}$ is a subgaussian space, we have that $\|f\|_{\psi_2}\leq c_1\|f\|_{L_2(\Pi)},
f\in {\mathcal F}(\delta,\Delta,R)$ for some constant $c_1.$ Therefore,   
$$
\gamma_2\l({\mathcal F};\|\cdot\|_{\psi_2}\r)\leq 
c_1
\gamma_2\l({\mathcal F};L_2(\Pi)\r).
$$
Let $G(t), t\in {\mathbb T}$ be a centered Gaussian process with the same covariance as 
the process $\l\{X(t), t\in {\mathbb T}\r\}.$ Then the stochastic processes $u\mapsto \langle G,u\rangle$ 
has the same covariance as $u\mapsto \dotp{X-{\mathbb E}X}{u}=f_u^{(0)}(X),$ that is,
$
{\mathbb E}\dotp{G}{u_1}\dotp{G}{u_2}= \dotp{f_{u_1}^0}{f_{u_2}^0}_{L_2(\Pi)}.
$
By Talagrand's generic chaining theorem for Gaussian processes (Theorem 2.1.1 in  \cite{talagrand2005generic}), this implies
$$
\gamma_2\l({\mathcal F},L_2(\Pi)\r)
\leq c_2 {\mathbb E}\sup \Bigl\{ \l|\dotp{G}{\lambda-\bar \lambda}\r|: \lambda \in \Lambda(\delta,\Delta,R)\Bigr\}. 
$$
Using Lemma \ref{subg}, we get the following bound on $\gamma_2\l({\mathcal F},\|\cdot\|_{\psi_2}\r):$
\begin{equation}
\label{gamma2_A}
\gamma_2\l({\mathcal F},\|\cdot\|_{\psi_2}\r)\leq 
C\Bigl[\delta\sqrt{d}\vee (R+\|\bar\lambda\|_1)\gamma_2(\rho)\vee \Delta S({\mathbb T})\Bigr].
\end{equation}
Next, note that 
$$
\eta (X,Y)=Y-f_{\bar \lambda, \bar a}(X)=f_{\ast}(X)+\xi - f_{\bar \lambda, \bar a}(X).
$$
We also have ${\mb E}f_{\ast}(X)={\mb E}Y$ and 
$$
{\mb E}f_{\bar \lambda, \bar a}(X) = {\mb E}Y-\langle \bar \lambda, {\mb E}X\rangle
+{\mb E}\langle \bar \lambda, X\rangle= {\mb E} Y
$$
which implies that ${\mathbb E}\eta(X,Y)=0.$ The random variable $f_{\ast}(X)-f_{\bar \lambda,\bar a}(X)$ belongs to the subgaussian space ${\mathcal L},$ implying that
\begin{align}
\label{eq:eta}
\|\eta\|_{\psi_2} = 
\|f_{\ast}(X)-f_{\bar \lambda,\bar a}(X)+\xi\|_{\psi_2}
&\leq \|f_{\ast}(X)-f_{\bar \lambda,\bar a}(X)\|_{\psi_2}
+\|\xi\|_{\psi_2}
 \\
& \nonumber
\leq 
c \|f_{\bar \lambda,\bar a}-f_{\ast}\|_{L_2(\Pi)}+\|\xi\|_{\psi_2}
\end{align}
with an absolute constant $c>0.$
In view of (\ref{mendelson_A}), (\ref{gamma2_A}) and (\ref{eq:eta}) easily imply the bound of the lemma. 

\end{proof}

Our next goal is to derive an upper bound on   
$\alpha_n (\delta,\Delta, R)$ that holds uniformly 
in $\delta \in [\delta_-,\delta_+],$ $\Delta\in [\Delta_-, \Delta_+],$ $R\in [R_-,R_+]$ for some $\delta_-<\delta_+, \Delta_-<\Delta_+,R_-<R_+$ to be 
determined later. Let 
$$
J_1:=\biggl[\log_2 \l(\frac{\delta_+}{\delta_-}\r)\biggr]+1, 
J_2:=\biggl[\log_2 \l(\frac{\Delta_+}{\Delta_-}\r)\biggr]+1, 
J_3:=\biggl[\log_2 \l(\frac{R_+}{R_-}\r)\biggr]+1
$$
and, given $s>0,$ let 
$$
\bar s:=s+\log ((J_1+1)(J_2+1)(J_3+1)).
$$
Finally, denote 
\index{be@$\nu_n(\delta,\Delta,R)$}
\begin{align}
\label{eq:nu}
\nu_n(\delta,\Delta,R):=
\inf_{L\subset {\m L}}\biggl[\delta \sqrt{\frac{\dim(L)}{n}}\bigvee (R\vee\|\bar\lambda\|_1)\frac{\gamma_2(\rho(L))}{\sqrt{n}}
\bigvee \Delta \frac{S({\mathbb T})}{\sqrt{n}}\biggr],
\end{align}
where the infimum is taken over all finite dimensional subspaces $L\subset \m L$ and $\rho(L)$ is defined in (\ref{eq:rho}).

\begin{lemma} 
\label{adamczak_001}
There exists a constant $C>0$ with the following property.
With probability at least $1-e^{-s}$, the following inequality holds uniformly
for all $\delta \in [\delta_-,\delta_+],$ $\Delta\in [\Delta_-, \Delta_+],$ $R\in [R_-,R_+]$:
\begin{align*}
\alpha_n (\delta,\Delta, R)\leq 
C\l(\|f_{\bar \lambda,\bar a}-f_{\ast}\|_{L_2(\Pi)}\vee \|\xi\|_{\psi_2}\r)
\biggl[\delta \sqrt{\frac{\bar s}{n}}\bigvee \nu_n(\delta,\Delta,R)\biggr]
\bigvee
C\nu_n^2(\delta,\Delta,R).
\end{align*}
\end{lemma}

\begin{proof}
First, we use Adamczak's version of Talagrand's inequality (\ref{adamczak}) to deduce an exponential bound on 
$\alpha_n (\delta,\Delta, R)$ from the bound 
on ${\mathbb E}\alpha_n (\delta,\Delta, R)$ (for fixed $\delta,\Delta, R>0$).  
To this end, observe that, by the properties of Orlicz norms and subgaussian spaces,  
\begin{align*}
\big\|\eta (f_{\lambda}^0-f_{\bar \lambda}^0)\big\|_{L_2(P)}&
\leq 
\|\eta\|_{L_4(P)}\l\|f_{\lambda}^0-f_{\bar \lambda}^0\r\|_{L_4(\Pi)}\leq 
c_1 \|\eta\|_{\psi_2}\l\|f_{\lambda}^0-f_{\bar \lambda}^0\r\|_{L_4(\Pi)}\leq \\
&
\leq 
c_2 \l(\|f_{\bar \lambda,\bar a}-f_{\ast}\|_{L_2(\Pi)}+ \|\xi\|_{\psi_2}\r)
\l\|f_{\lambda}^0-f_{\bar \lambda}^0\r\|_{L_2(\Pi)},
\end{align*}
where we used (\ref{eq:eta}) to bound $\|\eta\|_{\psi_2}$. 
For all $\lambda \in \Lambda(\delta,\Delta,R),$ this implies 
$$
\|\eta (f_{\lambda}^0-f_{\bar \lambda}^0)\|_{L_2(P)}\leq 
c\delta \l(\|f_{\bar \lambda,\bar a}-f_{\ast}\|_{L_2(\Pi)}+ \|\xi\|_{\psi_2}\r).
$$
Using (\ref{i3}), we will also estimate the envelope of the class ${\mathcal F}(\delta,\Delta,R)$
as follows:
\begin{align*}
\Bigl\|\sup_{\lambda \in \Lambda(\delta,\Delta,R)}\Big|\eta(X,Y) (f_{\lambda}^0(X)-&f_{\bar \lambda}^0(X))\Big|
\Bigr\|_{\psi_1}
\\
&
\leq c \|\eta(X,Y)\|_{\psi_2}\Bigl\|
\sup_{\lambda \in \Lambda(\delta,\Delta,R)}\l|f_{\lambda}^0(X)-f_{\bar \lambda}^0(X)\r|\Bigr\|_{\psi_2}.
\end{align*}
Recall that $L$ is a subspace of subgaussian space ${\mathcal L}$ with ${\rm dim}(L)=d$
and 
$$
\rho = \sup_{t\in \mb T_{\bar w}}\l\|P_{L^{\perp}}(X(t)-{\mathbb E}X(t))\r\|_{L_2({\mb P})}. 
$$
Let $\zeta_1,\dots, \zeta_d$ be an orthonormal basis of $L\subset L_2({\mb P}).$
For $u=\lambda-\bar \lambda,$ the following decomposition holds:
\begin{align*}
&
f_{\lambda}^0(X)-f_{\bar \lambda}^0(X)=
\langle u, X-{\mathbb E}X\rangle=
P_L \langle u, X-{\mathbb E}X\rangle  \\
&
+\int\limits_{{\mathbb T}_{\bar w}}P_{L^{\perp}}(X-{\mathbb E}X)(t)u(t)\mu(dt)
+
\int\limits_{{\mathbb T}\setminus{\mathbb T}_{\bar w}}P_{L^{\perp}}(X-{\mathbb E}X)(t)u(t)\mu(dt).
\end{align*}
We have 
\begin{align*}
\Bigl\|\sup_{\lambda \in \Lambda(\delta,\Delta,R)}\Big|P_L \langle \lambda-\bar \lambda, X-{\mathbb E}&X\rangle\Big|\Bigr\|_{\psi_2}
\leq 
\biggl\|\sup \biggl\{\biggl|\sum_{k=1}^d \alpha_k \zeta_k\biggr|:\sum_{k=1}^d\alpha_k^2\leq \delta^2 \biggr\}\biggr\|_{\psi_2}\\
&
\leq
\delta \sqrt{d}\biggl\|\biggl(\frac{1}{d}\sum_{k=1}^d \zeta_k^2\biggr)^{1/2}\biggr\|_{\psi_2}\leq 
\delta \sqrt{d}\biggl\|\frac{1}{d}\sum_{k=1}^d \zeta_k^2\biggr\|_{\psi_1}^{1/2} \\
&
\leq \delta \sqrt{d}\max_{1\leq k\leq d}\|\zeta_k^2\|_{\psi_1}^{1/2}\leq 
\delta \sqrt{d}\max_{1\leq k\leq d}\|\zeta_k\|_{\psi_2}\leq
C \delta \sqrt{d}.
\end{align*}
We also easily get 
\begin{align*}
\biggl\|
\sup_{\lambda \in \Lambda(\delta,\Delta,R)}
\int\limits_{{\mathbb T}_{\bar w}}&P_{L^{\perp}}(X-{\mathbb E}X)(t)(\lambda-\bar \lambda)(t)\mu(dt)\biggr\|_{\psi_2}\\
&
\leq
(R+\|\bar\lambda\|_1) \l\|\sup_{t\in {\mathbb T}_{\bar w}}P_{L^\perp}(X-{\mathbb E}X)(t)\r\|_{\psi_2} 
\leq C(R+\|\bar\lambda\|_1) \gamma_2(\rho)
\end{align*}
and 
\begin{align*}
\biggl\|
\sup_{\lambda \in \Lambda(\delta,\Delta,R)}
\int\limits_{{\mathbb T}\setminus {\mathbb T}_{\bar w}}P_{L^{\perp}}(X-{\mathbb E}X)(t)(\lambda-\bar \lambda)(t)\mu(dt)\biggr\|_{\psi_2}
&\leq 
\Delta \l\|\sup_{t\in {\mathbb T}}P_{L^\perp}(X-{\mathbb E}X)(t)\r\|_{\psi_2}\\\
&
\leq C\cdot\Delta S({\mathbb T}).
\end{align*}
It implies that
\begin{equation} 
\label{psipsi}
\l\|\sup_{\lambda \in \Lambda(\delta,\Delta,R)}|f_{\lambda}^0(X)-f_{\bar \lambda}^0(X)|\r\|_{\psi_2}\leq 
C\Bigl[\delta\sqrt{d}+(R+\|\bar\lambda\|_1)\gamma_2(\rho)+\Delta S({\mathbb T})\Bigr].
\end{equation}
Thus,
\begin{align*}
\Bigl\|\sup_{\lambda \in \Lambda(\delta,\Delta,R)}&|\eta(X,Y) (f_{\lambda}^0(X)-f_{\bar \lambda}^0(X))|
\Bigr\|_{\psi_1}
\leq \\
&\leq 
C\l(\|f_{\bar \lambda,\bar a}-f_{\ast}\|_{L_2(\Pi)}+ \|\xi\|_{\psi_2}\r)
\Bigl[\delta\sqrt{d}+(R+\|\bar\lambda\|_1)\gamma_2(\rho)+\Delta S({\mathbb T})\Bigr].
\end{align*}
It follows from Adamczak's bound (\ref{adamczak}) and the second statement of Proposition \ref{psi_2} that, with probabiltiy at least $1-e^{-s},$
\begin{align*}
&
\alpha_n(\delta,\Delta,R)\leq 
C\biggl[{\mathbb E}\alpha_n(\delta,\Delta,R)
+ (\|f_{\bar \lambda,\bar a}-f_{\ast}\|_{L_2(\Pi)}+ \|\xi\|_{\psi_2})\delta \sqrt{\frac{s}{n}}+\\
&
+(\|f_{\bar \lambda,\bar a}-f_{\ast}\|_{L_2(\Pi)}+ \|\xi\|_{\psi_2})
\Bigl[\delta\sqrt{d}+(R+\|\bar\lambda\|_1)\gamma_2(\rho)+\Delta S({\mathbb T})\Bigr]
\frac{s \log n}{n}
\biggr].
\end{align*}
Combining this with the bound of Lemma \ref{expect_001},
taking the infimum of the right hand side with respect to $L\subset {\mathcal L}$ 
and recalling that, according to our assumptions, $\frac{s\log n}{\sqrt{n}}$ is bounded by an absolute constant, 
we derive the following inequality: 
\begin{align}
\label{alpha_bd}
& \alpha_n (\delta,\Delta, R)\leq \beta_n(\delta, \Delta, R;s) :=
\\
\nonumber
& C(\|f_{\bar \lambda,\bar a}-f_{\ast}\|_{L_2(\Pi)}\vee \|\xi\|_{\psi_2})
\biggl[\delta \sqrt{\frac{s}{n}}\bigvee 
\nu_n(\delta,\Delta,R)\biggr]\bigvee C\nu_n^2(\delta,\Delta,R)
\end{align}
that holds with probability at least $1-e^{-s}.$

We still need to make the last bound uniform in 
$\delta \in [\delta_-,\delta_+],$ $\Delta\in [\Delta_-, \Delta_+],$ $R\in [R_-,R_+].$  
To this end, define $\delta_{j_1}:=\delta_+ 2^{-j_1}, \Delta_{j_2}:=\Delta_+ 2^{-j_2}$ and $R_{j_3}:=R_+ 2^{-j_3}$ for $j_1=0,1,\dots, J_1,$ $j_2=0,1,\dots, J_2,$ and 
$j_3=0,1,\dots, J_3.$ 
Using bound (\ref{alpha_bd}) for each $\delta_{j_1}, \Delta_{j_2}, R_{j_3}$ with $s$ replaced by 
$
\bar s:=s+\log ((J_1+1)(J_2+1)(J_3+1))
$
and applying then the union bound, we get that with probability at least $1-e^{-s}$
$
\alpha_n (\delta_{j_1},\Delta_{j_2}, R_{j_3})\leq \beta_n (\delta_{j_1},\Delta_{j_2}, R_{j_3};\bar s)
$  
for all $j_k=0,\dots, J_k, k=1,2,3.$ By monotonicity of the functions $\alpha_n, \beta_n$
in their variables this easily implies that with the same probability 
\begin{equation}
\nonumber
\alpha_n (\delta,\Delta, R) \leq 
C\l(\|f_{\bar \lambda,\bar a}-f_{\ast}\|_{L_2(\Pi)}\vee \|\xi\|_{\psi_2}\r)
\biggl[\delta \sqrt{\frac{\bar s}{n}}\bigvee 
\nu(\delta,\Delta,R)\biggr]
\bigvee
C\nu_n^2(\delta,\Delta,R).
\end{equation}
for all $\delta \in [\delta_-,\delta_+],$ $\Delta\in [\Delta_-, \Delta_+],$ $R\in [R_-,R_+]$
and for a large enough constant $C>0.$

\end{proof}

Bounding the last three terms in the right hand side of (\ref{bou_1}) is easier.
Since $\eta(X,Y)$ is a subgaussian random variable (its mean is equal to zero
and its $\psi_2$-norm is finite) and (\ref{eq:eta}) holds, 
we have the following tail bound:
\begin{equation}
\label{eta_bound}
|(P_n-P)\eta|=\biggl|n^{-1}\sum_{j=1}^n \eta(X_j,Y_j)\biggr|
\leq C(\|f_{\bar \lambda,\bar a}-f_{\ast}\|_{L_2(\Pi)}\vee \|\xi\|_{\psi_2})\sqrt{\frac{s}{n}}
\end{equation}
with probability at least $1-e^{-s}$ and with some constant $C>0.$
Moreover, using the representation 
$
Y-{\mb E}Y=f_{\ast}(X)-{\mb E}f_{\ast}(X)+\xi
$
and the assumption that $f_{\ast}(X)-{\mb E}f_{\ast}(X)\in {\mathcal L},$
we get 
\begin{align*}
\|Y-{\mathbb E}Y\|_{\psi_2} &\leq 
\|f_{\ast}(X)-{\mb E}f_{\ast}(X)\|_{\psi_2}+\|\xi\|_{\psi_2} \\
&
\leq 
c\|f_{\ast}(X)-{\mb E}f_{\ast}(X)\|_{L_2(\Pi)}+\|\xi\|_{\psi_2}\\
&
= c\|f_{\ast}-\Pi f_{\ast}\|_{L_2(\Pi)}+\|\xi\|_{\psi_2}.
\end{align*}
Since $Y-{\mathbb E}Y$ is subgaussian, it is easy to deduce that 
\begin{equation}
\label{bary}
|\bar Y_n -{\mathbb E}Y|=\biggl|n^{-1}\sum_{j=1}^n (Y_j-{\mathbb E}Y)\biggr|
\leq C(\|f_{\ast}-\Pi f_{\ast}\|_{L_2(\Pi)}\vee \|\xi\|_{\psi_2})\sqrt{\frac{s}{n}}
\end{equation}
with probability at least $1-e^{-s}$ and with some constant $C>0.$
Therefore,
\begin{equation}
\label{duo}
\l|\bar Y_n- {\mathbb E}Y\r|\cdot\l|(P_n-P)\eta\r|\leq C \l(\|f_{\ast}-\Pi f_{\ast}\|_{L_2(\Pi)}^2\vee \|f_{\bar \lambda,\bar a}-f_{\ast}\|_{L_2(\Pi)}^2\vee \|\xi\|_{\psi_2}^2\r)\frac{s}{n}
\end{equation}
with probability at least $1-2 e^{-s}.$

Since $\dotp{\bar\lambda}{X_j-\mb EX_j}$ are i.i.d. subgaussian random variables, their average $\dotp{\bar\lambda}{\bar X_n-\mb EX}$ is also subgaussian. 
This easily yields the bound 
\begin{align}
\label{barx}
\l|\dotp{\bar\lambda}{\bar X_n-\mb EX}\r| &
\leq C\|f^0_{\bar\lambda}\|_{L_2(\Pi)}\sqrt{\frac s n}\leq 
C\l(\|f_{\ast}-\Pi f_{\ast}\|_{L_2(\Pi)}\vee\|f_{\bar\lambda,\bar a}-f_\ast\|_{L_2(\Pi)}\r)\sqrt{\frac s n}
\end{align}
that holds with probability at least $1-e^{-s}$ and with some $C>0.$
Therefore, with probability at least $1-2 e^{-s}$
\begin{equation}
\label{uno}
\l|\dotp{\bar\lambda}{\bar X_n-\mb EX}\r|\l|(P_n-P)\eta\r|
\leq  C\l(\|f_{\ast}-\Pi f_{\ast}\|_{L_2(\Pi)}^2\vee \|f_{\bar \lambda,\bar a}-f_{\ast}\|_{L_2(\Pi)}^2\vee \|\xi\|_{\psi_2}^2\r)\frac{s}{n}.
\end{equation}

The proof of the next lemma is a simplified version of the proofs of Lemmas 
\ref{expect_001}, \ref{adamczak_001}. Together with (\ref{eta_bound}) it will
be used to control the term 
$$
\l|(P_n-P)\eta\r|\cdot\tau_n\biggl(\|f_{\lambda}^0-f_{\bar \lambda}^0\|_{L_2(\Pi)},\int\limits_{{\mathbb T}\setminus {\mathbb T}_{\bar w}}|\hat \lambda|d\mu ,\|\hat \lambda\|_1\biggr)
$$
in the right hand side of (\ref{bou_1}).

\begin{lemma}
\label{linear_001}
There exists a constant $C>0$ such that the following holds. 
Under the notations of Lemma \ref{adamczak_001},
with probability at least $1-e^{-s}$ and with $\bar s:=s+\log ((J_1+1)(J_2+1)(J_3+1))$
satisfying the condition $\bar s \sqrt{\log n}\leq \sqrt{n}$,
\begin{equation}
\nonumber
\tau_n(\delta,\Delta,R)=\sup\Bigl\{\l|(\Pi_n-\Pi)(f_{\lambda}^0-f_{\bar \lambda}^0)\r|:\lambda \in \Lambda(\delta, \Delta, R)\Bigr\}
\leq 
C
\biggl[\delta \sqrt{\frac{\bar s}{n}}\bigvee \nu_n(\delta,\Delta,R)\biggr]
\end{equation}
uniformly for all $\delta \in [\delta_-,\delta_+],$ $\Delta\in [\Delta_-, \Delta_+],$ $R\in [R_-,R_+].$ 
\end{lemma}

\textit{Step 3. Bounds for $(\Pi-\Pi_n)\l(f_{\hat \lambda, \hat a}-f_{\bar \lambda, \bar a}\r)^2$.} 

We will need the following representation (that is a consequence of (\ref{decomp})):
\begin{align}
\nonumber
(\Pi-\Pi_n)&(f_{\hat \lambda, \hat a}-f_{\bar \lambda, \bar a})^2=
(\Pi-\Pi_n)(f_{\hat \lambda}^0-f_{\bar \lambda}^0)^2+
2(\Pi-\Pi_n)(f_{\hat \lambda}^0-f_{\bar \lambda}^0)
(\bar Y_n-{\mathbb E}Y)+ \\
&
\nonumber
+2(\Pi-\Pi_n)(f_{\hat \lambda}^0-f_{\bar \lambda}^0)\dotp{\bar \lambda}{{\mathbb E}X-\bar X_n}
+
2\Bigl[(\Pi-\Pi_n)(f_{\hat \lambda}^0-f_{\bar \lambda}^0)\Bigr]^2
\\
&
:=(\Pi-\Pi_n)(f_{\hat \lambda}^0-f_{\bar \lambda}^0)^2+\zeta_n(\hat \lambda).
\end{align}
Using bounds (\ref{bary}), (\ref{barx}) and Lemma \ref{linear_001}, it yields
that with probability at least $1-3 e^{-s}$ for the same $\delta, \Delta, R$
\begin{align}
\label{trio}
\sup&
\Bigl\{\l|\zeta_n(\lambda)\r|:\lambda \in \Lambda(\delta, \Delta, R)\Bigr\}
\leq C\biggl[\delta \sqrt{\frac{\bar s}{n}}\bigvee \nu_n(\delta,\Delta,R)\biggr]^2\bigvee
\\
&\nonumber
\bigvee C\l(\|f_{\ast}-\Pi f_{\ast}\|_{L_2(\Pi)}\vee \|f_{\bar\lambda,\bar a}-f_\ast\|_{L_2(\Pi)}\vee\|\xi\|_{\psi_2}\r)\sqrt{\frac{s}{n}}
\biggl[\delta \sqrt{\frac{\bar s}{n}}\bigvee \nu_n(\delta,\Delta,R)\biggr].
\end{align}
Next, we have to estimate 
\index{bf@$\psi_n(\delta,\Delta,R)$}
$$
\psi_n(\delta,\Delta,R):= \sup_{\lambda \in \Lambda(\delta, \Delta, R)}
\Bigl|(\Pi_n-\Pi)(f_{\lambda}^0-f_{\bar \lambda}^0)^2\Bigr|.
$$

\begin{lemma}
\label{squared_class}
There exists a constant $C>0$ such that the following holds. 
Under the notations of Lemma \ref{adamczak_001},
with probability at least $1-e^{-s}$
\begin{equation}
\nonumber
\psi_n (\delta,\Delta, R)\leq 
C \delta
\biggl[\delta \sqrt{\frac{\bar s}{n}}\vee \nu_n(\delta,\Delta,R)\biggr]
\bigvee C \nu_n^2(\delta,\Delta,R)
\end{equation}
uniformly for all $\delta \in [\delta_-,\delta_+],$ $\Delta\in [\Delta_-, \Delta_+],$ $R\in [R_-,R_+].$
\end{lemma}

\begin{proof}

The proof is based on the inequality due to S. Dirksen and W. Bednorz (see Theorem \ref{th:dirksen1} in the appendix). 
To this end, we need to estimate several quantities appearing in that bound. 
First, note that, since $\Big\{f(X): \ f\in \m F(\delta,\Delta,R)\Big\}$ is a subset of a subgaussian space,
$$
\sup_{f\in {\mathcal F}(\delta,\Delta,R)}\|f\|_{\psi_2}\leq 
c\sup_{f\in {\mathcal F}(\delta,\Delta,R)}\|f\|_{L_2(\Pi)}\leq 
c\delta.
$$
Together with the bound (\ref{gamma2_A}) on $\gamma_2\l({\mathcal F};\psi_2\r),$ 
Theorem \ref{th:dirksen1} implies that with probability $\geq 1-e^{-s}$,
\begin{align*}
\psi_n(\delta,\Delta,R)\leq &
C \delta \biggl[\delta \sqrt{\frac{d}{n}}\bigvee (R\vee\|\bar\lambda\|_1)\frac{\gamma_2(\rho)}{\sqrt{n}}
\bigvee \Delta \frac{S({\mathbb T})}{\sqrt{n}}\biggr] \bigvee \\
&\bigvee
C\biggl[\delta\sqrt{\frac{d}{n}}\bigvee (R\vee\|\bar\lambda\|_1)\frac{\gamma_2(\rho)}{\sqrt{n}}
\bigvee \Delta \frac{S({\mathbb T})}{\sqrt{n}}\biggr]^2\bigvee C\delta^2\bigg[\sqrt{\frac s n}\vee \frac s n\bigg].
\end{align*}
It remains to combine the discretization argument as in the proof of Lemma \ref{adamczak_001} with 
an application of the union bound to get an estimate for $\psi_n (\delta,\Delta, R)$
that holds uniformly in $\delta, \Delta, R$ with a high probability. 
As a result, we get that 
\begin{align*}
\psi_n (\delta,\Delta, R)\leq 
C \delta
\biggl[\delta \sqrt{\frac{\bar s}{n}}\vee \nu_n(\delta,\Delta,R)\biggr]
\bigvee
C \nu_n^2(\delta,\Delta,R)
\end{align*}
with probability at least $1-e^{-s}$ for all $\delta \in [\delta_-,\delta_+],$ $\Delta\in [\Delta_-, \Delta_+],$ $R\in [R_-,R_+]$ uniformly, 
and for a large enough constant $C>0.$
\end{proof}

\textit{Step 4. Upper bound on $\|\hat \lambda\|_1.$}

\begin{lemma}
\label{norms}
There exist constants $C,D>0$ such that the following holds.
For all $s\geq 1$ and $\eps$ satisfying the assumptions $s\log n \leq \sqrt{n}$ and   
$
\eps\geq D\|\xi\|_{\psi_2}S({\mathbb T})\sqrt{\frac{s}{n}},
$
with probability at least $1-5e^{-s}$, 
$$
\|\hat \lambda\|_1\leq C\l(\frac{q(\eps)}{\eps}+\frac{\sigma_Y^2 s}{n \eps}\r).
$$
\end{lemma}

\begin{proof}
By the definition of $\hat \lambda,$ for all $\lambda\in \mb D, a\in \mb R$
\begin{align}
\label{def_min}
&
P_n(\ell \bullet f_{\hat \lambda, \hat a})+\eps\|\hat\lambda\|_1\leq 
P_n(\ell \bullet f_{\lambda,a})+\eps\|\lambda\|_1.
\end{align}
We will take $a=a(\lambda)=\mb EY-\langle \lambda, \mb E X\rangle$ everywhere below.
Let $\xi(x,y)=y-f_{\ast}(x)$ (then, $\xi_j=\xi(X_j,Y_j)$). Since 
$$
\ell \bullet f_{\hat \lambda,\hat a}-\ell \bullet f_{\lambda,a}=
(f_{\hat \lambda,\hat a}+f_{\lambda,a}-2f_{\ast}-2\xi)(f_{\hat \lambda,\hat a}-f_{\lambda,a}),
$$
it is easy to conclude that 
$$
P_n(\ell \bullet f_{\hat \lambda,\hat a})-P_n(\ell \bullet f_{\lambda,a})=
\|f_{\hat \lambda, \hat a}-f_{\ast}\|_{L_2(\Pi_n)}^2-
\|f_{\lambda, a}-f_{\ast}\|_{L_2(\Pi_n)}^2 -2 P_n \xi  (f_{\hat \lambda,\hat a}-f_{\lambda,a}).
$$
Thus, (\ref{def_min}) implies that
\begin{align}
\label{bound_zz}
&\|f_{\hat \lambda, \hat a}-f_{\ast}\|_{L_2(\Pi_n)}^2+ \eps\|\hat\lambda\|_1\leq
\\
\nonumber
&
\|f_{\lambda, a}-f_{\ast}\|_{L_2(\Pi)}^2+ 
(\Pi_n-\Pi)(f_{\lambda,a}-f_{\ast})^2 +2 P_n\l[ \xi  (f_{\hat \lambda,\hat a}-f_{\lambda,a})\r]
+\eps\|\lambda\|_1.
\end{align}
Using Bernstein's inequality for the random variable with finite $\|\cdot\|_{\psi_1}$-norm (see \cite{Koltchinskii2011Oracle-inequali00}, section A.2) we get that with probability at least $1-e^{-s}$, for $s\leq n$
\begin{equation}
\label{bound_zzz}
|(\Pi_n-\Pi)(f_{\lambda,a}-f_{\ast})^2|\leq C_1\|f_{\lambda,a}-f_{\ast}\|^2_{L_2(\Pi)}\l[\sqrt{\frac s n}\bigvee \frac s n\r]\leq C_1\|f_{\lambda,a}-f_{\ast}\|^2_{L_2(\Pi)}\sqrt{\frac{s}{n}},
\end{equation}
where we also used the fact that
$$
\|(f_{\lambda,a}-f_{\ast})^2\|_{\psi_1}=
\|f_{\lambda,a}-f_{\ast}\|_{\psi_2}^2\leq
c\|f_{\lambda,a}-f_{\ast}\|_{L_2(\Pi)}^2.
$$
Next, we apply representation (\ref{decomp}) to term $f_{\hat \lambda, \hat a}-f_{\lambda,a}$
in $P_n \xi  (f_{\hat \lambda,\hat a}-f_{\lambda,a})$ to get the following bound:
\begin{align}
\label{a110}
\l|P_n\xi(f_{\hat\lambda,\hat a}-f_{\lambda,a})\r| & \leq \|\hat\lambda-\lambda\|_1
 \biggl\|\frac1 n \sum_{j=1}^n \xi_j (X_j-\mb EX)\biggr\|_{\infty}
\\
 &
 \nonumber
+
\l|\frac 1 n \sum_{j=1}^n \xi_j\r|\l|\bar Y_n-\mb EY +\langle \hat \lambda, {\mathbb E}X-\bar X_n\rangle\r|.
\end{align}

To bound the first term in the right hand side of (\ref{a110}), we use a general multiplier inequality (see \cite{Vaart1996Weak-convergenc00}, Lemma 2.9.1):
$$
\mb E \biggl\|\frac{1}{\sqrt{n}} \sum_{j=1}^n \xi_j (X_j-\mb EX)\biggr\|_{\infty} \leq 2\sqrt{2}\|\xi\|_{2,1}
\max_{1\leq k\leq n}\mb E \biggl\|\frac{1}{\sqrt k} \sum_{j=1}^k \eps_j (X_j-\mb EX)\biggr\|_{\infty},
$$
where $\|\xi\|_{2,1}:=\int_{0}^{\infty}\sqrt{{\mb P}\{\xi \geq u\}}du.$
Note that the process $t\mapsto \frac{1}{\sqrt{k}} \sum\limits_{j=1}^k \eps_j (X_j(t)-\mb EX(t))$ is subgaussian for every $k$ with respect to the distance $d_X.$ Therefore,
$$
\mb E \biggl\|\frac{1}{\sqrt{k}} \sum_{j=1}^k \eps_j (X_j-\mb EX)\biggr\|_{\infty}\leq 
C_1 S(\mb T),
$$
which yields 
\begin{align}
\label{a2}
&
\mb E \biggl\|\frac{1}{n} \sum_{j=1}^n \xi_j (X_j-\mb EX)\biggr\|_{\infty}\leq 
C_2\|\xi\|_{\psi_2}\frac{S(\mb T)}{\sqrt{n}},
\end{align}
where we also used the bound $\|\xi\|_{2,1}\leq c\|\xi\|_{\psi_2}.$

Adamczak's inequality (\ref{adamczak}) implies that with probability $\geq 1-e^{-s}$
\begin{align}
\label{adamcz_A}
\biggl\|\frac1 n \sum_{j=1}^n \xi_j (X_j -&\mb EX)\biggr\|_{\infty}\leq \\
& \nonumber
\leq C
\bigg[
\|\xi\|_{\psi_2}\frac{ S(\mb T)}{\sqrt{n}}+ 
\sigma_\xi \sup_{t\in {\mathbb T}}\sqrt{\Var \l(X(t)\r)}\sqrt{\frac s n}
+\|\xi\|_{\psi_2} S(\mb T)\frac{s\log n}{n} \bigg] \\
&\nonumber
\leq C'\|\xi\|_{\psi_2}
 \l[
S({\mathbb T})\sqrt{\frac {s}{n}}\bigvee
S(\mb T)\frac{s\log n}{n}\r]  
\leq
C\|\xi\|_{\psi_2} S({\mathbb T})\sqrt{\frac{s}{n}},
\end{align}
where we also used the bound 
\begin{align}
\label{eq:sup_var}
&
\sup_{t\in \mb T}\sqrt{{\rm Var}(X(t))}\leq 
{\mb E}^{1/2}\sup_{t\in \mb T}|X(t)-\mb EX(t)|^2 \leq CS(\mb T).
\end{align} 
To estimate the second term in (\ref{a110}), we use inequality (\ref{bary})
and also the following tail bounds:
with probability at least $1-e^{-s}$,
\begin{equation}
\label{barxi}
\l|\frac 1 n\sum_{j=1}^n \xi_j\r|\leq C\|\xi\|_{\psi_2}\sqrt{\frac{s}{n}}
\end{equation}
and, with the same probability,  
\begin{equation}
\label{barxxx}
\l\|\bar X_n-\mb EX\r\|_\infty\leq CS(\mb T)\sqrt{\frac s n}.
\end{equation}
Together with (\ref{bary}), these bounds imply that, for some $C>0,$ with probability at least $1-3 e^{-s}$
\begin{align}
\label{term_2}
\l|\frac 1 n \sum_{j=1}^n \xi_j\r|\Big|\bar Y_n-\mb EY &+\langle \hat \lambda, {\mathbb E}X-\bar X_n\rangle\Big| \\
&\nonumber
\leq 
C\|\xi\|_{\psi_2}\l[S({\mathbb T})\frac{s}{n}\|\hat \lambda\|_1 + 
\frac{(\|f_{\ast}-\Pi f_{\ast}\|_{L_2(\Pi)}\vee \|\xi\|_{\psi_2})s}{n}\r].
\end{align}

It follows from bounds (\ref{bound_zz}), (\ref{bound_zzz}), (\ref{a110}), (\ref{adamcz_A})
and (\ref{term_2}) that with probability at least $1-5 e^{-s}$ 
\begin{align*}
\eps\|\hat\lambda\|_1 & \leq \|f_{\lambda,a}-f_{\ast}\|_{L_2(\Pi)}^2+
C_1\|f_{\lambda,a}-f_{\ast}\|^2_{L_2(\Pi)}\sqrt{\frac{s}{n}}
+\eps\|\lambda\|_1
 \\
& 
+C'\|\xi\|_{\psi_2} S({\mathbb T})\sqrt{\frac{s}{n}}\|\hat \lambda-\lambda\|_1
+C\|\xi\|_{\psi_2}S({\mathbb T})\frac{s}{n}\|\hat \lambda\|_1 \\
&
+C\frac{(\|f_{\ast}-\Pi f_{\ast}\|_{L_2(\Pi)}^2\vee \|\xi\|_{\psi_2}^2) s}{n}.
\end{align*}
If constant $D$ in the assumption on $\eps$ is large enough and $s\leq n,$ it implies that 
with some $C>0$
\begin{equation}
\nonumber
\frac{\eps}{2}\|\hat\lambda\|_1\leq
C\|f_{\lambda,a}-f_{\ast}\|_{L_2(\Pi)}^2+
2\eps\|\lambda\|_1 + 
C\frac{(\|f_{\ast}-\Pi f_{\ast}\|_{L_2(\Pi)}^2\vee \|\xi\|_{\psi_2}^2)s}{n},
\end{equation}
and the result immediately follows.
\end{proof}

\textit{Step 5. Putting all the bounds together.} 

We have all the necessary estimates to complete the proof. Let $E$ denote the event on 
which the bounds of Lemma \ref{adamczak_001}, Lemma \ref{linear_001}, Lemma \ref{squared_class}
and also bounds (\ref{eta_bound}), (\ref{bary}), (\ref{barx}), (\ref{bound_zzz}), (\ref{adamcz_A}), (\ref{barxi}) and (\ref{barxxx}) hold. The probability of this 
event is at least $1-10e^{-s}.$ In what follows, we assume that event $E$ occurs. 
Note that in this case the bound of Lemma \ref{norms} also holds. 
Denote 
$$
\hat \delta := \|f_{\hat \lambda}^0-f_{\bar \lambda}^0\|_{L_2(\Pi)}, \ 
\hat \Delta := \int\limits_{{\mathbb T}\setminus {\mathbb T}_w}|\hat \lambda|d\mu,  \ 
\hat R:=\|\hat \lambda\|_1.
$$
Suppose that 
\begin{equation}
\label{lower_upper}
\hat \delta \in [\delta_-,\delta_+], \hat \Delta\in [\Delta_-, \Delta_+], \hat R\in [R_-,R_+].
\end{equation}
It follows from bound (\ref{bou_1}), Lemma \ref{adamczak_001} and bounds (\ref{duo}) -- (\ref{uno}) that 
\begin{align}
\label{eq:process_1}
\frac{1}{C}(P_n&-P)\l[\eta\l(f_{\hat \lambda, \hat a}-f_{\bar \lambda, \bar a}\r)\r] \leq 
\\
\nonumber
& \leq(\|f_{\bar \lambda,\bar a}-f_{\ast}\|_{L_2(\Pi)}\vee \|\xi\|_{\psi_2})
\biggl[\hat \delta \sqrt{\frac{\bar s}{n}}\bigvee \nu_n(\hat \delta,\hat \Delta,\hat R)\biggr]
\bigvee \\
&\nonumber
\bigvee \nu_n^2(\hat \delta,\hat \Delta, \hat R)\bigvee
\l(\|f_{\ast}-\Pi f_{\ast}\|_{L_2(\Pi)}^2\vee \|f_{\bar \lambda,\bar a}-f_{\ast}\|_{L_2(\Pi)}^2\vee \|\xi\|_{\psi_2}^2\r)\frac{s}{n} \\
\nonumber
&\bigvee
\l(\|f_{\bar \lambda,\bar a}-f_{\ast}\|_{L_2(\Pi)}\vee \|\xi\|_{\psi_2}\r) \sqrt{\frac{s}{n}}
\biggl[\hat \delta \sqrt{\frac{\bar s}{n}}\bigvee \nu_n(\hat \delta,\hat \Delta,\hat R)\biggr]
\end{align}
for some absolute constant $C>0.$  
Similarly, Lemma \ref{squared_class} and bound (\ref{trio}) imply
that  
\begin{align}
\label{eq:process_2}
\frac{1}{C}(\Pi&-\Pi_n)(f_{\hat \lambda, \hat a}-f_{\bar \lambda, \bar a})^2
\leq \\
&\nonumber
\hat \delta
\biggl[\hat \delta \sqrt{\frac{\bar s}{n}}\bigvee \nu_n(\hat \delta, \hat \Delta, \hat R)\biggr]
\bigvee
 \nu_n^2(\hat \delta,\hat \Delta, \hat R)
\bigvee \\
&\nonumber
\l(\|f_{\ast}-\Pi f_{\ast}\|_{L_2(\Pi)}\vee  \|f_{\bar\lambda,\bar a}-f_\ast\|_{L_2(\Pi)}\vee\|\xi\|_{\psi_2}\r)\sqrt{\frac{s}{n}}
\biggl[\hat \delta \sqrt{\frac{\bar s}{n}}\bigvee \nu_n(\hat \delta,\hat \Delta,\hat R)\biggr]
\bigvee \\
&\nonumber
\biggl[\hat \delta \sqrt{\frac{\bar s}{n}}\bigvee \nu_n(\hat \delta,\hat \Delta,\hat R)\biggr]^2.
\end{align}
The last two inequalities will be replaced by simplified upper bounds. 
To this end, we use elementary inequalities 
such as $ab\leq \frac{a^2}{2c}+ \frac{c b^2}{2},$ for instance:
$$
C(\|f_{\bar \lambda,\bar a}-f_{\ast}\|_{L_2(\Pi)}\vee \|\xi\|_{\psi_2})
\hat \delta \sqrt{\frac{\bar s}{n}} \leq 
\frac{2C^2}{2}(\|f_{\bar \lambda,\bar a}-f_{\ast}\|_{L_2(\Pi)}^2\vee \|\xi\|_{\psi_2}^2)\frac{\bar s}{n}
+\frac{1}{8}\hat \delta^2. 
$$
Also recall that by (\ref{barlambda_1}), (\ref{sigmaY}) and the assumption that
$\xi \in {\mathcal L},$ 
$$\|f_{\ast}-\Pi f_{\ast}\|_{L_2(\Pi)}\vee 
\|f_{\bar \lambda,\bar a}-f_{\ast}\|_{L_2(\Pi)}\vee \|\xi\|_{\psi_2}\leq \sigma_Y.
$$
Whenever it is more convenient, we can replace 
the maximum $\bigvee$ by the sum, or vice versa (with a proper change of constant $C$), we can drop
repetitive terms in the maximum, etc. With such simple transformations, it is easy to get the 
following bound (with some constant $C>0$ and under the assumption that $\bar s\leq n$):   
\begin{align}
\label{compressed}
(P_n&-P)\eta(f_{\hat \lambda, \hat a}-f_{\bar \lambda, \bar a}) + 
(\Pi-\Pi_n)(f_{\hat \lambda, \hat a}-f_{\bar \lambda, \bar a})^2
\leq \\
&\nonumber 
\leq\frac{1}{8}\hat \delta^2 + C \hat \delta^2 \sqrt{\frac{\bar s}{n}}+
C\sigma_Y\nu_n (\hat \delta,\hat \Delta,\hat R)
+C \nu_n^2 (\hat \delta,\hat \Delta,\hat R)+
C\frac{\sigma_Y^2\bar s}{n}. 
\end{align}
Note that 
$$
\nu_n (\hat \delta,\hat \Delta,\hat R)=
\inf_{L\subset {\mathcal L}}\biggl[\hat \delta \sqrt{\frac{{\rm dim}(L)}{n}}\bigvee
(\hat R\vee \|\bar \lambda\|_1)\frac{\gamma_2(\rho(L))}{\sqrt{n}}\bigvee \hat \Delta \frac{S({\mathbb T})}{\sqrt{n}}\biggr]
\leq (\hat R\vee \|\bar\lambda\|_1) \frac{S({\mathbb T})}{\sqrt{n}},
$$
where we used the bounds $\hat \Delta \leq \hat R,$ $\gamma_2(\rho)\leq S({\mathbb T})$ 
and computed the expression in the right hand side of the definition of $\nu_n$ for a trivial
subspace of zero dimension. 
Using Lemma \ref{norms}, we get the following bound:
$$
\hat R \vee \|\bar \lambda\|_1 \leq 
c\biggl(\frac{q(\eps)}{\eps}\bigvee 
\frac{\sigma_Y^2 s}{n\eps}\biggr)\bigvee \|\bar \lambda\|_1,
$$  
which holds with probability at least $1-e^{-s}.$
%
%
Therefore, 
\begin{align*}
\nu_n^2 (\hat \delta,\hat \Delta,\hat R)&\leq 
c\biggl(\frac{q(\eps)}{\eps}\bigvee 
\frac{\sigma_Y^2 s}{n\eps}\biggr)\frac{S(\mb T)}{\sqrt n}\nu_n (\hat \delta,\hat \Delta,\hat R)+
\|\bar \lambda\|_1\frac{S(\mb T)}{\sqrt n}\nu_n (\hat \delta,\hat \Delta,\hat R)\leq \\
& 
c\biggl(\frac{q(\eps)}{\eps}\bigvee 
\frac{\sigma_Y^2 s}{n\eps}\biggr)\frac{S(\mb T)}{\sqrt n}\nu_n (\hat \delta,\hat \Delta,\hat R)+
\frac 1 2 \frac{\|\bar\lambda\|_1^2 S^2(\mb T)}{n}+\frac 1 2 \nu_n^2 (\hat \delta,\hat \Delta,\hat R),
\end{align*}
and inequality (\ref{compressed}) easily yields
\begin{align}
\label{compressed_1}
(P_n&-P)\eta(f_{\hat \lambda, \hat a}-f_{\bar \lambda, \bar a}) + 
(\Pi-\Pi_n)(f_{\hat \lambda, \hat a}-f_{\bar \lambda, \bar a})^2
\leq \frac{1}{8}\hat \delta^2 + C \hat \delta^2 \sqrt{\frac{\bar s}{n}}+ \\
&\nonumber 
+C\sigma_Y\nu_n (\hat \delta,\hat \Delta,\hat R)+
C \biggl(\frac{q(\eps)}{\eps}\bigvee 
\frac{\sigma_Y^2s}{n\eps}\biggr)\frac{S({\mathbb T})}{\sqrt{n}}
\nu_n(\hat \delta,\hat \Delta,\hat R)+\\
&\nonumber 
+C\frac{\|\bar\lambda\|_1^2 S^2(\mb T)}{n}+
C\frac{\sigma_Y^2 \bar s}{n}. 
\end{align}

\begin{remark}
\label{rem:bounded_norm}
Note that under an additional assumption that 
$$
\mb D\subset \l\{\lambda\in L_1(\mu): \  \|\lambda\|_1\leq \frac{C\sigma_Y\sqrt{n}}{S(\mathbb T)}\r\}
$$
(in particular, $\big\|\bar\lambda\big\|_1\leq C\frac{\sigma_Y\sqrt n}{S({\mathbb T})} \big),$ 
we have 
\begin{align*}
\nu_n^2 (\hat \delta,\hat \Delta,\hat R)&
\leq 
c\biggl(\frac{q(\eps)}{\eps}\bigvee 
\frac{\sigma_Y^2 s}{n\eps}\biggr)\frac{S(\mb T)}{\sqrt n}\nu_n (\hat \delta,\hat \Delta,\hat R)+
\|\bar \lambda\|_1\frac{S(\mb T)}{\sqrt n}\nu_n (\hat \delta,\hat \Delta,\hat R) \\
&\leq
c\biggl(\frac{q(\eps)}{\eps}\bigvee 
\frac{\sigma_Y^2 s}{n\eps}\biggr)\frac{S(\mb T)}{\sqrt n}\nu_n (\hat \delta,\hat \Delta,\hat R)+
C\sigma_Y\nu_n (\hat \delta,\hat \Delta,\hat R),
\end{align*}
so that the term $\frac{\|\bar\lambda\|_1^2 S^2(\mb T)}{n}$ disappears from (\ref{compressed_1}). In this case, the remainder of the proof yields Theorem \ref{th:bounded_norm}. 
\end{remark}

Under the assumption (\ref{eps-condition}) on $\eps$ and the inequality
$s\leq c n$ (which easily follows from the the main conditions of the theorem), 
we have  
$$ 
\frac{s\sigma_Y}{n\eps}\frac{S({\mathbb T})}{\sqrt{n}}\leq c_1
$$
with some constant $c_1>0.$
Also, condition (\ref{eps-condition}) and the inequality $q(\eps)\leq \|f_\ast-\Pi f_\ast\|_{L_2(\Pi)}^2\leq \sigma_Y^2$
implies that 
$$
\frac{q(\eps)}{\eps} \frac{S({\mathbb T})}{\sqrt{n}}
\leq \fr c \sigma_Y.
$$
Hence, with some constant $C>0$ and for any subspace $L\subset {\mathcal L}_X$
with ${\rm dim}(L)=d$ and $\rho(L)=\rho,$
\begin{align}
\label{compressed_2}
(P_n-P)&\eta(f_{\hat \lambda, \hat a}-f_{\bar \lambda, \bar a}) + 
(\Pi-\Pi_n)(f_{\hat \lambda, \hat a}-f_{\bar \lambda, \bar a})^2
\leq
\\
\nonumber 
&
\leq\frac{1}{8}\hat \delta^2 + C \hat \delta^2 \sqrt{\frac{\bar s}{n}}+
C\sigma_Y
\biggl[\hat \delta \sqrt{\frac{d}{n}}\bigvee (\hat R\vee \|\bar \lambda\|_1)\frac{\gamma_2(\rho)}{\sqrt{n}}
\bigvee \hat \Delta \frac{S({\mathbb T})}{\sqrt{n}}\biggr]+
\\
\nonumber 
&
+C\frac{\|\bar \lambda\|_1^2 S^2({\mathbb T})}{n}+
C\frac{\sigma_Y^2 \bar s}{n}. 
\end{align}

We will now substitute (\ref{compressed_2}) in the right hand side of bound (\ref{r3'}).
Recall that $\mb T_{\bar w}=\{t: \bar w(t)\geq 1/2\}$. 
Since, by monotonicity 
of subdifferentials, $(\hat w(t)-w(t))(\hat \lambda(t)-\lambda(t))\geq 0$ for all $t\in \mb T$, and $\bar w, \hat w$
take their values in $[-1,1]$ by definition, we also have that  
\begin{equation}
\label{outside}
\dotp{\hat w-\bar w}{\hat\lambda-\bar\lambda}\geq \frac12
\int\limits_{\mb T\setminus \mb T_{\bar w}}
|\hat\lambda|d\mu.
\end{equation}
Taking this into account, we get  
\begin{align}
\label{bound_X}
\|f_{\hat \lambda, \hat a}-f_{\ast}\|_{L_2(\Pi)}^2 &
+\|f_{\hat \lambda, \hat a}-f_{\bar \lambda, \bar a}\|_{L_2(\Pi)}^2 
+\frac{\eps}{2}\dotp{\hat w-\bar w}{\hat\lambda-\bar\lambda}
+\frac{\eps}{4}
\int\limits_{\mb T\setminus \mb T_{\bar w}}
|\hat\lambda|d\mu\leq \\
&\nonumber
\leq  
\|f_{\bar \lambda, \bar a}-f_{\ast}\|_{L_2(\Pi)}^2+
\eps\dotp{\bar w}{\bar\lambda-\hat\lambda}+ 
\frac{1}{8}\hat \delta^2 + C \hat \delta^2 \sqrt{\frac{\bar s}{n}}+
\\
\nonumber
&
+C\sigma_Y
\biggl[\hat \delta \sqrt{\frac{d}{n}}\bigvee (\hat R \vee 
\|\bar \lambda\|_1 )\frac{\gamma_2(\rho)}{\sqrt{n}}
\bigvee \hat \Delta \frac{S({\mathbb T})}{\sqrt{n}}\biggr]+ \\
&\nonumber 
+C\frac{\|\bar \lambda\|_1^2 S^2({\mathbb T})}{n}+
C\frac{\sigma_Y^2 \bar s}{n}. 
\end{align}
Note also that 
$$
\|\hat \lambda\|_1\leq \|\bar \lambda\|_1 + 
\dotp{\bar w}{\hat \lambda - \bar \lambda}
+ \dotp{\hat w- \bar w}{\hat \lambda - \bar \lambda},
$$
which will be used to control $\hat R\vee \|\bar \lambda\|_1=\|\hat \lambda\|_1\vee \|\bar \lambda\|_1.$
Then, bound (\ref{bound_X}) implies the following (with a different value of $C$):
\begin{align}
\label{bound_Y}
\|f_{\hat \lambda, \hat a}-f_{\ast}\|_{L_2(\Pi)}^2 &
+\|f_{\hat \lambda, \hat a}-f_{\bar \lambda, \bar a}\|_{L_2(\Pi)}^2 
+\frac{\eps}{2}\dotp{\hat w-\bar w}{\hat\lambda-\bar\lambda} 
+\frac{\eps}{4}\hat \Delta\leq \\
&\nonumber
\leq  
\|f_{\bar \lambda, \bar a}-f_{\ast}\|_{L_2(\Pi)}^2+
\eps\dotp{\bar w}{\bar\lambda-\hat\lambda}
+\frac{1}{4}\hat \delta^2 + C \hat \delta^2 \sqrt{\frac{\bar s}{n}}+ \\
&\nonumber
+C\frac{\sigma_Y^2 d}{n}
+C\sigma_Y
\|\bar \lambda\|_1 \frac{\gamma_2(\rho)}{\sqrt{n}}
+
\hat \Delta C\sigma_Y\frac{S({\mathbb T})}{\sqrt{n}}+\\
&\nonumber
+ C\sigma_Y\frac{\gamma_2(\rho)}{\sqrt{n}}\l(\dotp{\bar w}{\hat \lambda - \bar \lambda}\vee 0\r) 
+ C\sigma_Y\frac{\gamma_2(\rho)}{\sqrt{n}}\dotp{\hat w- \bar w}{\hat \lambda - \bar \lambda}+\\
&\nonumber 
+C\frac{\|\bar \lambda\|_1^2 S^2({\mathbb T})}{n}+
C\frac{\sigma_Y^2 \bar s}{n}. 
\end{align}
If constant $D$ in the condition on $\eps$ is large enough, we have 
$
C\sigma_Y\frac{\gamma_2(\rho)}{\sqrt{n}}\leq  C
\sigma_Y\frac{S({\mathbb T})}{\sqrt{n}}\leq 
\eps/8,
$
which implies 
\begin{align}
\label{bound_Z}
\|f_{\hat \lambda, \hat a}-f_{\ast}\|_{L_2(\Pi)}^2 &
+\|f_{\hat \lambda, \hat a}-f_{\bar \lambda, \bar a}\|_{L_2(\Pi)}^2 
+\frac{\eps}{4}\dotp{\hat w-\bar w}{\hat\lambda-\bar\lambda} 
+\frac{\eps}{8}\hat \Delta\leq \\
&\nonumber
\leq  
\|f_{\bar \lambda, \bar a}-f_{\ast}\|_{L_2(\Pi)}^2+
\frac{9}{8}\eps \l(\dotp{\bar w}{\bar\lambda-\hat\lambda}\vee 0\r)
+\frac{1}{4}\hat \delta^2 + C \hat \delta^2 \sqrt{\frac{\bar s}{n}}+\\
&\nonumber
+C\frac{\sigma_Y^2 d}{n}
+C\sigma_Y
\|\bar \lambda\|_1 \frac{\gamma_2(\rho)}{\sqrt{n}}
+C\frac{\|\bar \lambda\|_1^2 S^2({\mathbb T})}{n}+
C\frac{\sigma_Y^2 \bar s}{n}. 
\end{align}
Finally, note that 
$
\hat \delta = 
\|f_{\hat \lambda}^0-f_{\bar \lambda}^0\|_{L_2(\Pi)}
\leq 
\|f_{\hat \lambda, \hat a}-f_{\bar \lambda, \bar a}\|_{L_2(\Pi)}. 
$
Because of this, under the assumption that $C,\bar s$ and $n$ are such that $C\sqrt{\frac{\bar s}{n}}\leq 1/4,$ 
we get from (\ref{bound_Z})
\begin{align}
\label{bound_W}
\|f_{\hat \lambda, \hat a}-f_{\ast}\|_{L_2(\Pi)}^2 &
+
\frac{1}{2}\|f_{\hat \lambda, \hat a}-f_{\bar \lambda, \bar a}\|_{L_2(\Pi)}^2 
+\frac{\eps}{4}\dotp{\hat w-\bar w}{\hat\lambda-\bar\lambda} 
+\frac{\eps}{8}
\int\limits_{\mb T\setminus \mb T_{\bar w}}
|\hat\lambda|d\mu\leq \\
&\nonumber
\leq  
\|f_{\bar \lambda, \bar a}-f_{\ast}\|_{L_2(\Pi)}^2+
\frac{9}{8}\eps \l(\dotp{\bar w}{\bar\lambda-\hat\lambda}\vee 0\r)+
C\frac{\sigma_Y^2 d}{n}
+\\
&\nonumber
+C\sigma_Y
\|\bar \lambda\|_1 \frac{\gamma_2(\rho)}{\sqrt{n}}+
C\frac{\|\bar \lambda\|_1^2 S^2({\mathbb T})}{n}+
C\frac{\sigma_Y^2 \bar s}{n}. 
\end{align}
First, assume that 
\begin{align}
\label{case_X}
\frac{7}{8}\eps \l(\dotp{\bar w}{\bar\lambda-\hat\lambda}\vee 0\r)\geq 
C\frac{\sigma_Y^2 d}{n}
+C\sigma_Y
\|\bar \lambda\|_1 \frac{\gamma_2(\rho)}{\sqrt{n}}
+C\frac{\|\bar \lambda\|_1^2 S^2({\mathbb T})}{n}+
C\frac{\sigma_Y^2 \bar s}{n}. 
\end{align}
In this case, bound (\ref{bound_W}) implies that 
\begin{align}
\label{bound_W'}
\|f_{\hat \lambda, \hat a}-f_{\ast}\|_{L_2(\Pi)}^2 &
+
\frac{1}{2}\|f_{\hat \lambda, \hat a}-f_{\bar \lambda, \bar a}\|_{L_2(\Pi)}^2 
+\frac{\eps}{8}
\int\limits_{\mb T\setminus \mb T_{\bar w}}
|\hat\lambda|d\mu\leq \\
&\nonumber
\leq  
\|f_{\bar \lambda, \bar a}-f_{\ast}\|_{L_2(\Pi)}^2+
2\eps \dotp{\bar w}{\bar\lambda-\hat \lambda} .
\end{align}
If 
$
\|f_{\hat \lambda, \hat a}-f_{\ast}\|_{L_2(\Pi)}^2\leq \|f_{\bar \lambda, \bar a}-f_{\ast}\|_{L_2(\Pi)}^2,
$
the inequality of the theorem trivially holds. 
Otherwise, (\ref{bound_W'}) implies that 
$$
\int\limits_{\mb T\setminus \mb T_{\bar w}}
|\hat\lambda -\bar \lambda|d\mu \leq 16\dotp{\bar w}{\bar\lambda -\hat \lambda},
$$
which means that $\bar \lambda -\hat \lambda \in C_{\bar w}^{(16)}$ and 
$$
\dotp{\bar w}{\bar\lambda - \hat \lambda}\leq 
\fr a(\bar w)\l\|f^0_{\hat \lambda}-f^0_{\bar \lambda}\r\|_{L_2(\Pi)}
\leq 
\fr a(\bar w)\|f_{\hat \lambda, \hat a}-f_{\bar \lambda,\bar a}\|_{L_2(\Pi)}
$$
by the definition of $\fr a(\cdot)=\fr a^{(16)}(\cdot).$
Therefore, we have 
\begin{align}
\|f_{\hat \lambda, \hat a}-f_{\ast}\|_{L_2(\Pi)}^2 &
+
\frac{1}{2}\|f_{\hat \lambda, \hat a}-f_{\bar \lambda, \bar a}\|_{L_2(\Pi)}^2 +
\frac{\eps}{8}\int\limits_{\mb T\setminus \mb T_{\bar w}}
|\hat\lambda|d\mu \leq \\
&\nonumber
\leq  
\|f_{\bar \lambda, \bar a}-f_{\ast}\|_{L_2(\Pi)}^2+
2\eps \fr a(\bar w) \|f_{\hat \lambda, \hat a}-f_{\bar \lambda,\bar a}\|_{L_2(\Pi)}\leq \\
&\nonumber
\leq  
\|f_{\bar \lambda, \bar a}-f_{\ast}\|_{L_2(\Pi)}^2+
2\fr a^2(\bar w)\eps^2+ \frac{1}{2}\|f_{\hat \lambda, \hat a}-f_{\bar \lambda,\bar a}\|_{L_2(\Pi)}^2,
\end{align}
which again implies the bound of the theorem. 

If condition (\ref{case_X}) does not hold, then bound (\ref{bound_W}) implies that with 
some constant $C>0$
\begin{align}
\label{bound_V}
\|f_{\hat \lambda, \hat a}-f_{\ast}\|_{L_2(\Pi)}^2 &
+\frac{\eps}{8}\int\limits_{\mb T\setminus \mb T_{\bar w}}|\hat\lambda|d\mu
\leq  
\|f_{\bar \lambda, \bar a}-f_{\ast}\|_{L_2(\Pi)}^2+
C\frac{\sigma_Y^2 d}{n}+ \\
& \nonumber
+C\sigma_Y
\|\bar \lambda\|_1 \frac{\gamma_2(\rho)}{\sqrt{n}}
+C\frac{\|\bar \lambda\|_1^2 S^2({\mathbb T})}{n}+
C\frac{\sigma_Y^2 \bar s}{n}, 
\end{align}
which gives the bound of the theorem in this case. 
\index{bg@$\delta_{-}, \delta_{+}, \Delta_{-}, \Delta_{+}, R_{-}, R_{+}$}
To complete the proof, it remains to choose the values of quantities $\delta_{-}, \delta_{+}, \Delta_{-}, \Delta_{+}$ and $R_{-}, R_{+}$ and to explain how to establish the bound 
of the theorem in the case when conditions (\ref{lower_upper}) do not hold. 
We will 
choose the values 
\begin{align*}
&
\delta_{+}:= C_1 \sigma_Y \sqrt{n}, \ 
\delta_{-}:= \frac{C_1\sigma_Y }{\sqrt{n}}, \\
&
R_{+}=\Delta_{+}=\frac{C_1 \sigma_Y\sqrt{n}}{S({\mathbb T})}, 
R_{-}=\Delta_{-}=\frac{C_1 \sigma_Y}{S({\mathbb T})\sqrt{n}},
\end{align*}
where $C_1$ is a large enough constant. Recall that  
$
\bar s = s+\log ((J_1+1)(J_2+1)(J_3+1))
$
and, for our choice of $\delta_-,\delta_+, \Delta_-,\Delta_+, R_-,R_+$
we have $J_1=J_2=J_3=\lfloor\log_2 n\rfloor+1.$ 
Therefore, $\bar s=s+ 3\log (\lfloor\log_2 n\rfloor+2).$
Since $q(\eps)\leq \|f_{\ast}-\Pi f_{\ast}\|_{L_2(\Pi)}^2,$
it easily follows from Lemma \ref{norms} that 
$$
\|\hat \lambda\|_1 \leq C\biggl[\frac{q(\eps)}{\eps}+\frac{\sigma_Y^2 s}{n\eps}\biggr]
\leq 2C_2 \frac{\sigma_Y^2}{\eps}\leq C_3 \frac{\sigma_Y}{S({\mathbb T})}\sqrt{n}\leq R_+, 
$$ 
provided that constant $C_1$ is large enough. It is also easy to see from (\ref{barlambda_2})
that $\|\bar \lambda\|_1\leq R_+.$ Thus, $\hat R \vee \|\bar \lambda\|_1\leq R_+,$ and also $\hat \Delta\leq \hat R\leq R_{+}=\Delta_{+}.$ 
In addition, 
\begin{align*}
\hat \delta &
= \|f_{\hat \lambda}^0-f_{\bar \lambda}^0\|_{L_2(\Pi)}
= {\mathbb E}^{1/2}\dotp{\hat \lambda-\bar \lambda}{X-{\mathbb E}X}^2
\leq \|\hat \lambda-\bar \lambda\|_1 \sup_{t\in {\mathbb T}}\sqrt{{\rm Var}(X(t))}\leq \\
&
\leq C_2(\hat R\vee \|\bar \lambda\|_1) S(\mb T)\leq C_3 \frac{\sigma_Y}{S({\mathbb T})}\sqrt{n}S({\mathbb T})
\leq \delta_{+},
\end{align*}
again, provided that constant $C_1$ is large enough. Here, we also used 
the bound (\ref{eq:sup_var}) to estimate $\sup\limits_{t\in \mb T}\sqrt{\Var(X(t))}$. 

Thus, conditions $\hat \delta\leq \delta_{+}, \hat \Delta\leq \Delta_{+}, 
\hat R\leq R_{+}$ hold on the event $E.$
If some of the conditions $\hat \delta \geq \delta_{-},$ $\hat \Delta \geq \Delta_{-},$
$\hat R\geq R_{-}$ are violated, we can still use bound (\ref{bound_Y}) with quantities $\hat \delta, \hat \Delta, \hat R$ that fall outside the intervals being replaced in its right hand side by the corresponding upper bound $\delta_{-}, \Delta_{-}, R_{-}.$ It is easy to check that the inequality of the theorem
still holds in this case with a proper constant $C.$

It now remains to replace $s$ by $s+3$ (so that ${\mb P}(E)\geq 1-10e^{-s-3}\geq 1-e^{-s}$) to get that 
the bound of the theorem holds with probability at least $1-e^{-s}.$

\subsection{Proof of Theorem \ref{th:slow_rate}.}

Most of the necessary ingredients have been already developed in the proof of Theorem \ref{th:main_AA}.
Let $(\bar \lambda, \bar a)$ be a couple that minimizes the right hand side 
of bound (\ref{slow_oracle}). 
As before, if the infimum is not attained, the proof can be easily modified. 
We also have that (plugging $(0,\Pi f_\ast)$ in the right hand side of (\ref{slow_oracle}))
\begin{align*}
&
\|f_{\bar\lambda,\bar a}-f_\ast\|_{L_2(\Pi)}^2\leq \|f_\ast-\Pi f_{\ast}\|^2_{L_2(\Pi)},\ \ 
&
\|\bar\lambda\|_1\leq \frac 2 3 \frac{\|f_\ast-\Pi f_{\ast}\|_{L_2(\Pi)}^2}{\eps}.
\end{align*}
The following inequality is equivalent to (\ref{r3'}):
\begin{align}
\label{eq:sl0}
\|f_{\hat \lambda, \hat a}-f_{\ast}\|_{L_2(\Pi)}^2 &
+\|f_{\hat \lambda, \hat a}-f_{\bar \lambda, \bar a}\|_{L_2(\Pi)}^2 
+\eps\dotp{\hat w}{\hat\lambda-\bar\lambda}\leq 
\|f_{\bar \lambda, \bar a}-f_{\ast}\|_{L_2(\Pi)}^2+\\ 
&\nonumber
+
2(P_n-P)\eta(f_{\hat \lambda, \hat a}-f_{\bar \lambda, \bar a})
+2(\Pi-\Pi_n)\l(f_{\hat \lambda, \hat a}-f_{\bar \lambda, \bar a}\r)^2.
\end{align}
Note that 
\begin{align}
\label{eq:sl1}
&
\eps\dotp{\hat w}{\hat\lambda-\bar\lambda}\geq \eps\l(\|\hat\lambda\|_1-\|\bar\lambda\|_1\r).
\end{align}
To bound the empirical processes on the right hand side of (\ref{eq:sl0}), we use inequalities (\ref{eq:process_1}) and (\ref{eq:process_2}) which imply that (see (\ref{compressed_2}) above for details) with some constant $C>0$ and for any subspace $L\subset {\mathcal L}$
with ${\rm dim}(L)=d$ and $\rho(L)=\rho,$
\begin{align}
\label{eq:process_3}
(P_n-P)&\eta(f_{\hat \lambda, \hat a}-f_{\bar \lambda, \bar a}) + 
(\Pi-\Pi_n)(f_{\hat \lambda, \hat a}-f_{\bar \lambda, \bar a})^2
\leq
\\
\nonumber 
&
\leq\frac{1}{8}\hat \delta^2 + C \hat \delta^2 \sqrt{\frac{\bar s}{n}}+
C\sigma_Y
\biggl[\hat \delta \sqrt{\frac{d}{n}}\bigvee (\hat R \vee \|\bar \lambda\|_1)\frac{\gamma_2(\rho)}{\sqrt{n}}
\bigvee \hat \Delta \frac{S({\mathbb T})}{\sqrt{n}}\biggr]+
\\
\nonumber 
&
+C\frac{\|\bar \lambda\|_1^2 S^2({\mathbb T})}{n}+
C\frac{\sigma_Y^2 \bar s}{n}
\end{align}
holds on the event $E$ (defined in the proof of Theorem \ref{th:main_AA}) of probability at least $\geq 1-10e^{-s}$,
where
$$
\hat \delta := \l\|f_{\hat \lambda}^0-f_{\bar \lambda}^0\r\|_{L_2(\Pi)}
\leq \l\|f_{\hat \lambda, \hat a}-f_{\bar \lambda, \bar a}\r\|_{L_2(\Pi)}, \ 
\hat \Delta := \int\limits_{{\mathbb T}\setminus {\mathbb T}_w}|\hat \lambda|d\mu,  \ 
\hat R:=\|\hat \lambda\|_1
$$
and we assume that bounds (\ref{lower_upper}) hold. 
Using the inequalities $\gamma_2(\rho)\leq S(\mb T)$, $\hat\Delta\leq\hat R$ and choosing $L$ to be the trivial subspace of dimension $0$, we get 
\begin{align}
\label{eq:sl3}
(P_n-P)&\eta(f_{\hat \lambda, \hat a}-f_{\bar \lambda, \bar a}) + 
(\Pi-\Pi_n)(f_{\hat \lambda, \hat a}-f_{\bar \lambda, \bar a})^2
\leq
\\
\nonumber 
&
\leq\frac{1}{8}\hat \delta^2 + C \hat \delta^2 \sqrt{\frac{\bar s}{n}}+
C\sigma_Y
(\|\bar\lambda\|_1\vee \|\hat\lambda_\eps\|_1)\frac{S(\mb T)}{\sqrt n}+
\\
\nonumber 
&
+C\frac{\|\bar \lambda\|_1^2 S^2({\mathbb T})}{n}+
C\frac{\sigma_Y^2 \bar s}{n},
\end{align}
Substituting (\ref{eq:sl3}) and (\ref{eq:sl1}) back in (\ref{eq:sl0}), we get that with some $C>0$
\begin{align}
\label{eq:sl5}
\|f_{\hat \lambda, \hat a}-f_{\ast}\|_{L_2(\Pi)}^2 &
+\|f_{\hat \lambda, \hat a}-f_{\bar \lambda, \bar a}\|_{L_2(\Pi)}^2 
+\eps\|\hat\lambda_\eps\|_1\leq 
\|f_{\bar \lambda, \bar a}-f_{\ast}\|_{L_2(\Pi)}^2+\\ 
&\nonumber
+\eps\|\bar\lambda\|_1+\frac{1}{4} \l\|f_{\hat \lambda,\hat a}-f_{\bar \lambda,\bar a}\r\|^2_{L_2(\Pi)} + 
C \l\|f_{\hat \lambda,\hat a}-f_{\bar \lambda,\bar a}\r\|^2_{L_2(\Pi)} \sqrt{\frac{\bar s}{n}}+\\
&\nonumber
+C\sigma_Y
(\|\bar\lambda\|_1+ \|\hat\lambda_\eps\|_1)\frac{S(\mb T)}{\sqrt n}
+C\frac{\|\bar \lambda\|_1^2 S^2({\mathbb T})}{n}+
C\frac{\sigma_Y^2 \bar s}{n}.
\end{align}
If the constant $D$ in condition (\ref{eps-condition_sl}) is large enough, we have 
$C\sigma_Y\frac{S(\mb T)}{\sqrt n}\leq \frac{\eps}{4}$ and, 
since $\|\bar\lambda\|_1\leq \frac{2\|f_\ast-\Pi f_{\ast}\|^2_{L_2(\Pi)}}{3\eps}\leq \frac{2\sigma_Y^2}{3\eps}$, 
$$
C\frac{\|\bar \lambda\|_1^2 S^2({\mathbb T})}{n}
\leq C\eps\|\bar\lambda\|_1 \frac{2\|f_\ast-\Pi f_{\ast}\|^2_{L_2(\Pi)}S^2(\mb T)}{3\eps^2 n}
\leq 
\frac \eps 4 \|\bar\lambda\|_1.
$$
Moreover, if $C\sqrt{\frac{\bar s}{n}}\leq \frac 3 4$, (\ref{eq:sl5}) yields
$$
\|f_{\hat \lambda, \hat a}-f_{\ast}\|_{L_2(\Pi)}^2 
+\frac 3 4\eps\|\hat\lambda_\eps\|_1\leq 
\|f_{\bar \lambda, \bar a}-f_{\ast}\|_{L_2(\Pi)}^2+\frac 3 2 \eps\|\bar\lambda\|_1+C\frac{\sigma_Y^2 \bar s}{n}.
$$
The case when (\ref{lower_upper}) does not hold can be handled exactly as 
at the end of the proof of Theorem \ref{th:main_AA}.

\subsection{Proof of Proposition \ref{interpolation}.}
\label{pf:interpolation}

For simplicity, we consider the case $d=1$. 
Extension to arbitrary dimension follows the same proof pattern. 

Note that by (\ref{bochner}),
\begin{align}
\label{variance}
&
\Var\l(\sum_{j=1}^N u_j X(t_j)\r)=\sum_{1\leq j,l\leq N}k(t_j-t_l)u_j u_l=\int_\mb R \l|\sum_{j=1}^N e^{i t_j z }u_j\r|^2 v(z)dz.
\end{align}
Clearly, the function $q(z)=\l|\sum\limits_{j=1}^N e^{i t_j z }u_j\r|^2$ is periodic with period $2\pi N$; 
let $I:=\int\limits_{-\pi N}^{\pi N}q(z)v(z)dz$ and $0\ne m\in \mb Z$.  
Together with (\ref{density}), this gives 
\begin{align*}
&
\int\limits_{2\pi mN-\pi N}^{2\pi mN+\pi N}q(z)v(z)dz=\int\limits_{2\pi mN-\pi N}^{2\pi mN+\pi N}q(z)v(z-2\pi mN)\frac{v(z)}{v(z-2\pi mN)}dz \\
& 
\leq \sup_{|y-2\pi mN|\leq \pi N}\frac{v(y)}{v(y-2\pi mN)}\int\limits_{-\pi N}^{\pi N}q(z)v(z)dz\leq \frac{C}{(|m|-1/2)^2}\cdot I.
\end{align*}
Hence
\begin{align*}
\int\limits_{-\pi N}^{\pi N} \l|\sum_{j=1}^N e^{i t_j z }u_j\r|^2 v(z)dz &
\leq 
\int_\mb R \l|\sum_{j=1}^N e^{i t_j z }u_j\r|^2 v(z)dz
\leq \\
&
\leq \underbrace{C\sum_{m\in \mb Z}\frac{1}{(|m|-1/2)^2}}_{C_2}\int\limits_{-\pi N}^{\pi N}\l|\sum_{j=1}^N e^{i t_j z }u_j\r|^2 v(z)dz.
\end{align*}
Recall that our goal is to bound $\|\vec w\|_K$ for $\vec w\in \partial \|\lambda\|_1$ where $\lambda\in \mb R^N$. 
It will be convenient to represent $\vec w=(w(t_1),\ldots,w(t_N))^T$ as a restriction of a smooth, compactly supported function $w(t), \ t\in \mb R$  on 
a grid $\m G_N$. 
Clearly, $w(t)$ is not unique, and we will be interested in interpolation of ``minimal energy'', as explained below.  

Note that the map $\ell_2(\mb Z)\ni x\mapsto \hat x_N\in L_2([-\pi N ,\pi N ],dy), 
\hat x_N(y):=\frac{1}{\sqrt{2\pi N}}\sum\limits_{j\in \mb Z} x_j e^{i \frac{j}{N} y}$ is an isometry. 
With the convention $u_j=0, \ j\notin\l\{1,\ldots,N\r\}$, this implies
\begin{align*}
\dotp{\vec w}{\vec u}_2 &
=\sum_{j\in \mb Z} w\l(\frac{2\pi j}{N}\r) u_j=\dotp{\hat w_N}{\hat u_N}_{L_2([-\pi N,\pi N],dy)}=
\dotp{\frac{\hat w_N}{\sqrt{Nv}}}{\hat u_N\sqrt{Nv}}_{L_2([-\pi N,\pi N],dy)} \\
&
\leq \frac{1}{\sqrt N}\l(\int\limits_{-\pi N}^{\pi N}\frac{\l|\hat w_N(y)\r|^2}{v(y)}dy \int\limits_{-\pi N}^{\pi N} N\l|\hat u_N(y)\r|^2 v(y)dy \r)^{1/2} \\
&
\leq
\frac{c}{\sqrt N}\l(\int\limits_{-\pi N}^{\pi N}(1+y^2)^p\l| \hat w_N(y)\r|^2 dy \int\limits_{-\infty}^{\infty}N\l|\hat u_N(y)\r|^2 v(y)dy\r)^{1/2},
\end{align*}
hence by (\ref{variance})
\begin{align}
\label{fourier}
&
\|\vec w\|^2_K\leq \frac{C}{N}\int\limits_{-\pi N}^{\pi N}(1+y^2)^p\l| \hat w_N(y)\r|^2 dy.
\end{align}
Next, define $w_N(y):=\frac{1}{\sqrt {2\pi N}}\int\limits_{-\pi N}^{\pi N}e^{- i t\cdot y} \hat w_N(t) dt$. 
A simple direct computation gives 
$$
w_N(y)=\sum_{j\in \mb Z} w\l(\frac{2\pi j}{N}\r)\sinc(\pi N (y-j/N)),
$$ 
where $\sinc(x)=\frac{\sin x}{x}$. 
In other words, $w_N(y)$ is the {\it spectral approximation} of $w(y)$. 
Define $w^{(p)}_N(y):=\frac{e^{-i\frac{\pi p}{2}}}{\sqrt{2\pi N}}\int\limits_{-\pi N}^{\pi N}e^{- i t\cdot y}\,t^p \hat w_N(t) dt$ (note that for $p\in \mb N$ this is just the $p$'th derivative of $w_N(y)$).  
By the isometric property of Fourier transform, this gives
$$
\frac 1 N \int\limits_{-\pi N}^{\pi N} t^{2p} |\hat w_N(t)|^2 dt = C\int_\mb R |w_N^{(p)}(t)|^2 dt,
$$
hence (\ref{fourier}), together with the triangle inequality, implies
\begin{align}
\label{result}
&
\|\vec w\|_K^2\leq C_1\|w_N\|^2_{\mb W^{2,p}(\mb R)}\leq 2C_1\l(\|w\|^2_{\mb W^{2,p}(\mb R)}+\l\|w-w_N\r\|^2_{\mb W^{2,p}(\mb R)}\r).
\end{align}
We will need the following important fact (it will be used for $m=p$): 

\begin{theorem}[\cite{Bal2009Numerical-metho00}, Theorem 5.4]
\label{norm}
Assume that $w\in \mb W^{2,p}(\mb R)$ and $m\leq p$. 
Then
$$
\l\|w_N^{(m)}-w^{(m)}\r\|_{L_2(\mb R)}\leq C(p,m) N^{-(p-m)}\l\|w\r\|_{\mb W^{2,p}(\mb R)},
$$
where $C(p,m)$ is independent of $w$ and $N$. 
\end{theorem}
Together with (\ref{result}) this implies the claim of the proposition.

\subsection{Proofs of Theorems \ref{th:main_stationary}, \ref{th:main_stationary_A}
and \ref{th:main_stationary_B}.}

Recall that for every $\lambda\in {\mb D}_{\Delta}$, for each $j\in J_{\lambda}$ we either 
have that $\lambda (t)\geq 0$ for all $t\in \mb T_j$ (in this case, set $\sigma_j=+1$),
or $\lambda (t)\leq 0$ for all $t\in \mb T_j$ (set $\sigma_j=-1$). 
Clearly, the function 
$$
w:= \sum_{j\in J_{\lambda}}\sigma_j I_{{\mb T}_j}
$$
satisfies the conditions $|w(t)|\leq 1, t\in \mb T$ and 
$w(t):={\rm sign}(\lambda(t))$ if $\lambda (t)\neq 0.$
Therefore, $w\in \partial \|\lambda\|_1.$ In what follows, 
we will use such $w$ as a version of subgradient of $\lambda\in \mb D_{\Delta}.$

We will start by providing upper bounds on RKHS-norms of $w_j, j\in J_{\lambda}.$

\begin{lemma}
\label{rkhs-b}
Suppose that, for each $j=1,\dots, N,$ the set $\mb T_j$ is contained 
in a ball $B(t_j;r)$ with some center $t_j\in \mb R^d$ and of radius $r.$
Suppose also that  
\begin{equation}
\label{lower_spectrum}
v_j(t)\geq \frac{c}{(1+|t|^2)^{p}}, t\in {\mb R}^d, j=1,\dots, N.
\end{equation}
Then 
$$
\|w_j\|_{K_j}\leq C r^{d/2}(1+r^{-p}), j\in J_{\lambda}.
$$
\end{lemma}

\begin{proof}
Note that for arbitrary functions $w_j$ defined on $\mb T_j,$ 
\begin{equation}
\label{RKHS_SOB}
\|w_j\|_{K_j}\leq C \inf_{\tilde w_j\in \Omega (w_j)}\|\tilde w_j\|_{{\mb W}^{2,p}({\mb R}^d)},
\end{equation}
where $\Omega (w_j)$ is the set of all extensions of $w_j$ onto $\mb R^d$
(see Proposition \ref{stationary_sobolev}). 
To control the RKHS-norms of $w_j,$ consider an arbitrary
nonnegative $C^{\infty}$-function $\phi$ supported in the unit ball $\{t:|t|\leq 1\}$
such that $\int_{{\mb R}^d}\phi (t)dt=1.$ Denote 
$
\phi_r (t):=r^{-d}\phi\biggl(\frac{t}{r}\biggr)
$
and let 
$$
\tilde w_j(t):= \sigma_j \int_{{\mb R^d}}\phi_r(t-s)I_{B(t_j,2r)}(s)ds
=\sigma_j (\phi_r \ast I_{B(t_j,2r)})(t),\ 
t\in \mb R^d.
$$
It is immediate that for $t\in \mb T_j,$ $\tilde w_j(t)=w_j(t),$ so, $\tilde w_j\in \Omega (w_j).$
Thus, we have 
$$
\|w_j\|_{K_j}\leq C\|\tilde w_j\|_{{\mb W}^{p,2}(\mb R^d)}
=C\|\phi_r \ast I_{B(t_j,2r)}\|_{{\mb W}^{p,2}(\mb R^d)}
\leq C' \Bigl\|(1+|t|^2)^{p/2}\widehat{\phi_r}\widehat{I_{B(t_j,2r)}}\Bigr\|_{L_2(\mb R^d)}.
$$
Since, by an easy computation, 
$$
\Bigl\|\widehat{I_{B(t_j,2r)}}\Bigr\|_{L_{\infty}}\leq \mu (B(t_j,2r))\leq c' r^d
$$
and 
$$
\Bigl\|(1+|t|^2)^{p/2}\widehat{\phi_r}\Bigr\|_{L_2(\mb R^d)}
\leq C_1 r^{-d/2}(1+r^{-p}), 
$$
we conclude that 
$
\|w_j\|_{K_j}\leq C r^{d/2}(1+r^{-p}), j\in J_{\lambda}
$
for some constant $C.$

\end{proof}

The next lemma provides bounds on $S(\mb T,d_X)$ and $\gamma_2(\delta;d_X).$

\begin{lemma}
\label{gamma_dva}
Let $\mb T$ be a bounded measurable subset of $\mb R^d$ and let $X(t), t\in \mb R^d$ be a centered subgaussian stationary random field
with spectral measure $\nu$ and spectral density $v.$
Suppose that bound (\ref{cover_T}) holds for some $R\geq 1.$
Suppose also that 
\begin{equation}
\label{upper_spectrum}
v(t)\leq \frac{B}{(1+|t|^2)^{p}}
\end{equation}
for some $p>d/2, B>0.$ 
Then, there exists a constant $C>0$ depending on $d, p, B$ such that 
$$
S(\mb T;d_X)\leq C\sqrt{\log N\vee |\log r|}
$$
and 
$$
\gamma_2(\mb T;d_X;\delta)\leq C\delta 
\sqrt{\log \frac{CR^{(p-d/2)\wedge 1}}{\delta}\bigvee \log N}.
$$ 
\end{lemma}

\begin{proof}
By the spectral representation of covariance, for all $t_1,t_2\in \mb R^d$ and for $A>0,$
\begin{align}
\label{bddx}
d_X^2(t_1,t_2)= {\Var}(X(t_1)-X(t_2)) = \int_{\mb R^d}\biggl|e^{i\langle t_1-t_2,s \rangle}-1\biggr|^2 v(s)ds \leq &
\\
\nonumber
B|t_1-t_2|^2\int_{|s|\leq A} \frac{|s|^2}{(1+|s|^2)^{p}}ds+
B\int_{|s|>A} \frac{1}{(1+|s|^2)^{p}}ds.
\end{align}
If $2p>d+2,$ we take $A=\infty$ and get 
$
d_X^2(t_1,t_2)\leq C' |t_1-t_2|^2
$
for some $C'>0$ that depends on $p$ and $d.$
If $2p=d+2,$ a simple computation of the integrals in the right hand side of (\ref{bddx})
and minimizing the resulting bound with respect to $A$ yields
$$
d_X^2(t_1,t_2)\leq C' |t_1-t_2|^2 \biggl(\log \biggl(\frac{1}{|t_1-t_2|}\biggr)\bigvee 1\biggr).
$$
Finally, if $2p<d+2,$ then a similar argument yields the bound  
$
d_X^2(t_1,t_2)\leq C' |t_1-t_2|^{2p-d}.
$
Using bound (\ref{cover_T}), it is easy to show that in each of these 
three cases we have
$$
\log N(\mb T;d_X;\eps)\leq 
C\biggl(\log\frac{CR^{(p-d/2)\wedge 1}}{\eps}\vee \log N\biggr), \eps \in (0,CR^{(p-d/2)\wedge 1}).
$$
The bound on $\gamma_2(\mb T;d_X;\delta)$ now follows by controlling the generic 
chaining complexity in terms of Dudley's entropy integral.
We also have that, under condition (\ref{upper_spectrum}), diameter $D(\mb T;d_X)$ admits the following estimate:
$$
D^2(\mb T;d_X)\leq 2\sup_{t\in \mb T}{\Var}(X(t)) =2 
\int_{\mb R^d}v(s)ds \leq C'', 
$$ 
where $C''$ is a constant depending on $d,p,B.$ Bound on $S(\mb T;d_X)$ now follows 
from the bound on $\gamma_2(\mb T;d_X;\delta)$ by substituting $\delta=\sqrt{C''}.$ 

\end{proof}

We will also need a bound on Kolmogorov's width of the set of random variables 
$X_{\mb T}$ given in the following lemma.

\begin{lemma}
\label{kolm_width}
Let $\mb T$ be a bounded measurable subset of $\mb R^d$ satisfying condition (\ref{cover_T}) and let $X(t), t\in \mb R^d$ be a centered subgaussian stationary random field
with spectral measure $\nu$ and spectral density $v.$
Suppose that 
\begin{equation}
\label{upper_spectrum2}
v(t)\leq \frac{B}{(1+|t|^2)^{p}}
\end{equation}
for some $p>d/2, B>0.$
Then, there exists a constant $C>0$ depending only on $d, p$ and $B$
such that for all $m\geq C N$
\begin{equation}
\label{ko_di}
\rho_m (X_{\mb T})\leq C \l(\frac{R}{m^{1/d}}\r)^{p-d/2}.
\end{equation}
\end{lemma}

\begin{remark}
Note that bound (\ref{ko_di}) is sharp (up to a constant). A matching lower bound can be proved via an argument based on replacing the spectral density $v$ by a smaller density $\bar v$ that is constant in a cube of a proper size and zero outside of the cube. For such a smaller density, it is possible to find a grid of points of sufficiently large cardinality $m$ such that the values of the stationary random field with spectral density $\bar v$ at the points of the grid are uncorrelated. Bounding the corresponding Kolmogorov's width from below can be 
now reduced to bounding Kolmogorov numbers of the embedding of $\ell_1^m$ into  
$\ell_2^m,$ see Gluskin \cite{Gluskin} for the solution of the last problem. The authors are very thankful to M. Lifshits for pointing out this beautiful argument. 
\end{remark}

\begin{proof}
We will construct an approximation of the set of random variables  
$X_{\mb T}=\{X(t): t\in \mb T\}$ by a finite dimensional subspace
of subgaussian random variables $L\subset {\mathcal L}_X.$ 
Since $X$ is a stationary random field, the following 
spectral representation holds 
$$
X(t)=\int_{\mb R^d} e^{i \langle t,s\rangle} Z(ds),  
$$
where $Z$ is an orthogonal random measure such that
$$
{\mb E}Z(A)\overline{Z(B)} = \nu(A\cap B), A,B\in {\mathcal B}_{\mb R^d}.
$$
By a standard isometry argument, to approximate the random variable $X(t)$
in the space $L_2(\mb P),$ it is enough to approximate the function $e^{2\pi i \langle t, \cdot\rangle}$ 
in the space $L_2({\mb R}^d, \nu).$ For $\delta\leq r,$ consider a $\delta$-net 
of the set $\mb T$ that consists of $N'\leq \Bigl(\frac{R}{\delta}\Bigr)^d$ points
$\tau_1,\dots, \tau_{N'}.$ To construct an approximation of the exponential function, 
we will use Taylor expansion of order $l$ in a  
$\delta$-neighborhood of each of the points $\tau_k.$ We use the the following standard bound 
on the remainder of Taylor expansion: 
\begin{equation}
\label{exp_appr}
\Bigl|e^{i \langle h, s\rangle}-Q_{l}(h;s)\Bigr|\leq \frac{|h|^{l}|s|^{l}}{l!},\ \
Q_l(h;s):=\sum_{j=0}^{l-1} \frac{i^j \langle h, s\rangle^j }{j!}.
\end{equation}
For $t\in B(\tau_k;\delta),$ 
$$e^{i \langle t, s\rangle}
=e^{i \langle \tau_k, s\rangle}e^{i \langle t-\tau_k, s\rangle}=
e^{i \langle \tau_k, s\rangle}Q_{l}(t-\tau_k;s) +
e^{i \langle \tau_k, s\rangle}(e^{i \langle t-\tau_k, s\rangle}-Q_{l}(t-\tau_k;s)).
$$
Denote (for some $A>0$ to be chosen later) 
$$
\zeta_l^{(k)}(h):= {\rm Re}\biggl(
\int_{\mb R^d}e^{i \langle \tau_k, s\rangle}Q_l(h;s)I(|s|\leq A\delta^{-1})Z(ds)\biggr).
$$
By spectral isometry (using the fact that $X$ is real valued), we get that for all $k=1,\dots, N'$ and all $t\in B(\tau_k;\delta)$ (thus, for all $t\in \mb T$) 
\begin{align}
\label{isom}
{\mb E}\Bigl|X(t)-\zeta_l^{(k)}(t-\tau_k)\Bigr|^2
\leq 
{\mb E}\biggl|X(t)-\int_{\mb R^d}e^{i \langle \tau_k, s\rangle}Q_l(h;s)I(|s|\leq A\delta^{-1})Z(ds)\biggr|^2 \leq &
\\
\nonumber 
\int\limits_{|s|\leq A\delta^{-1}}\Bigl|e^{i \langle t, s\rangle}-
e^{i \langle \tau_k, s\rangle}Q_{l}(t-\tau_k;s)\Bigr|^2 v(s)ds
+
\int_{|s|>A\delta^{-1}}v(s)ds.
\end{align}
Under condition (\ref{upper_spectrum2}) and the assumption $p>\frac{d}{2},$ using (\ref{exp_appr}), we get that with some constant $C>0$ depending only 
on $B,d $ and for all $k=1,\dots, N'$ and $l\geq (2p-d)\vee 1$ 
$$
\int_{|s|\leq A\delta^{-1}}\Bigl|e^{i \langle t, s\rangle}-
e^{i \langle \tau_k, s\rangle}Q_{l}(t-\tau_k;s)\Bigr|^2 v(s)ds\leq \frac{\delta^{2l}}{(l!)^2}
\int_{|s|\leq A\delta^{-1}}|s|^{2l}v(s)ds \leq 
$$
$$
B\frac{\delta^{2l}}{(l!)^2}
\int_{|s|\leq A\delta^{-1}} \frac{|s|^{2l}}{(1+|s|^2)^{p}}ds
\leq C\frac{\delta^{2p-d}}{A^{2p-d-2l}(2l-2p+d)(l!)^2}.
$$
We also have 
$$
\int_{|s|>A\delta^{-1}}v(s)ds\leq B\int_{|s|>A\delta^{-1}}\frac{1}{(1+|s|^2)^{p}}ds
\leq C\frac{\delta^{2p-d}}{(2p-d)A^{2p-d}}.
$$
We will now set 
$$
A:=A_l := (2l)^{1/(2l)}\Bigl(l!\Bigr)^{1/l}. 
$$
Then, (\ref{isom}) easily implies that with some constant $C$ depending only on $p$ and $d$
$$
{\mb E}\Bigl|X(t)-\zeta_l^{(k)}(t-\tau_k)\Bigr|^2 \leq 
C\biggl(\frac{\delta}{A_l}\biggr)^{2p-d}.
$$
Using Stirling's approximation, it is easy to see that $A_l\geq \frac{l}{e},$ 
implying that 
\begin{equation}
\label{approx_X}
{\mb E}\Bigl|X(t)-\zeta_l^{(k)}(t-\tau_k)\Bigr|^2 \leq 
C\biggl(\frac{\delta}{l}\biggr)^{2p-d}.
\end{equation}

Note that $Q_l(h;\cdot)$ is polynomial of degree $l-1$ of $d$ variables, hence, the family of functions 
$$
\Bigl\{e^{i \langle \tau_k, \cdot\rangle}Q_l(h;\cdot)I(|\cdot|\leq A\delta^{-1}):h\in \mb R^d\Bigr\}
$$ 
belongls to a (complex) linear space of dimension 
$
{l-1+d \choose d} \leq (l+d-1)^d. 
$
This immediately implies that the family of random variables $\{\zeta_l^{(k)}(h): h\in \mb R^d\}$ belongs to a linear subspace of 
${\mathcal L}_X$ whose dimension is at most $2(l+d-1)^d.$ 
Therefore, $\Bigl\{\zeta_l^{(k)}(t-\tau_k): t\in B(\tau_k;\delta), k=1,\dots,N'\Bigr\}$ belongs to a subspace of $\mathcal L$ of dimension 
$\leq 2(l+d-1)^d N' \leq 2(l+d-1)^d\Bigl(\frac{R}{\delta}\Bigr)^d$. 
Let $m \geq 2 (l+d-1)^d$ and let 
$$
\delta= 2^{1/d}(l+d-1)\frac{R}{m^{1/d}}.
$$
Assuming that $m\geq C_1 N,$ where $C_1:=2(l+d-1)^d \kappa^d,$
we have $\delta \leq r.$
Then $2(l+d-1)^d\Bigl(\frac{R}{\delta}\Bigr)^d=m$ and it follows from (\ref{approx_X}) that 
$$
\rho_m (X_{\mb T})\leq C\biggl(\frac{l+d-1}{l}\biggr)^{p-d/2} \frac{R^{p-d/2}}{m^{p/d-1/2}},
$$
with some constant $C$ depending on $B,d,p.$ The claim of the lemma follows by substituting 
the smallest $l\geq (2p-d)\vee d.$

\end{proof}

We will now provide an upper bound on the ``approximate dimension'' $d(w;\lambda)$ needed to complete the proof of the theorem. 
To this end, recall that we assume that for all $j=1,\dots, N$ the set $\mb T_j$ belongs 
to a ball of radius $r\geq N^{-1/d}$ and $R=\kappa N^{1/d}r,\ R\geq 2.$
Also, for an oracle $\lambda,$ $R(\lambda)=\kappa(N(\lambda))^{1/d}r,$ so, we have $r\leq R(\lambda)\leq R.$ 
In what follows, $C, C',$ etc are constants depending on $B,d,p.$
First, let us upper bound $\gamma_2(\rho_m(w))=\gamma_2(\rho_m(X_{T_{w}})).$ 
Using Lemmas \ref{gamma_dva} and \ref{kolm_width}, we get that for all $m\geq C_1 N(\lambda)$
\begin{align}
\label{eq:gamma_2}
\gamma_2(\rho_m(w))\leq C\frac{(R(\lambda))^{p-d/2}}{m^{p/d-1/2}}
\sqrt{\log \biggl(\frac{CR^{p-d/2}m^{p/d-1/2}}{(R(\lambda))^{p-d/2}}\biggr)\bigvee \log N}. 
\end{align}
Since $\frac{R}{R(\lambda)}=\frac{\kappa N^{1/d}r}{\kappa (N(\lambda))^{1/2}r}\leq N^{1/d},$
it is easy to conclude that 
$$
\gamma_2(\rho_m(w))\leq 
C\frac{(R(\lambda))^{p-d/2}}{m^{p/d-1/2}}\sqrt{\log m}\bigvee C\frac{(R(\lambda))^{p-d/2}}{m^{p/d-1/2}}\sqrt{\log N}. 
$$
To provide an upper bound on $d(w,\lambda),$ we first find the smallest 
$m$ satisfying the inequality 
$$
\frac{\sigma_Y^2 m}{n}
\geq C\frac{\|\lambda\|_1}{\sqrt{n}}\frac{(R(\lambda))^{p-d/2}}{m^{p/d-1/2}}\sqrt{\log m}
\bigvee
C\frac{\|\lambda\|_1}{\sqrt{n}}\frac{(R(\lambda))^{p-d/2}}{m^{p/d-1/2}}\sqrt{\log N}.
$$
This is equivalent to the bound 
\begin{align}
\label{belowm}
m\geq  C\frac{n^{d/(2p+d)}\|\lambda\|_1^{2d/(2p+d)}R(\lambda)^{d(2p-d)/(2p+d)}}
{\sigma_Y^{4d/(2p+d)}}\Bigl((\log m)^{d/(2p+d)}\bigvee (\log N)^{d/(2p+d)}\Bigr)
\end{align}
Note that in the oracle inequality of Theorem \ref{th:main_AA}, it is enough 
to restrict oracles $\lambda$ to the ball  
\begin{equation}
\nonumber
\|\lambda\|_1 \leq 
C'\|f_{\ast}-\Pi f_{\ast}\|_{L_2(\Pi)} n^{1/2}
\end{equation}
for some constant $C'>0$ (see bound (\ref{barlambda_2}) in the proof of this theorem).
Recall that also 
$$N^{-1/d}\leq r\leq R(\lambda)\leq R=\kappa N^{1/d}r.$$ 
Therefore, bound (\ref{belowm})
easily implies that  
\begin{equation}
\label{belm}
m\geq C\frac{(n\|\lambda\|_1^2)^{d/(2p+d)}R(\lambda)^{d(2p-d)/(2p+d)}}
{\sigma_Y^{4d/(2p+d)}}
\Bigl(\log N\vee \log n \vee |\log \sigma_Y|\vee |\log r|\Bigr)^{d/(2p+d)}.
\end{equation}
It easily follows from the definition of $d(w,\lambda)$ that either we have $d(w,\lambda)\leq C_1 N(\lambda),$ or $d(w,\lambda)\leq m$ for any $m$ satisfying (\ref{belm}).
Therefore, with some constant $C>0$
\begin{align}
\label{eq:dim}
& d(w;\lambda) \leq CN(\lambda)\bigvee \\  
\nonumber
& C\frac{(n\|\lambda\|_1^2)^{d/(2p+d)}R(\lambda)^{d(2p-d)/(2p+d)}}
{\sigma_Y^{4d/(2p+d)}}
\Bigl(\log N\vee \log n \vee |\log \sigma_Y|\vee |\log r|\Bigr)^{d/(2p+d)}.
\end{align}
To complete the proof, it is enough to substitute this bound on $d(w;\lambda)$
in the oracle inequality of Theorem \ref{th:main_AA}. 
Bounds (\ref{a_beta}) and 
Lemma \ref{rkhs-b} should be used to control the alignment coefficient ${\fr a}(w).$

As to the proof of Theorem \ref{th:main_stationary_A}, the main difference is in the bounds on the alignment coefficient ${\fr a}(w).$ 
For a given oracle $\lambda\in {\mb D}_r$ and a covering  
$B(t_1;r), \dots, B(t_{N(\lambda)};r)$ of $\supp(\lambda)$, let 
$\sigma_j$ be the sign of $\lambda$ on $B(t_j;r)\cap \supp(\lambda)$ and  
$
\tilde w_j :=\sigma_j (\phi_r\ast I_{B(t_j;2r)}), j=1,\dots, N(\lambda)
$
(see the notations of the proof of Lemma \ref{rkhs-b}). 
It is easy to see that $\sum\limits_{j\in J_{\lambda}}\tilde w_j$ is 
an extension of a subgradient $w\in \partial\|\lambda\|_1.$ 
Thus, by Proposition \ref{stationary_sobolev},
$$
{\fr a}^2(w)\leq \|w\|_K^2\leq \biggl\|\sum_{j=1}^{N(\lambda)}\tilde w_j\biggr\|_{{\mb W}^{2,p}}^2.
$$ 
Since functions $\tilde w_j$ have disjoint support, we can further bound this using Proposition \ref{additive_sob} and of Lemma \ref{rkhs-b} as  
\begin{align*}
{\fr a}^2(w)\leq C\l[\sum_{j=1}^{N(\lambda)}\|\tilde w_j\|_{{\mb W}^{2,p}}^2+\frac{1}{r^{2\alpha}}\sum_{j=1}^{N(\lambda)}\|\tilde w_j\|_{{\mb W}^{2,\lfloor p\rfloor}}^2\r]
\leq C\l(r^{d}+r^{d-2p}\r) N(\lambda).
\end{align*}

Finally, to prove the result of Theorem \ref{th:main_stationary_B}, we need to bound the alignment coefficient $\fr a(w)$ as follows. 
Let $\phi$ be an arbitrary nonnegative $C^{\infty}$-function supported in the unit ball $\{t:|t|\leq 1\}$ 
such that for all $t\in \mb R^d$, $\phi(t)\leq \phi(0)=1$.
Given $\lambda\in \bar{\mb D}$ and $r\leq \delta(\lambda)$, let 
$$
\tilde w_j=\sign(\lambda_j) \phi\l(\frac{t-t_j}{r}\r), \ j\in J(\lambda).
$$
Clearly, restriction of $w=\sum\limits_{j\in J(\lambda)}\tilde w_j$ to the grid $\m G_N$ is an element of $\partial \|\lambda\|_1$. 
By Proposition \ref{interpolation}, we have 
$$
{\fr a}^2(w)\leq \|w\|_K^2\leq \biggl\|\sum_{j=1}^{N(\lambda)}\tilde w_j\biggr\|_{{\mb W}^{2,p}}^2.
$$
By Proposition \ref{additive_sob} and a simple computation is spirit of Lemma \ref{rkhs-b}
$$
{\fr a}^2(w)\leq C\l(r^{d}+r^{d-2p}\r) N(\lambda).
$$
It is easy to see that $\gamma_2(\rho_m(w))$ and $d(w;\lambda)$ can be bounded above by their ``continuous'' counterparts for $\mb T=[0,2\pi]^d$, in particular, inequalities (\ref{eq:gamma_2}) and (\ref{eq:dim}) hold. 
To complete the proof, it is enough to substitute bounds on $d(w;\lambda)$ and $\fr a^2(w)$
in the oracle inequality of Theorem \ref{th:main_AA} and optimize the resulting expression with respect to $r$. 
Choose $r_\ast(\lambda)$ as $r_\ast(\lambda)=\min(\tilde r, \delta(\lambda))$ with $\tilde r$ defined as
$$
\tilde r^{2p-d}=
\l(\frac{N(\lambda)^2}{n}\r)^{\frac{d}{2p+2d}}
\frac{\sigma_Y^{\frac{2d}{p+d}}s^{\frac{2p+d}{2p+2d}}}{\l(L\|\lambda\|_1^2\r)^{\frac{d}{2p+2d}}},
$$
where $L=\log n\vee\log N\vee |\log \sigma_Y|.$ 
The claim now follows from simple algebra.

\begin{acknowledgements}
The authors are very thankful to Mikhail Lifshits and Mauro Maggioni for insightful discussions and their valuable input. 
\end{acknowledgements}

\bibliographystyle{abbrv}
\bibliography{bibliography}
\nocite{van2008reproducing}
\nocite{Ledoux1991Probability-in-00}
\nocite{Koltchinskii2011Oracle-inequali00}
\nocite{Bogachev2007Measure-theory.00}
\nocite{james2009functional}
\nocite{ramsay2006functional}
\nocite{ramsay2002applied}
\nocite{crambes2009smoothing}
\nocite{yuan2010reproducing}
\nocite{cai2006prediction}

\appendix

\section{Technical background and remaining proofs.}

\subsection{Existence of solutions of optimization problems.}

We provide below sufficient conditions for existence of solutions to the problems (\ref{true}) and (\ref{empirical}).

\begin{theorem}
\label{existence}
Let $\mb D$ be a convex, weakly compact subset of $L_1(\mu)$. 
Then 
\begin{enumerate}
\item $F(\lambda,a), \ F_n(\lambda,a)$ are weakly lower semicontinous;
\item Solutions to problems (\ref{true}) and (\ref{empirical}), denoted by $\lambda_\eps$ and $\hat\lambda_\eps$, exist.
\end{enumerate}
\end{theorem}

\begin{proof}
We prove the statement for $F(\lambda)$, and the result for $F_n(\lambda)$ follows similarly. 
The functional $\lambda\mapsto \|\lambda\|_1$ is continuous.
Assume $\|\lambda_k-\lambda_0\|_1\to 0$. 
Using H\"{o}lder's inequality, we get
\begin{align*}
& P(\ell \bullet f_{\lambda_k,a(\lambda_k)})-P(\ell \bullet f_{\lambda_0,a(\lambda_0)})=
\mb E(Y-f_{\lambda_k,a(\lambda_k)}(X))^2-\mb E(Y-f_{\lambda_0,a(\lambda_0)}(X))^2\\
&
=\mb E\l[\l(2Y-f_{\lambda_k,a(\lambda_k)}(X)-f_{\lambda_0,a(\lambda_0)}(X)\r)\int\limits_\mb T (\lambda_0-\lambda_k)(X-\mb EX)d\mu  \r] \\
&
\leq \mb E^{1/2}\l(\int\limits_\mb T (\lambda_0-\lambda_k)(X-\mb EX)d\mu\r)^2 \mb E^{1/2}\l(2Y-(f_{\lambda_k,a(\lambda_k)}+f_{\lambda_0,a(\lambda_0)})(X)\r)^2 \\
&
\leq 
\|\lambda_k-\lambda_0\|_1 \mb E^{1/2} \|X-\mb EX\|^2_{\infty}\l(2\sqrt{\Var(Y)}+\|\lambda_k+\lambda_0\|_1\mb E^{1/2} \|X-\mb EX\|^2_{\infty}\r)\to 0,
\end{align*}
where in the last step we used the fact that 
$$
\l|\int\limits_\mb T \lambda(t) (X(t)-\mb EX(t))\mu(dt)\r|\leq \|\lambda\|_1 \sup\limits_{t\in\mb T} |X(t)-\mb EX(t)|.
$$
Thus, $F(\lambda)$ is continuous, hence it is lower semi-continuos. 
In turn, this is equivalent to the fact that the level sets $\m L_t=\l\{\lambda: F(\lambda)\leq t\r\}$ are closed. 
Moreover, they are convex since $F$ is. 
Mazur's theorem (see \cite{lang1993real}, Theorem 2.1) implies that they are also closed in weak topology, so $F$ is weakly lower semi-continuos. \\
Now it is easy to show existence of solutions. 
Given a minimizing sequence $\l\{\lambda_k\r\}\subset \mb D$, we can extract a weakly convergent subsequence 
$$
\lambda_{k_l}{\buildrel\sigma\over\longrightarrow}\lambda_{\infty}.
$$ 
It remains to note that by weak compactness and lower semi-continuity
$\lambda_{\infty}\in \mathbb{D}$ and $ -\infty<F(\lambda_{\infty})\leq \liminf\limits_{l\to\infty} F\left(\lambda_{k_l}\right), $ which means that $\lambda_\infty$ is the solution. 

\end{proof}
When $\mb T$ is finite, then one can clearly take $\mb D=L_1(\mb T,\mu)\subseteq\mb R^{|T|}$, and Theorem \ref{existence} is not needed to prove existence of 
$\hat\lambda_\eps$.
However, in general the unit ball in $L_1(\mb T,\mu)$ is not weakly compact, so one way to proceed is to choose $\mb D$ to be uniformly integrable (which implies weak compactness, see Theorem 4.7.18 in \cite{Bogachev2007Measure-theory.00}). 
A possible choice is 
$$
\mb D=\l\{\lambda: \ \l|\int\limits_\mb T \max\l(|\lambda(t)|\log |\lambda(t)|,0\r)d\mu(t)\r|\leq L\r\} \text{ for some } L>0.
$$

\subsection{Orlicz norms.}
\label{orlicz_n}
Let $\psi: \mb R_+\mapsto \mb R_+$ be a convex nondecreasing function with $\psi(0)=0$.
\index{bh@Orlicz norm $\|\cdot\|_\psi$}
\begin{definition}\label{orlicz}
The Orlicz norm of a random variable $\eta$ on a probability space $(\Omega, \Sigma, {\mb P})$ is defined via
$$
\l\|\eta\r\|_{\psi}:=\inf\left\{C>0: \ \mb E\psi\left(\frac{|\eta|}{C}\right)\leq 1 \right\}
$$
\end{definition}

By $\|\cdot\|_{\psi_1}, \ \|\cdot\|_{\psi_2}$ we denote the Orlicz norms for 
$\psi_1(x):=e^x-1$ and $\psi_2(x):=e^{x^2}-1$, respectively;  the following inequalities are elementary: 

\begin{align}
\label{i1}
&
\|\eta\|_{\psi_1}\leq \sqrt{\log 2}\|\eta\|_{\psi_2},\\
\label{i2}
&
\|\eta^2\|_{\psi_1}=\|\eta\|^2_{\psi_2}, \\
& \label{i3}
\|\xi\eta\|_{\psi_1}\leq \|\xi\|_{\psi_2}\|\eta\|_{\psi_2}.
\end{align}
It is easy to check from the definition that every subgaussian random variable $\eta$ 
(meaning that $\mb E e^{s\eta}\leq e^{\Gamma\sigma^2_\eta s^2}, \ s\in\mb R$)
satisfies the following property:
\begin{align}
\label{p1}
\l\|\eta\r\|^2_{\psi_2}\leq 8\Gamma\sigma^2_\eta.
\end{align}

In what follows, we use the same notations for Orlicz norms on other probability spaces 
(for instance, $C_{bu}(\mb T;d_X)$ with its Borel $\sigma$-algebra and probability 
measure $\Pi$). 

\subsection{Bounds for subgaussian processes
and Talagrand's generic chaining complexities.} 
\label{sec:dudley}
\begin{theorem}
\label{tal1}
Let $\l\{Z(t), \ t\in \mb T\r\}$ be a centered subgaussian process. 
Then, for all $u\geq0, \ t_0\in \mb T$,
\begin{align*}
&
1) \ \mb P\l(\sup_t \l(Z(t)-Z(t_0)\r)\geq 2u\cdot\gamma_2(\mb T,d_Z)\right)\leq Ce^{-u^2/4}, \\
&
2) \  \mb E\|Z\|_\infty\leq \mb E|Z(t_0)|+L\gamma_2(\mb T,d_Z),
\end{align*}
where $d_Z(t,s)=\sqrt{\Var(Z(t)-Z(s))}$.
\end{theorem}
\begin{proof}
See Chapter 1.2 in \cite{talagrand2005generic}.

\end{proof}
A simple corollary is the following inequality:
\begin{align}
\label{g33}
\mb P\l(\|Z\|_\infty\geq C\sqrt{t}\l(\gamma_2(\mb T,d_Z)+\inf_{t\in \mb T}\sqrt{\Var(Z(t))}\r)\r)\leq e^{-t}.
\end{align}
We mention another result which is useful in our investigation:
\begin{proposition}
\label{psi_2}
Let $Z$ be a centered subgaussian stochastic process such that 
$$
\gamma_2(\mb T,d_Z)<\infty
$$ 
and let $Z_1,\ldots,Z_n$ be iid copies of $Z$.
Then for any $t_0\in \mb T$
\begin{align*}
&
1) \ \l(\log 2\r)^{-1/2}\big\|\|Z\|_\infty\big\|_{\psi_1}\leq\big\|\|Z\|_\infty\big\|_{\psi_2}\leq  \|Z(t_0)\|_{\psi_2}+L\gamma_2(\mb T,d_Z),\\
&
2) \ \big\|\max_{j=1\ldots n}\|Z_j\|_\infty\big\|_{\psi_1}\leq C\log n \big\|\|Z\|_\infty\big\|_{\psi_1}.
\end{align*}
\end{proposition}
\begin{proof}
First statement is a straightforward corollary of Talagrand's result and integration-by-parts formula. 
For the proof of the second claim, see \citep{Vaart1996Weak-convergenc00}, Lemma 2.2.2.

\end{proof}

In the case when $Z(t), t\in {\mb T}$ is a centered Gaussian process,
a famous result of Talagrand (see Theorem 2.1.1 in  \cite{talagrand2005generic})  states that
\begin{align}
\label{tal2}
&
\frac{1}{K}\gamma_2(\mb T;d_Z)\leq \mb E\sup_{t\in \mb T}Z(t)\leq
K \gamma_2(\mb T;d_Z)
\end{align}
for some universal constant $K$. 
Moreover, the upper bound also holds for the centered subgaussian process $Z.$

In practice, a useful way to estimate the generic chaining complexity $\gamma_2(\mb T;d_Z)$ and its ``local version'' $\gamma_2(\delta)$ 
is to evaluate Dudley's entropy integral:
\begin{theorem}
\label{dudley}
The following inequality holds for all $\delta\leq \sup\limits_{t,s\in \mb T} d_X(t,s)$:
$$
\gamma_2(\delta)\leq (2\sqrt 2-1)^{-1}\int\limits_0^{\delta}\sqrt{\log N(\mb T,d_Z,\eps/4)}d\eps,
$$
where $N(\mb T,d_Z,\eps)$ is the minimal number of balls of radius $\eps$ required to cover $\mb T$.
\end{theorem}
\begin{proof}
This well-known bound can be obtained by repeating the argument of Proposition 1.2.1 in \cite{talagrand2005generic}.
\end{proof}

The following immediate corollary covers two important examples.
\begin{corollary} 
\label{entropy}
$\text{ }$
\begin{enumerate}
\item 
If $\card(\mb T)=N$, then 
$$
\gamma_2(\delta)\leq C\delta \sqrt{\log N};
$$
\item 
If the covering numbers grow polynomially, i.e. $N(\mb T,d_X,\eps)\leq C_1\l(\frac{A}{\eps}\r)^V$, then
$$
\gamma_2(\delta)\leq C_2 \delta  \sqrt{V\log\frac A \delta}.
$$
\end{enumerate}
\end{corollary}

\subsection{Empirical processes.}
We state a version of generic chaining bounds for empirical processes due to S. Mendelson, S. Dirksen and W. Bednorz which are used in our proofs. 
Let ${\mathcal F}$ be a class of functions defined 
on a measurable space $(S,\mathcal A).$ 
Suppose ${\mathcal F}$ is symmetric, that is, $f\in {\mathcal F}$ implies $-f\in {\mathcal F}$ (in applications, we often deal with the classes that do not satisfy this assumption and then replace $\mathcal F$ by ${\mathcal F}\cup {-\mathcal F}$). 
Let $(X,\xi), (X_1,\xi), \dots, (X_n,\xi)$ be i.i.d. random variables with values in $S\times \mb R$ such that ${\mb E}f(X)=0, f\in {\mathcal F}$ and $\xi$ is a subgaussian random variable. 
Let $\Pi$ be the marginal distribution of $X.$ It will be used as a measure on $(S,{\mathcal A}).$

\begin{theorem}
\label{th:mendelson1}
There exists an absolute constant $C>0$ such that 
$$
\mb E \sup_{f\in {\mathcal F}}\biggl|n^{-1}\sum_{j=1}^n \xi_j f(X_j)-\mb E\xi f(X)\biggr|
\leq C \biggl[\|\xi\|_{\psi_2}\frac{\gamma_2({\mathcal F};\psi_2)}{\sqrt{n}}\bigvee 
\frac{\gamma_2^2({\mathcal F};\psi_2)}{n}
\biggr].
$$
\end{theorem}

This inequality follows from Corollary 3.9 in \cite{mendelson2012oracle}. 
We will often combine it with a version of Talagrand's concentration inequality for unbounded function classes due to Adamczak \cite{law2008tail}
(stated in a convenient form for our purposes).  
Let ${\mathcal F}$ be a class of functions defined on a measurable space $(S,\mathcal A)$ and let $X,X_1,\dots, X_n$ be i.i.d. random variables 
sampled from distribution $P$ on $(S,{\mathcal A}).$ Let $F$ be a measurable 
envelope for ${\mathcal F},$ that is $F$ is a measurable function on $S$
such that $|f(x)|\leq F(x), x\in S, f\in {\mathcal F}.$
Then, there exists a universal constant $K>0$ such that 
\begin{align}
\label{adamczak}
&
\sup_{f\in {\mathcal F}}\l|\frac 1 n \sum_{j=1}^n f(X_j)-{\mb E}f(X)\r|\leq 
K\biggl[{\mb E}\sup_{f\in {\mathcal F}}\l|\frac 1 n \sum_{j=1}^n f(X_j)-{\mb E}f(X)\r|
\\ 
&
\nonumber
+\sup_{f\in {\mathcal F}}\sqrt{\Var(f(X))}\sqrt{\frac{s}{n}}+
\biggl\|\max_{1\leq j\leq n}|F(X_j)|\biggr\|_{\psi_1} \frac{s}{n}\biggr]
\end{align}
with probability $\geq 1-e^{-s}$.

Finally, we state a recent sharp bound for the empirical processes due to S. Dirksen \cite{dirksen2013tail} and W. Bednorz \cite{bednorz2014} (earlier versions
of exponential generic chaining bounds for similar empirical processes are due to Mendelson \cite{mendelson2010empirical}, \cite{mendelson2012oracle}). 
Assume that $\{f(X), \ f\in \m F\}$ is a subset of the subgaussian space $\m L$. 
\begin{theorem}
\label{th:dirksen1}
There exists an absolute constant $C>0$ such that 
\begin{align*}
\sup_{f\in {\mathcal F}}\biggl|n^{-1}\sum_{j=1}^n f^2(X_j)-{\mb E}f^2(X)\biggr|
\leq 
C &\bigg[\sup\limits_{f\in \m F}\|f\|_{\psi_2}\frac{\gamma_2(\m F;\psi_2)}{\sqrt n}+\frac{\gamma_2^2({\mathcal F};\psi_2)}{n}\\
&
+\sup\limits_{f\in \m F}\|f\|^2_{\psi_2}\Big(\sqrt{\frac s n}\vee\frac{s}{n}\Big)\bigg]
\end{align*}
with probability $\geq 1-e^{-s}$. 
\end{theorem}
For a proof and discussion of this bound, see Theorem 5.5 in \cite{dirksen2013tail}. 
Note that in (ii), the generic chaining complexity $\gamma_2(\m F;\psi_2)$ in the right-hand side is for the class $\m F$ itself rather than $\m F^2$. 

\subsection{Sobolev norms.}

\index{bi@Sobolev space $\mb W^{2,p}(\mb R^d)$}
For any $p\in \mb R_+$, define the Sobolev space $\mb W^{2,p}(\mb R^d)$ as 
$$
\mb W^{2,p}(\mb R^d)=\l\{f\in L^2(\mb R^d): \ \|f\|^2_{\mb W^{2,p}(\mb R^d)}:=\int_{\mb R^d} (1+|t|^2)^{p}|\hat f(t)|^2 dt<\infty\r\},
$$
where $\hat f$ is the Fourier transform of $f$. 
It is well known that for $p\in \mb Z_+$, this coincides with another definition of Sobolev spaces (in terms of partial derivatives). 

Assume that $f\in \mb W^{2,p}(\mb R^d)$ for $p\in \mb Z_+$ is such that $f=\sum\limits_{j=1}^k f_j$, where $f_j,\ j=1\ldots k$ have disjoint supports. 
Clearly, in this case we have $\|f\|^2_{\mb W^{2,p}(\mb R^d)}=\sum\limits_{j=1}^k \|f_j\|^2_{\mb W^{2,p}(\mb R^d)}$.
When $p$ is not an integer, we will use the following proposition. 
\begin{proposition}
\label{additive_sob}
Assume that $p\in \mb R_+$, $\alpha:=p-\lfloor p\rfloor>0$ and $f\in \mb W^{2,p}(\mb R^d)$ is such that 
$f=\sum\limits_{j=1}^k f_j$, where $f_j,\ j=1\ldots k$ have disjoint supports and 
$\min\limits_{1\leq i<j\leq k}{\dist}(\supp(f_i),\supp(f_j))\geq r>0$, where $\dist$ is the Euclidean distance. 
Then
$$
\|f\|^2_{\mb W^{2,p}(\mb R^d)}\leq C(d,p)\l[\sum\limits_{j=1}^k \|f_j\|^2_{\mb W^{2,p}(\mb R^d)}
+\frac{1}{r^{2\alpha}}\sum\limits_{j=1}^k \|f_j\|^2_{\mb W^{2,\lfloor p\rfloor}(\mb R^d)}\r].
$$
\end{proposition}
\begin{proof}
We will need to use equivalence of certain norms defined on Sobolev spaces $\mb W^{2,p}(\mb R^d)$. 
Let 
$$
'\|f\|^2_{\mb W^{2,p}(\mb R^d)}:=\|f\|^2_{\mb W^{2,\lfloor p\rfloor}}+
\max_{|m|=\lfloor p\rfloor}\iint\limits_{\mb R^d\times \mb R^d}\frac{\l(\partial^m f(x)-\partial^m f(y)\r)^2}{|x-y|^{2\alpha+d}}dxdy,
$$
where we use the usual multi-index notation. 
It is known that $\|\cdot\|_{\mb W^{2,p}}$ and $'\|\cdot\|_{\mb W^{2,p}}$ are equivalent (e.g., see \cite{adams1975sobolev} p. 219). 
Let $f=\sum\limits_{j=1}^k f_i$ satisfy conditions of the proposition. 
Making the change of variables $x=t+u, y=u$ in the expression of $'\|\cdot\|_{\mb W^{2,p}(\mb R^d)}$, we have 
\begin{align*}
'\|f\|_{\mb W^{2,p}}^2&=\|f\|^2_{\mb W^{2,\lfloor p\rfloor}} +
\max_{|m|=\lfloor p\rfloor} \iint\limits_{\mb R^d\times \mb R^d}\l(\partial^m f(t+u)-\partial^m f(u)\r)^2du\frac{dt}{|t|^{2\alpha+d}}
\\
&=\sum_{j=1}^k \|f_j\|^2_{\mb W^{2,\lfloor p\rfloor}}+
\max_{|m|=\lfloor p\rfloor}\Bigg[\iint\limits_{\mb R^d\times B(0,r)}\l(\partial^m f(t+u)-\partial^m f(u)\r)^2du\frac{dt}{|t|^{2\alpha+d}} \\
&
+\iint\limits_{\mb R^d\times \bar B(0,r)}\l(\partial^m f(t+u)-\partial^m f(u)\r)^2du\frac{dt}{|t|^{2\alpha+d}}\Bigg].
\end{align*}
It remains to notice that 
\begin{align*}
\iint\limits_{\mb R^d\times B(0,r)}\big(\partial^m f(t+u)-&\partial^m f(u)\big)^2du\frac{dt}{|t|^{2\alpha+d}} \\
&
=\sum_{j=1}^k \iint\limits_{\mb R^d\times B(0,r)}\big(\partial^m f_j(t+u)-\partial^m f_j(u)\big)^2du\frac{dt}{|t|^{2\alpha+d}}\\
&
\leq \sum_{j=1}^k \iint\limits_{\mb R^d\times \mb R^d}\frac{\l(\partial^m f_j(x)-\partial^m f_j(y)\r)^2}{|x-y|^{2\alpha+d}}dxdy
\end{align*}
and, since $(f(t+u)-f(u))^2\leq 2 f^2(t+u)+2f^2(u)$, 
\begin{align*}
\iint\limits_{\mb R^d\times \bar B(0,r)}\l(\partial^m f(t+u)-\partial^m f(u)\r)^2du\frac{dt}{|t|^{2\alpha+d}}& 
\leq 
2\|\partial^m f\|^2_{L_2(\mb R^d)}\int_{|t|\geq r}\frac{dt}{|t|^{d+2\alpha}}\\ 
&
=C_1(d,\alpha)\frac{\|\partial^m f\|^2_{L_2(\mb R^d)}}{r^{2\alpha}},
\end{align*}
where $C_1(d,\alpha)=\frac{2\pi^{d/2}(d+2\alpha)}{\alpha\Gamma(d/2)}$, and the claim easily follows. 

\end{proof}

\subsection{Proof of Proposition \ref{prop:rip}.}
\label{sec:rip}
Let $J_1,J_2$ be two disjoint subsets of $\l\{1,\ldots,N\r\}$, and define
$$
r(J_1;J_2):=\sup_{u,v} \l|\frac{\dotp{\sum\limits_{j\in J_1} f_{u_j}}{\sum\limits_{j\in J_2} f_{v_j}}_{L_2(\Pi)}}
{\sqrt{\sum\limits_{j\in J_1} \|f_{u_j}\|_{L_2(\Pi)}^2\sum\limits_{j\in J_2} \|f_{v_j}\|_{L_2(\Pi)}^2}}\r|  
$$ 
where the supremum is taken over all $u=\sum\limits_{j\in J_1}u_j, \ v=\sum\limits_{j\in J_2}v_j$ such that $\supp(u_j)\subseteq \mb T_j$, $\supp(v_j)\subseteq \mb T_j$ and $\sum\limits_{j\in J_1} \|f_{u_j}\|_{L_2(\Pi)}^2\ne 0, \ \sum\limits_{j\in J_2} \|f_{v_j}\|_{L_2(\Pi)}^2\ne 0$.

Next, let 
$
\rho_d:=\max\l\{ r(J_1;J_2): \ J_1\cap J_2=\emptyset, \ \card(J_1)+\card(J_2)\leq 3d\r\}.
$
In what follows, we set $\lambda(u_j):=\l\|f_{u_j}\r\|_{L_2(\Pi)}$ and $h(u_j):=\frac{f_{u_j}}{\l\|f_{u_j}\r\|_{L_2(\Pi)}}$.
\begin{lemma}
The following inequality holds: $\rho_d\leq \delta_{3d}$. 
\end{lemma}
\begin{proof}
See Lemma 2.1 in \cite{candes2008restricted}.

\end{proof}
Set $J_0:=J$, $\lambda^{(0)}:=\l\{\lambda(u_j), \ j\in J_0\r\}$, and let 
$(\lambda(u_{\pi(1)}),\ldots, \lambda(u_{\pi_{N-d}}))$ be the vector $\l(\lambda(u_j), \ j\in J_0^c\r)$ sorted in the decreasing order, so that $\pi$ is some permutation. 
We further define
$J_1:=(\pi(1),\ldots,\pi(d))$, $J_2:=(\pi(d+1),\ldots,\pi(2d))$, etc., and $\lambda^{(k)}=(\lambda(u_j), \ j\in J_k)$.
Everywhere below, $\|\cdot\|_1, \ \|\cdot\|_2$ denote the usual vector $p$-norms. 

First, we will show that 
\begin{align}
\label{eq:cone_norm}
\sum\limits_{k\geq 2}\|\lambda^{(k)}\|_2\leq b\|\lambda^{(0)}\|_2:=b\sqrt{\sum\limits_{j\in J_0}\lambda^2(u_j)}.
\end{align}
Indeed, for all $j\in J_{k}, \ k\geq 2$ we have 
$|\lambda(u_j)|\leq \frac 1 d \sum\limits_{i\in J_{k-1}}|\lambda(u_i)|$, implying that 
$\|\lambda^{(k)}\|_2\leq \frac{1}{\sqrt d}\|\lambda^{(k-1)}\|_1$. 
Summing up, we get
$$
\sum_{k\geq 2}\|\lambda^{(k)}\|_2\leq \frac{1}{\sqrt{d}}\sum_{j\notin J_0 }|\lambda(u_j)|\leq \frac{b}{\sqrt{d}}\sum_{j\in J_0}|\lambda(u_j)|,
$$
where the last inequality follows from the definition of the cone $C_{b,J}$. 
Inequality (\ref{eq:cone_norm})  follows since $\frac{b}{\sqrt{d}}\sum\limits_{j\in J_0}|\lambda(u_j)|\leq b\sqrt{\sum\limits_{j\in J_0}\lambda^2(u_j)}$.

Let $P_J$ be the $L_2(\Pi)$-orthogonal projection onto $L_J$, the linear span of $\l\{h(u_j), \ j\in J\r\}$. 
The following sequence of inequalities establishes the claim of Proposition \ref{prop:rip}:
\begin{align*}
&
\|\sum_{j=1}^N \lambda(u_j) h(u_j)\|_{L_2(\Pi)}\geq 
\|P_{J_0\cup J_1}\sum_{j=1}^N \lambda(u_j) h(u_j)\|_{L_2(\Pi)}\geq \\
&
\geq\|\sum_{j\in J_0\cup J_1}\lambda(u_j) h(u_j)\|_{L_2(\Pi)} - \sum_{k\geq 2}\|P_{J_0\cup J_1}\sum_{j\in J_k} \lambda(u_j) h(u_j)\|_{L_2(\Pi)}\geq \\
&
\geq \|\sum_{j\in J_0\cup J_1}\lambda(u_j) h(u_j)\|_{L_2(\Pi)}-\rho_d  \sum_{k\geq 2}\|\lambda^{(k)}\|_2 
\underbrace{\sup\limits_{v\in L_{J_0\cup J_1}, \|v\|_{L_2(\Pi)}=1} \|v\|_2}_{\leq 1/(1-\delta_{2d})}\geq \\
&
\geq
\|\sum_{j\in J_0\cup J_1}\lambda(u_j) h(u_j)\|_{L_2(\Pi)}-\frac{\rho_d}{1-\delta_{2d}} \sum_{k\geq 2}\|\lambda^{(k)}\|_2 \geq \\
&
\geq \|\sum_{j\in J_0\cup J_1}\lambda(u_j) h(u_j)\|_{L_2(\Pi)}-\frac{\delta_{3d} b}{1-\delta_{2d}}\sqrt{\sum\limits_{j\in J_0\cup J_1}\lambda^2(u_j)}\geq \\
&
\geq
\|\sum_{j\in J_0\cup J_1}\lambda(u_j) h(u_j)\|_{L_2(\Pi)}-\frac{\delta_{3d} b}{(1-\delta_{2d})^2}\l\|\sum_{j\in J_0\cup J_1}\lambda(u_j) h(u_j)\r\|_{L_2(\Pi)}= \\
&
=\l(1-\delta_{3d} \frac{b}{(1-\delta_{2d})^2}\r)\l\|\sum_{j\in J_0\cup J_1}\lambda(u_j) h(u_j)\r\|_{L_2(\Pi)}\geq \\
&
\geq \l(1-\delta_{3d} \frac{b}{(1-\delta_{2d})^2}\r)(1-\delta_{2d})\sqrt{\sum_{j\in J_0}\lambda^2(u_j)},
\end{align*}
hence $\beta_2^{(b)}(J)\leq \frac{1-\delta_{2d}}{(1-\delta_{2d})^2-b\delta_{3d}}$.

It remains to show that $\delta_{3d}<\frac{1}{2+b}$ is sufficient for $\beta_2^{(b)}<\infty$. 
Since $\delta_{2d}\leq \delta_{3d}$, it is enough to show that $\delta_{3d}<\frac{1}{2+b}$ implies $(1-\delta_{3d})^2-b\delta_{3d}>0$. 
The latter is satisfied whenever 
$\delta_{3d}< \frac{2+b}{2}\l(1-\sqrt{1-\frac{4}{(2+b)^2}}\r)$. 
Elementary inequality $\sqrt{1-x}\leq 1-\frac x 2, \ x\in [0,1]$ gives 
$\frac{2+b}{2}\l(1-\sqrt{1-\frac{4}{(2+b)^2}}\r)\geq \frac{1}{2+b}$, and the result follows.


\printindex

\end{document}